\documentclass[a4paper]{article}
\usepackage{amsmath}
\usepackage{amsfonts}
\usepackage{amssymb}
\usepackage{amsthm}
\usepackage{mathtools}
\usepackage{tikz}
\usepackage{tikz-cd}
\usepackage{enumerate}
\usepackage{layout}
\usepackage{mathabx}
\usepackage{bm}
\usetikzlibrary{patterns}

\DeclareMathOperator{\id}{id}
\DeclareMathOperator{\Aut}{Aut}

\DeclareMathOperator{\cd}{cd}
\DeclareMathOperator{\hd}{hd}
\DeclareMathOperator{\Hom}{Hom}
\DeclareMathOperator{\bHom}{{\bf Hom}}
\DeclareMathOperator{\Ext}{Ext}
\DeclareMathOperator{\Tor}{Tor}
\DeclareMathOperator{\bExt}{{\bf Ext}}
\DeclareMathOperator{\bTor}{{\bf Tor}}
\DeclareMathOperator{\res}{res}
\DeclareMathOperator{\cor}{cor}
\DeclareMathOperator{\ind}{ind}
\DeclareMathOperator{\bind}{{\bf ind}}
\DeclareMathOperator{\coind}{coind}
\DeclareMathOperator{\bcoind}{{\bf coind}}
\DeclareMathOperator{\im}{im}

\newcommand{\bdy}{\ensuremath{\partial}}

\renewcommand{\ss}[2]{\ensuremath{\mathbb{S}^{#1}\times\mathbb{S}^{#2}}}

\newcommand{\iso}{\ensuremath{\cong}}
\newcommand{\Z}[1][]{\ensuremath{\mathbb{Z}_{#1}}}
\newcommand{\Q}{\ensuremath{\mathbb{Q}}}

\newcommand{\N}{\ensuremath{\mathbb{N}}}

\newcommand{\F}{\ensuremath{\mathbb{F}}}

\newcommand{\nsgp}[1][]{\ensuremath{\triangleleft_{#1}}}

\newcommand{\gp}[1]{\ensuremath{\langle #1\rangle}}

\newcommand{\Gmod}[2][]{\ensuremath{{\mathfrak{#2}_{#1}(G)}}}
\newcommand{\famS}[0]{\ensuremath{\mathcal{S}}}

\newcommand{\Zpof}[1]{{\ensuremath{\Z[p][\![#1]\!]}}}
\newcommand{\Fpof}[1]{{\ensuremath{\F_p[\![#1]\!]}}}
\newcommand{\Zpiof}[1]{{\ensuremath{\Z[{\pi}][\![#1]\!]}}}
\newcommand{\ZG}[1]{{\ensuremath{\Z[\pi][\![G #1]\!]}}}
\newcommand{\ZGS}{\ensuremath{\Z[\pi][\![G/\famS]\!]}}
\newcommand{\Zhat}{\ensuremath{\widehat{\Z}}}
\newcommand{\fkB}{\ensuremath{\mathfrak{B}}}
\newcommand{\lqt}{\backslash}
\newcommand{\pFP}[1][]{\ensuremath{p}-\ensuremath{{\rm FP}_{#1}}}
\newcommand{\hotimes}[1][]{\ensuremath{\,\widehat{\otimes}_{#1}\,}}
\newcommand{\bhotimes}[1][]{\ensuremath{\,\widehat{\bm{\otimes}}_{#1}\,}}

\newtheorem{theorem}{Theorem}[section]

\newtheorem{prop}[theorem]{Proposition}

\newtheorem{lem}[theorem]{Lemma}
\newtheorem{clly}[theorem]{Corollary}

\theoremstyle{definition}
\newtheorem{defn}[theorem]{Definition}
\newtheorem*{cnv}{Conventions}

\theoremstyle{remark}
\newtheorem*{rmk}{Remark}

\theoremstyle{plain}

\newcounter{introthmcount}
\theoremstyle{definition}

\title{Relative cohomology theory for profinite groups}
\author{Gareth Wilkes}
\date{October 2, 2017}
\numberwithin{equation}{section}

\begin{document}
\maketitle

\begin{abstract}
In this paper we define and develop the theory of the cohomology of a profinite group relative to a collection of closed subgroups. Having made the relevant definitions we establish a robust theory of cup products and use this theory to define profinite Poincar\'e duality pairs. We use the theory of groups acting on profinite trees to give Mayer-Vietoris sequences, and apply this to give results concerning decompositions of 3-manifold groups. Finally we discuss the relationship between discrete duality pairs and profinite duality pairs, culminating in the result that profinite completion of the fundamental group of a compact aspherical 3-manifold is a profinite Poincar\'e duality group relative to the profinite completions of the fundamental groups of its boundary components.
\end{abstract}

\section*{Introduction}
The classical theory of the cohomology of groups has developed in several directions. Our principal source for the theory of relative group cohomology and Poincar\'e duality for group pairs will be the paper \cite{BE77} by Bieri and Eckmann. For the cohomology theory of profinite groups we have relied on Serre \cite{Serre13} and Symonds and Weigel \cite{SW00}. The principle aim of this paper is to establish a relative cohomology theory for profinite groups which exhibits the salient features of both these theories. Many of the outcomes closely parallel those of classical relative group cohomology, and naturally some of the arguments in this paper are broadly similar to those in \cite{BE77}---albeit with extra effort needed to cope with the topology. One major way in which the profinite theory differs from the classical theory is in the absence of a sufficiently broad cap product. Thus most results are stated and proved in terms of cup products and maps derived from them.

We will now describe the general plan of the paper. In the first section we will discuss the various preliminary concepts which will be required for our theory. Particular attention will be paid to the different categories involved. The homology theory of a profinite group naturally takes coefficients among compact modules, while the cohomology theory takes coefficients in discrete modules. The theory of duality groups requires, under some conditions, extension of the cohomology theory to compact modules and the homology theory of discrete modules. We will spend some time describing the way in which these categories fit together, as well as detailing the functors we wish to consider.

Section \ref{SecRelCohDef} we define the (co)homology of a profinite group pair and derive some basic properties, such as the appropriate version of an Eckmann-Shapiro lemma. We also define the cohomological dimension of a profinite group pair with respect to a choice of prime.

In Section \ref{SecCupProds} we will define and derive the properties of a cup product sufficiently broad for our needs. One might expect an article concerned with duality groups to rely heavily on cap products as \cite{BE77}. However as will be discussed at the time, the theory of cap products does not fit comfortably with the framework of profinite groups, and we will exclusively use cup products and mappings derived from them. 

Section \ref{SecTrees} develops various excision and Mayer-Vietoris-type results for profinite amalgamated products and HNN extensions, as well as other consequences for relative cohomology of an action on a profinite.

In Section \ref{SecFPPairs} we define the notion of a profinite duality pair at a prime $p$ and prove various foundational results concerning such pairs. We also use the Mayer-Vietoris sequences from the previous section to discuss graphs of Poincar\'e duality pairs. 

In Section \ref{SecGoodness} we explore the relationship between classical duality pairs and profinite duality pairs, in the context of `cohomological goodness' properties of discrete group pairs. We pay particular attention to the case of 3-manifold groups and conclude with an application of our theory, proving results concerning the Kneser-Milnor and JSJ decompositions of 3-manifolds. We intend to explore this direction further in a future paper.

The article concludes with two appendices proving technical results about profinite groups and modules. These results are not really necessary for an understanding of the rest of the paper, and the consequences deriving from them are plausible enough that the reader will believe them without subjecting themselves to the proofs. However the proofs should appear somewhere, so we leave them at the end of the paper.

\begin{cnv} The following conventions will be in force through the paper.
\begin{itemize}
\item Unless otherwise specified $G$ will be a profinite group and $\pi$ will be a non-empty set of primes.
\item For a set of primes $\pi$, the ring of $\pi$-adic integers (that is, the pro-$\pi$ completion of $\Z$) will be denoted $\Z[\pi]$. If $\pi=\{p\}$ then we will write $\Z[p]$ rather than $\Z[\{p\}]$. The field with $p$ elements will be denoted $\F_p$. 
\item Maps of topological groups or modules should be assumed to be continuous homomorphisms in the appropriate sense.
\item If $G$ is a profinite group, $H$ is a subgroup of $G$ and $M$ is a $G$-module of some kind, then $\res^G_H(M)$ denotes the $H$-module obtained by restriction.
\item If $G$ is a group and $M$ is a left $G$-module then $M^\perp$ will denote the canonical right $G$-module associated to $M$---that is, the  module with underlying abelian group $M$ and right $G$-action $m\cdot g=g^{-1}m$ for $g\in G$ and $m\in M$.
\item For $\pi$ a set of primes, $I_\pi$ will denote $\Q_\pi/\Z[\pi]$. A model for this is the group of rational numbers whose denominators are $\pi$-numbers.  
\item Conjugation in groups will be a right action, so that $x^y=y^{-1}xy$.
\end{itemize}
\end{cnv}

\section{Preliminaries}
\subsection{Categories of modules}\label{SecModuleCats}
There are several different categories in play in the theory to be developed. Let $G$ be a profinite group and let $\pi$ be a set of primes. Recall that an abelian group is called {\em $\pi$-primary} if it is a torsion group and the order of every element is divisible only by primes in $\pi$. Define the following categories. All modules are left modules unless otherwise specified.
\begin{defn} Let $G$ be a profinite group and let $\pi$ be a set of primes.
\begin{itemize}
\item \Gmod[\pi]{F} is the category of finite $\pi$-primary $G$-modules with $G$-linear morphisms
\item \Gmod[\pi]{C} is the category of compact pro-$\pi$ $G$-modules---that is, inverse limits of modules in \Gmod[\pi]{F}. The morphisms are continuous $G$-linear group homomorphisms. Similarly define $\Gmod[\pi]{C}^\perp$ to be the category of compact pro-$\pi$ right $G$-modules. These are abelian categories with enough projectives but not enough injectives.
\item \Gmod[\pi]{D} is the category of discrete $\pi$-primary $G$-modules---that is, direct limits of modules in \Gmod[\pi]{F}. The morphisms are the $G$-linear group homomorphisms.
This is an abelian category with enough injectives but not enough projectives.
\item \Gmod[\pi]{P} is the category of modules either in \Gmod[\pi]{C} or in \Gmod[\pi]{D}. The morphisms are those $G$-linear group homomorphisms which are continuous and \emph{strict}, in the sense that the quotient topology on the image of a map agrees with the subspace topology induced from the codomain. This is not an abelian category as it lacks a direct sum, but satisfies all the other axioms of an abelian category. Such a category is called an {\em exact category}. The injectives and projectives are precisely the injectives and projectives in \Gmod[\pi]{D} and \Gmod[\pi]{C} respectively.
\end{itemize}
\end{defn}
When $\pi$ is the set of all primes we omit $\pi$ from the notation. We will similarly omit $G$ when $G$ is the trivial group. We remark that our usage of the symbol \Gmod[]{D} disagrees with the usage in \cite{SW00}, where all discrete modules are allowed. Our restriction to torsion modules does therefore restrict the domain of definition of the cohomology theory. However allowing non-torsion modules does not seem to be particularly useful and does cause certain issues, so we will feel no qualms about ignoring them.

Since the category \Gmod[\pi]{P} is not quite an abelian category, and does not have enough projectives or injectives, one must consider a restricted notion of functor. Once this is done, the machinery of homological algebra functions in a satisfactory way. Effectively one reduces all proofs to deal first with the case when the argument lies in \Gmod[\pi]{D} (or \Gmod[\pi]{C}, depending on context) and take limits to deal with \Gmod[\pi]{C} (respectively \Gmod[\pi]{D}).
\begin{defn}
A covariant additive functor ${\bf T}\colon \Gmod[\pi]{P}\to \mathfrak{P}_\pi$ is called {\em continuous} if 
\begin{itemize}
\item ${\bf T}(M)\in\mathfrak{C}_\pi$ and ${\bf T}(A)\in\mathfrak{D}_\pi$ for $M\in\Gmod[\pi]{C}, N\in \Gmod[\pi]{D}$
\item the restriction of ${\bf T}$ to \Gmod[\pi]{D} commutes with direct limits
\item the restriction of ${\bf T}$ to \Gmod[\pi]{C} commutes with inverse limits
\end{itemize} 
\end{defn}
Dually, one has
\begin{defn}
A contravariant additive functor ${\bf T}\colon \Gmod[\pi]{P}\to \mathfrak{P}_\pi$ is called {\em co-continuous} if 
\begin{itemize}
\item ${\bf T}(M)\in\mathfrak{D}_\pi$ and ${\bf T}(A)\in\mathfrak{C}_\pi$ for $M\in\Gmod[\pi]{C}, N\in \Gmod[\pi]{D}$
\item the restriction of ${\bf T}$ to \Gmod[\pi]{D} commutes with direct limits in the sense that
\[{\bf T}|_{\Gmod[\pi]{D}}\circ \varinjlim = \varprojlim\circ\, {\bf T}|_{\Gmod[\pi]{D}}  \]
\item the restriction of ${\bf T}$ to \Gmod[\pi]{C} commutes with inverse limits in the sense that
\[{\bf T}|_{\Gmod[\pi]{C}}\circ \varprojlim = \varinjlim\circ\, {\bf T}|_{\Gmod[\pi]{C}}  \]
\end{itemize} 
\end{defn}
One defines notions of connected sequences of continuous functors, and continuous cohomological and homological functors in the natural way. A continuous cohomological functor ${\bf T}^\bullet$ is {\em co-effaceable} if ${\bf T}^i(J)=0$ for all $i>0$ and all injective modules $J\in \Gmod[\pi]{D}$. Similarly for effaceability.  
\begin{prop}[Proposition 3.6.1 of \cite{SW00}]\label{FptoPp1}
Let ${\bf S},{\bf T}\colon \Gmod[\pi]{P}\to \mathfrak{P}_\pi$ be continuous functors and let $f_0\colon {\bf S}|_{\Gmod[\pi]{F}}\to {\bf T}|_{\Gmod[\pi]{F}}$ be a natural transformation. Then there is a unique extension of $f_0$ to a natural transformation $f\colon {\bf S}\to {\bf T}$. If $f_0$ is a natural isomorphism, so is $f$.
\end{prop}
\begin{prop}[Corollary 3.6.3 of \cite{SW00}]\label{FptoPp2}
Let ${\bf T}^\bullet$ be a continuous co-effaceable cohomological \Gmod[\pi]{P}-functor, and ${\bf U}^\bullet$ a non-negative connected sequence of continuous functors. Then every natural transformation $f^0\colon {\bf T}^0\to {\bf U}^0$ lifts uniquely to a mapping of sequences $f^\bullet\colon {\bf T}^\bullet\to {\bf U}^\bullet$. If ${\bf U}^\bullet$ is also a co-effaceable cohomological functor and $f^0$ is a natural isomorphism, then all the $f^\bullet$ are natural isomorphisms. 
\end{prop}
Both of these propositions of course have dual forms for co-continuous functors and for homological functors.

There are various functors that one may wish to define on these categories. The properties of these functors are largely what one expects upon taking limits of functors on finite modules. We will give a list here of those most directly involved and quote those results we shall need. For a fuller account of the theory of modules over profinite groups see for example \cite{brumer66, RZ00, SW00}. 

In the sequel we shall use the convention that unadorned symbols will denote those functors which are always defined, and bold symbols for those requiring additional conditions.

\begin{itemize}
\item For $M\in\Gmod[\pi]{C}^\perp$ and $N\in\Gmod[\pi]{C}$ one may form the {\em completed tensor product} $M\hotimes[G] N\in\mathfrak{C}_\pi$. For the theory of completed tensor products see \cite{RZ00}, Section 5.5. This functor commutes with inverse limits in both variables.
\item If $M\in\Gmod[\pi]{C}^\perp$ is a finitely generated module (over \Zpiof{G}) and $N\in\Gmod[\pi]{F}$ the tensor product $M\hotimes[G] N$ is finite and agrees with the usual tensor product. Therefore defining $M\bhotimes[G] A$ for $A\in\Gmod[\pi]{D}$ by endowing the usual tensor product with the discrete topology yields a continuous functor $M\bhotimes[G]-$ from \Gmod[\pi]{P} to $\mathfrak{P}_\pi$. 
\item The discrete tensor product functors on \Gmod[\pi]{D} are defined using their classical definitions.
\item For $M\in\Gmod[\pi]{C}$ the complete tensor product with diagonal action gives a functor $M\hotimes[]-$ from \Gmod[\pi]{C} to itself. Similarly for \Gmod[\pi]{D}.
\end{itemize}
We endow the various Hom-groups with the compact-open topology. 
\begin{itemize}
\item For $M\in\Gmod[\pi]{C}, A\in\Gmod[\pi]{D}$ the group $\Hom(M,A)$ is a discrete $\pi$-primary torsion module and may be equipped with the diagonal $G$-action. So $\Hom(-,-)$ is a functor \[\Gmod[\pi]{C}\times\Gmod[\pi]{D}\to \Gmod[\pi]{D}\] which is contravariant-covariant. It commutes with inverse limits in the first variable and direct limits in the second.
\item If $M\in\Gmod[\pi]{C}$ is finitely generated (over \Zpiof{G}) and $N\in\Gmod[\pi]{C}$ then $\bHom_G(M,N)$ is a compact module in \Gmod[\pi]{C}, and so we have a functor $\bHom_G(M,-)$ from \Gmod[\pi]{C} to itself. This commutes with inverse limits in the second variable. Taken together with the previous point $\bHom_G(M,-)$ is now a continuous functor from \Gmod[\pi]{P} to $\mathfrak{P}_{\pi}$.
\item If $J\in\Gmod[\pi]{D}$ is {\em co-finitely-generated} in the sense that it is a direct limit of finite modules with a uniform bound on the size of a generating set then $\Hom_G(M, J)$ is finite for any finite module $M$. In this case $\bHom_G(A, J)$ lies in $\mathfrak{C}_\pi$ for $A\in\Gmod[\pi]{D}$ and $\bHom_G(-, J)$ is a co-continuous contravariant functor from \Gmod[\pi]{P} to $\mathfrak{P}_\pi$. This notion is the Pontrjagin dual (see below) of the previous item.
\item These last two points have analogous functors $\bHom(-,-)$ defined when one of the arguments is (co-)finitely generated as a module over the trivial group.     
\end{itemize}
\begin{rmk}
The most correct notation for the tensor product of a compact right $G$-module $M$ and a compact left $G$-module $N$ would be $M\hotimes[\Zpiof{G}]N$. This is rather cumbersome so, as above, we will abbreviate various notations:
\[  M\hotimes[\ZG{}]N = M\hotimes[G] N, \quad \Hom_{\ZG{}}(M,A)=\Hom_G(M,A)\]
and so on where $M,N\in\Gmod[\pi]{C}$ and $A\in\Gmod[\pi]{D}$. Furthermore when the group involved is the trivial group it will be omitted, so that for example $\Hom(A,B)$ without qualification would mean $\Hom_{\Z[\pi]}(A,B)$. These contractions will be extended in the obvious way to the derived Tor and Ext functors.
\end{rmk}

\begin{prop}[`Pontrjagin duality', see Proposition 2.2.1 of \cite{SW00} or Section 5.1 of \cite{RZ00}]
Let $I_\pi=\Q_\pi/\Z[\pi]\in\mathfrak{D}_\pi$. Then the operation ${}^\ast$ from \Gmod[\pi]{P} to itself given by 
\[M^\ast = \bHom(M, I_\pi) \]
with $G$-action given by $(g\cdot f)(m)=f(g^{-1}m)$ is a contravariant exact additive functor. The composition ${}^\ast\circ {}^\ast$ is naturally isomorphic to the identity functor. Furthermore there is an isomorphism of topological \Z[\pi]-modules 
\[\bHom(A,B) \iso \bHom(B^\ast, A^\ast) \]
for $A,B\in\Gmod[\pi]{P}$. 
\end{prop}
\begin{prop}[Corollary 2.2.2 of \cite{SW00}]\label{TorFreeImpliesFree}
If $M\in\mathfrak{C}_p$ is $p$-torsion-free then it is a free abelian pro-$p$ group and hence free and projective in $\mathfrak{C}_p$.
\end{prop}

With the above functors are associated derived functors, defined in the usual way. 
\begin{defn}
For $M\in\Gmod[\pi]{C}^\perp$ one has the left-derived functors of $M\hotimes[G] -$, viz.
\[\Tor^G_\bullet(M,-)\colon \Gmod[\pi]{C}\to \mathfrak{C}_\pi \]
For $M\in\Gmod[\pi]{C}$ one has the right-derived functors of $\Hom_G(M,-)$, viz.
\[\Ext_G^\bullet(M,-)\colon \Gmod[\pi]{D}\to\mathfrak{D}_\pi \]
\end{defn}
One result that we note for future reference is that Pontrjagin duality induces isomorphisms
\[\Ext_G^\bullet(M, N^\ast)\iso \Tor^G_\bullet(M^\perp,N)^\ast \]
for $M\in\Gmod[\pi]{C}$, where $M^\perp$ is the canonical right $G$-module associated to $M$. This follows from the identification
\[\Hom_G(M, N^\ast)\iso (M^\perp\hotimes[G] N)^\ast, \quad f\mapsto \left(m\otimes n\mapsto f(m)(n) \right) \]
(where $f\in \Hom_G(M, N^\ast)$, $m\in M$, $n\in N$) and the exactness of the functor $(-)^\ast$. The reader is left to check that this morphism is actually well defined with regard to the specified $G$-actions. 

Since the categorical properties of the module categories involved are very similar to the classical case, these functors have properties closely analogous to those for modules over discrete groups. See \cite{Serre13, SW00} and \cite{RZ00}, Chapter 6. An important property of these functors, which we shall use constantly and without explicit citation, is the following. 
\begin{prop}[Corollaries 6.1.8 and 6.1.10 of \cite{RZ00}]
$\Tor^G_\bullet(M,-)$ commutes with inverse limits and $\Ext_G^\bullet(M,-)$ commutes with direct limits.
\end{prop}

\begin{defn}
A module $M\in \Gmod[\pi]{C}$ is {\em of type FP$_n$} if it has a projective resolution $P_\bullet\to M$ with $P_k$ finitely generated for $k\leq n$, where $n$ may be an integer or infinity. 

Given a module $M$ of type ${\rm FP}_\infty$ and a projective resolution as above, we have continuous \Gmod[\pi]{P}-functors $\bHom_G(P_n,-)$ for each $n$ and can therefore define continuous right derived functors of the functor $\bHom_G(M,-)$ to be 
\[\bExt_G^\bullet(M,N)= H_n(\Hom(P_\bullet,N))  \]
for $N\in \Gmod[\pi]{P}$. Similarly one may define $\bTor^G_\bullet(M,-)$ for a right module of type ${\rm FP}_\infty$. 
\end{defn}
These functors also satisfy Pontrjagin duality. See Sections 3.7 and 4.2 of \cite{SW00} for more details. 

Finally one defines homology and cohomology groups of $G$ in the usual way.
\begin{defn}
For $N\in\Gmod[\pi]{C}$ and $A\in\Gmod[\pi]{D}$ define
\[H_\bullet(G,N)=\Tor^G_\bullet(\Z[\pi], C), \quad H^n(G,A)= \Ext^\bullet_G(\Z[\pi], A)  \]
If $G$ is of type $\pi$-FP$_\infty$---that is, the $G$-module $\Z[\pi]$ is of type FP$_\infty$---then we may define 
\[{\bf H}_\bullet(G,N)=\bTor^G_\bullet(\Z[\pi], N), \quad {\bf H}^n(G,A)= \bExt^\bullet_G(\Z[\pi], N)  \]
for $N\in \Gmod[\pi]{P}$. 
\end{defn}

\subsection{Different sets of primes}
The above definitions have been made with respect to a choice of some set of primes $\pi$. One could ask whether these definitions are invariant under enlarging the set of primes. For instance for $A\in\Gmod[\pi]{D}$ one could ask whether the cohomology groups $H^\bullet(G, A)$ depend on whether $A$ is considered as an object of \Gmod[\pi]{D} or of \Gmod[]{D}. As one may imagine from the fact that we have not troubled to include $\pi$ in the notation, all the relevant notions are canonically isomorphic for different choices of $\pi$ in a way we shall now describe.

All finite $G$-modules, being {\em a fortiori} finite abelian groups, have a unique $p$-Sylow subgroup for each $p$. So the module $M$ splits as a direct sum over $p\in\pi$ of modules in $\Gmod[p]{F}$ in a canonical way. The direct sum of those components for $p\in\pi$ is called the {\em $\pi$-primary component} $M(\pi)$. By taking limits all modules in \Gmod[]{C} have a similar canonical decomposition as a direct product of $p$-primary components, and all modules in \Gmod[]{D} have a direct sum decomposition. For instance we have
\[I_\pi=\Q_\pi/\Z[\pi]=\bigoplus_{p \in\pi} \Q_p/\Z[p] \]

Not only $G$-modules but also the maps between them split into $p$-primary components. This is an easy consequence of the fact that no non-trivial maps exist between a finite $p$-group and a finite $q$-group for distinct primes $p$ and $q$. 

One consequence of this that will be buried in the notation is that Pontrjagin duality is independent of $\pi$---that is, for $M\in\Gmod[\pi]{P}$ we have
\[M^\ast = \Hom_{\Z[\pi]}(M, \Q_\pi/\Z[\pi])= \Hom_{\Zhat}(M, \Q/\Z)\]

If $\pi'\subseteq \pi$ are non-empty sets of primes then all of this ultimately boils down to a statement that `taking $\pi'$-primary components is an exact functor from \Gmod[\pi]{P} to \Gmod[\pi']{P} taking projectives to projectives and injectives to injectives'. One may easily check that this exact functor behaves well with respect to tensor products and Hom-functors. Therefore the different possible definitions of, for instance, $H^\bullet(G; A)$ for $A\in\Gmod[\pi]{D}$ are canonically isomorphic. In the same way take the $\pi$-primary component any of the exact sequences we will describe will yield an analogous exact sequence. Hence every definition or theorem statement we make will be true independent of which theory we will choose to develop.

We also remark that the notion of `finitely generated projective' is invariant under change of $\pi$ in a certain sense. We state these as formal propositions as, unlike most of the remarks in this section, we will refer to it later.
\begin{prop}\label{DiffPrimes}
Let $\pi'\subseteq \pi$ be non-empty sets of primes.
\begin{itemize}
\item Let $P$ be a projective in $\Gmod[\pi]{C}$. Let $Q=P(\pi')$ be the $\pi'$-primary component of $P$. Then $Q$ is projective, both in $\Gmod[\pi']{C}$ and $\Gmod[\pi]{C}$.
\item If $M\in\Gmod[\pi]{C}$ is finitely generated (over \Zpiof{G}) then $M(\pi')$ is finitely generated (both over \Zpiof{G} and over $\Z[\pi'][\![G]\!]$).  
\end{itemize} 
\end{prop}
\begin{proof}
By an inverse limit starting from finite modules one finds that for any $M\in \Gmod[\pi]{C}$ we have
\[ M = M(\pi')\oplus M(\pi\smallsetminus\pi') \]
From this both the first statement and finite generation over \Zpiof{G}\ rapidly follow. For finite generation over $\Z[\pi'][\![G]\!]$ note that if $\phi\colon \Zpiof{G}^{\oplus n}\to M(\pi')$ is a surjection then $\phi$ vanishes on the $(\pi\smallsetminus\pi')$-primary component of $\Zpiof{G}^{\oplus n}$, so the induced map on $\pi'$-primary components $\phi_{\pi'}\colon \Z[\pi'][\![G]\!]^{\oplus n}\to M(\pi')$ is also a surjection.
\end{proof}
\begin{prop}\label{pPrimProjectives}
Let $M\in\Gmod[\pi]{C}$. Suppose that for all $p\in \pi$ the $p$-primary component $M(p)$ is projective in \Gmod[p]{C}. Then $M$ is projective in \Gmod[\pi]{C}.
\end{prop}
\begin{proof}
Let $\phi\colon S\twoheadrightarrow T$ be a surjection of modules in \Gmod[\pi]{C} and let $f\colon M\to T$. Then each of $\phi$ and $f$ split as the direct product of their $p$-primary components, and each $\phi(p)\colon S(p)\to T(p)$ is a surjection. Then by assumption each $f(p)$ lifts to a map $\tilde f(p)\colon M(p)\to S(p)$. The direct product of all these lifts is a lift of $f$ to a map $M\to S$ as required.
\end{proof}
\subsection{Useful Identities}\label{SecUsefulIds}
We be frequently manipulating modules and subgroups, and we list here several identities of use and give indications of the proofs. We will use these often, and sometimes without explicit reference. The reader is warned that we will not always describe the $G$-actions on modules which appear later in the paper if the module is described in this section, as to do so would rather clutter the paper.

In cases where modules are equipped with extra finiteness properties yielding continuous (`bold') \Gmod[\pi]{P}-functors as in Section \ref{SecModuleCats}, one may take limits of the identities in this section to establish similar identities for the continuous \Gmod[\pi]{P}-functors. In all cases, the identites in this section can largely be derived from identities for finite modules via limiting process.

Throughout this section let $G$ be a profinite group and $S$ a closed subgroup of $G$.
\begin{defn}
Let $M\in \mathfrak{C}_\pi(S)$ and $A\in \mathfrak{D}_\pi(S)$. The {\em induced} and {\em coinduced} $G$-modules are the modules
\[\ind_G^S(M)=\ZG{}\hotimes[S]M\in\Gmod[\pi]{C}, \quad \coind_G^S(A)=\Hom_S(\ZG{}, A)\in\Gmod[\pi]{D} \]
where the $G$-actions are given by
\[g\cdot (x\otimes m)= (gx)\otimes m, \quad (g\cdot f)(x)=f(xg) \]
for $g,x\in G$, $m\in M$, $f\in\Hom_S(\ZG{}, A)$.
\end{defn}
 In the case where $U$ is open in $G$ so that $\Zpiof{G}$ is finitely generated as a $U$-module, one may also define
\[\bind_G^U(A)=\ZG{}\bhotimes[U]A\in\Gmod[\pi]{D}, \quad \bcoind_G^U(M)=\bHom_U(\ZG{}, M)\in\Gmod[\pi]{C} \]
for $M\in \mathfrak{C}_\pi(U)$ and $A\in \mathfrak{D}_\pi(U)$. In this case $\bind$ and $\bcoind$ are continuous functors $\mathfrak{P}_\pi(U)\to\Gmod[\pi]{P}$. 

There is an alternative formulation of these modules, in the case when $M$ and $A$ are in fact $G$-modules, with the $S$-module structure obtained by restriction.
\begin{prop}\label{IndandCoinsasDiags}
Let $M\in \mathfrak{C}_\pi(G)$ and $A\in \mathfrak{D}_\pi(G)$. Then there are isomorphisms of $G$-modules
\[\ind_G^S(M)\iso \ZG{/S}\hotimes[] M, \quad \coind_G^S(A) \iso \Hom(\ZG{/S}, A) \]
where the modules on the right hand side have the {\em diagonal} $G$-actions, given by
\[g\cdot (x\otimes m)= (gx)\otimes (gm), \quad (g\cdot f)(x)=gf(g^{-1}x) \]
for $g,x\in G$, $m\in M$, $f\in\Hom_S(\ZG{}, A)$.
\end{prop}
\begin{proof}
The isomorphisms are given by the maps
\[x\otimes m\to (xS)\otimes(xm), \quad f\mapsto \left(xS\mapsto xf(x^{-1}) \right)\]
where $x\in G$, $m\in M$, $f\in\Hom_S(\ZG{}, A)$. The reader is left to verify that these are well-defined isomorphisms.
\end{proof}

\begin{prop}\label{IndEqualsCoind}
Let $G$ be a profinite group and let $U$ be an open subgroup of $G$. Then for any $M\in \mathfrak{C}_\pi(U)$ there is a natural isomorphism of $G$-modules
\[\ind^U_G(M)= \bcoind^U_G(M)=\bHom_U(\Zpiof{G}, M) \]
\end{prop}
\begin{proof}
Given a choice of section $\sigma\colon G/U\to G$ we may define a natural isomorphism of $G$-modules
\[\Hom_U(\Zpiof{G}, M) \longrightarrow \Zpiof{G}\hotimes[U] M, \quad f\mapsto \sum_{x\in G/U} \sigma(x)\otimes f(\sigma(x)^{-1}) \]
where $M\in\mathfrak{F}_\pi(U)$. This is independent of the choice of $\sigma$. Now taking an inverse limit (noting that $\Zpiof{G}$ is finitely generated over \Zpiof{U}) gives the result for all $M\in \mathfrak{C}_\pi(U)$.
\end{proof}

\begin{prop}\label{TensorHom}
Let $M,N\in\Gmod[\pi]{C}$ and $A\in \Gmod[\pi]{D}$. Then there is an isomorphism of $G$-modules
\[\Hom(M\hotimes[] N, A)\iso \Hom(M,\Hom(N,A)) \]
where all $\Hom$ and tensor modules inherit the diagonal $G$-action from their constituent modules. In particular, taking the submodules of $G$-invariants, there are isomorphisms of $\Zhat$-modules
\[\Hom_G(M\hotimes[] N, A)\iso \Hom_G(M,\Hom(N,A)) \]
\end{prop}
In particular, using the formulations of induced and coinduced modules from Proposition \ref{IndandCoinsasDiags} we have the following.
\begin{clly}\label{HomIndAndCoind}
Let $M\in\Gmod[\pi]{C}$ and $A\in \Gmod[\pi]{D}$. There are natural isomorphisms
\[\Hom_G(\ind_G^S(M), A)\iso \Hom_G(M, \coind_G^S(A)) \]
\end{clly}
\begin{prop}\label{IndAndCoindAsGMods}
Let $M\in \mathfrak{C}_\pi(S)$ and $C\in\Gmod[\pi]{D}$. There is a natural identification of $G$-modules
\[ \coind^S_G(\Hom(M,C))\iso \Hom(\ind^S_G(M), C) \]
where $\Hom$-modules have diagonal actions.
\end{prop}
\begin{proof}
The isomorphism is given by the map
\begin{equation*}
\Phi\colon \Hom_S(\Zpiof{G},\Hom(M, C))\to  \Hom(\Zpiof{G}\hotimes_S M, C),\quad
\Phi(f)(x\otimes m) = xf(x^{-1})(m)
\end{equation*}
for $f\in \Hom_S(\Zpiof{G},\Hom(M, C)), x\in G$ and $m\in M$. One may readily verify that this is well-defined and $G$-linear.
\end{proof}
\begin{prop}\label{preShapiro}
Let $P\in\Gmod[\pi]{C}$ and $A\in\mathfrak{D}_\pi(S)$. Then there is a natural isomorphism of $\Zhat$-modules
\[\Hom_S(P, A) \iso \Hom_G(P, \coind_G^S(A)) \]
\end{prop}
\begin{proof}
The isomorphism is given by the map
\[h\mapsto \left(p\mapsto (x\mapsto h(xp))\right) \]
where $h\in \Hom_S(P, A)$, $p\in P$ and $x\in G$.
\end{proof}
Finally there are several statements about derived functors arising from the above identites, which together are known by the collective name of `the Shapiro lemma' or `Shapiro-Eckmann identities'.
\begin{prop}\label{Shapiro1}[`(Absolute) Shapiro Lemma']
Let $M\in\mathfrak{C}_\pi(S)$ and $A\in\mathfrak{D}_\pi(S)$. There are natural isomorphisms
\[H_\bullet(G,\ind_G^S(M))\iso H_\bullet(S, M), \quad H^\bullet(G,\coind_G^S(A))\iso H^\bullet(S, A) \]
\end{prop}
\begin{proof}
Take a projective resolution $P_\bullet$ of $\Z[\pi]$ by right $G$-modules. By Proposition \ref{ResProjectives} this is also a projective resolution in $\mathfrak{C}(S)$ by restriction. The identity
\[P_\bullet\hotimes[G]\ZG{}\hotimes[S]M = P_\bullet\hotimes[S] M \]
gives the first part of the proposition upon passing to homology.

For the cohomology identity take a projective resolution $P_\bullet$ of $\Z[\pi]$ by left $G$-modules. Applying Proposition \ref{preShapiro} and passing to cohomology gives the result.
\end{proof}
\begin{rmk}
In the case $U$ is open in $G$ and when $\Z[\pi]$ has type FP$_\infty$ as a $G$-module---and hence as a $U$-module by Proposition \ref{FPnModsAndSubgps} below---then all the homology, cohomology functors and the induction and coinduction functors are continuous functors and one may take limits to find natural isomorphisms
\[{\bf H}_\bullet(G,\bind_G^U(M))\iso {\bf H}_\bullet(U, M), \quad {\bf H}^\bullet(G,\bcoind_G^U(M))\iso {\bf H}^\bullet(U, M) \]
for all $M\in \mathfrak{P}_\pi(U)$.
\end{rmk}
Let $S$ be a normal subgroup of $G$ and let $M\in\Gmod[\pi]{C}, A\in\Gmod[\pi]{D}$. Then the cohomology groups of $S$ acquire $G$-actions as follows. Let $P_\bullet$ be a projective resolution of $\Z[\pi]$ in $\Gmod[\pi]{C}$. For $g\in G$, $p\in P_r$ and $f\in\Hom_S(P_r, A)$ define 
\[ (g\star f)(p) = gf(g^{-1}p)  \]
One may verify that this is well defined and so gives a left action of $G$ on $H^r(S, \res^G_S(A))$. Furthermore if we identify $\coind^S_G(A)$ with $\Hom(\Zpiof{G/S}, A)$ by Proposition \ref{IndandCoinsasDiags} then for $g\in G$ and $f\in \Hom_G(P_r, \coind^S_G(A))$ then we may define
\[(g\star f)(p)(xS) = gf(p)(xgS) \]
for $p\in P_r$ and $x\in G$. This gives a $G$-action on $H^r(G,\coind^S_G(A))$.
We may define similar constructs in homology. Let $P_\bullet$ be a projective resolution of $\Z[\pi]$ in $\Gmod[\pi]{C}^\perp$. Define a right $G$-action on $P_r\hotimes[S] \res^G_S(M)$ by 
\[(p\otimes m )\star g = (pg)\otimes (g^{-1}m)\]
for $p\in P_r, m\in M$ and $g\in G$. Identifying $\ind^S_G(M)$ with $\Zpiof{G/S}\hotimes M$ there is a right $G$-action on $P_\bullet\hotimes[G] \ind^S_G(M)$ given by 
\[(p\otimes xS\otimes m)\star g = p\otimes xgS \otimes m  \]
for $p\in P$, $g,x\in G$ and $m\in M$. These constructions give right actions on $H_r(S, \res^G_S(M))$ and $H_r(G, \ind^S_G(M)$.
\begin{prop}\label{ShapiroWithAction}
Suppose that $S$ is a normal subgroup of $G$ and let $M\in\Gmod[\pi]{C}, A\in\Gmod[\pi]{D}$. Then the Shapiro isomorphisms
\[H_\bullet(G,\ind_G^S(M))\iso H_\bullet(S, \res^G_S(M)), \quad H^\bullet(G,\coind_G^S(A))\iso H^\bullet(S, \res^G_S(A))\]
are isomorphisms of $G$-modules where all modules are endowed with the natural $G$-actions as defined above.
\end{prop}
The proof is simply a matter of tracing the various isomorphisms around and verifying that the $G$-actions match up. These constructions are unsurprisingly Pontrjagin dual to one another. In particular we note the special case that
\[H_0(1, \res^G_1(M)) \iso  M^\perp, \quad H^0(1, \res^G_1(A)) \iso A \]
are isomorphisms of $G$-modules.
\begin{prop}\label{Shapiro2}
Let $M\in\mathfrak{C}_\pi(G)$ and $A\in\mathfrak{D}_\pi(G)$. There are natural isomorphisms
\[\Tor_\bullet^G(\Zpiof{S\lqt G}, M)\iso \Tor_\bullet^S(\Z[\pi], M), \quad \Ext^\bullet_G(\ZG{/S}, A)\iso \Ext^\bullet_S(\Z[\pi], A) \]
\end{prop}
\begin{proof}
Take an projective resolution $P_\bullet$ of $M$ by left $G$-modules. The identity
\[\Zpiof{S\lqt G}\hotimes[G]P_\bullet=\Z[\pi]\hotimes[S]\ZG{}\hotimes[G]P_\bullet=\Z[\pi]\hotimes[S]P_\bullet \]
gives the required isomorphism on passing to homology. 

For the second identity take an injective resolution $J^\bullet$ of $A$ in $\Gmod[\pi]{D}$. By Proposition \ref{ResProjectives} and Pontrjagin duality $\res^G_S(J^\bullet)$ is still an injective resolution of $A$ in $\mathfrak{D}_{\pi}(S)$ by restriction. Using Propositions \ref{IndandCoinsasDiags} and \ref{preShapiro} and Corollary \ref{HomIndAndCoind} we have
\[\Hom_G(\Zpiof{G/S}, J^\bullet) =  \Hom_G(\ind_G^S(\Z[\pi]), J^\bullet) \iso \Hom_G(\Z[\pi], \coind^S_G(J^\bullet))\iso\Hom_S(\Z[\pi], J^\bullet)  \] 
Now pass to cohomology to get the result.  
\end{proof}
\subsection{Modules of type FP$_\infty$}
We mention some results about modules of type FP$_\infty$ which we shall need. The first is a consequence of the generalised Schanuel's Lemma. See Section VIII.4 of \cite{BrownCohom}.
\begin{prop}\label{Schanuel}
Let $M$ be a finitely generated module in \Gmod[\pi]{C}. Then the following are equivalent:
\begin{itemize}
\item $M$ has type FP$_n$
\item there exists a partial projective resolution $F_n\to\cdots \to F_0\to M$ with each $F_i$ free of finite rank
\item for every partial projective resolution $P_k\to\cdots \to P_0\to M$ for $k<n$ with each $P_i$ finitely generated, $\ker(P_k\to P_{k-1})$ is finitely generated.
\end{itemize}
\end{prop} 
\begin{prop}[Proposition 3.3.1 of \cite{SW00}]\label{ResProjectives}
Let $P\in\Gmod[\pi]{C}$ be projective and $S$ a closed subgroup of $G$. Then the restriction $\res^G_S(P)$ is projective in $\mathfrak{C}_{\pi}(S)$.
\end{prop}
\begin{prop}\label{FPnModsAndSubgps}
Let $G$ be a profinite group and $U$ an open subgroup of $G$. Let $M\in\Gmod[\pi]{C}$. Then $M$ has type FP$_n$ if and only if $\res^G_U(M)$ is of type FP$_n$ as a $U$-module.
\end{prop}
\begin{proof}
Since $\res^G_U$ takes finitely generated modules to finitely generated modules and projective modules to projective modules, one direction is clear. For the other, assume $M'=\res^G_U(M)$ is of type FP$_n$ as a $U$-module. First note that since $M'$ is finitely generated as a $U$-module, $M$ is a finitely generated $G$-module. Suppose we have a partial resolution $P_k\to\cdots\to P_0\to M$ ($0\leq k<n$) with each $P_k$ a finitely generated projective $G$-module. We must show that the kernel $\ker(P_k\to P_{k-1})$ is finitely generated over $G$. But since $\res^G_U(P_\bullet)$ is a partial resolution of $M'$ by finitely generated $U$-modules, the kernel is finitely generated over $U$ and we are done.
\end{proof}
\begin{prop}\label{IndAndFPProjs}
Let $G$ be a profinite group and $S$ a closed subgroup of $G$. Then $\ind^S_G(-)$ is an exact functor taking finitely generated projectives to finitely generated projectives. 
\end{prop}
\begin{proof}
Since $\Zpiof{G}$ is free as an $S$-module, it is projective and hence $\Zpiof{G}\hotimes[S]-$ is an exact functor. A finitely generated projective is by definition a summand of some finitely generated free module $\Zpiof{S}^{\oplus m}$. Taking induced modules gives a direct summand of the finitely generated free module $\Zpiof{G}^{\oplus m}$, hence a finitely generated projective module in \Gmod[\pi]{C}.
\end{proof}
\begin{prop}\label{IndAndFPMods1}
Let $G$ be a profinite group and $S$ a closed subgroup of $G$. Let $M\in \mathfrak{C}_\pi(S)$. If $M$ is finitely generated as an $S$-module then $\ind^S_G(M)$ is finitely generated as a $G$-module. 
\end{prop}
\begin{proof}
If $F\to M$ is an epimorphism from a finitely generated free $S$-module then applying $\ind^S_G$ gives an epimorphism from a finitely generated free $G$-module to $\ind^S_G(M)$.
\end{proof}
In the world of discrete groups the converse of this proposition is true, and indeed not difficult. Any element of the induced module may be written as a finite sum of basic tensors, hence finite generating set for the induced module gives a finite generating set consisting only of basic tensors. From this it is not difficult to derive a finite generating set for $M$. 

It is much less obvious whether the converse to Proposition \ref{IndAndFPMods1} is true---indeed it appears to be unknown.  Therefore we will coin the following property, while hoping that it turns out to be vacuous.
\begin{defn}
Let $G$ be a profinite group and let $\pi$ be a set of primes. Then the pair $(G, S)$ {\em has property FIM (with respect to $\pi$)} if the following property holds:
\begin{quote}
For every $M\in \mathfrak{C}_\pi(S)$ the induced module $\ind^S_G(M)$ is finitely generated as a $G$-module if and only if $M$ is finitely generated as an $S$-module.
\end{quote}
We say $G$ {\em has property FIM (with respect to $\pi$)} if $(G,S)$ has property FIM with respect to $\pi$ for all closed subgroups $S$ of $G$. 
\end{defn}
We will usually omit the set of primes when clear from the context. Here `FIM' stands for `finiteness (properties) of induced modules'. There is however no obstruction for open subgroups. 
\begin{prop}\label{IndAndFPMods1b}
Let $G$ be a profinite group and $U$ a open subgroup of $G$. Then $(G,U)$ has property FIM with respect to any set of primes. 
\end{prop}
\begin{proof}
Since $U$ is open, if $\ind^U_G(M)$ is finitely generated as a $G$-module then it is also finitely generated as a $U$-module. Since $M$ is a direct summand of $\res^G_U(\ind^U_G(M))$ (see Section 6.11 of \cite{RZ00}) there is a $U$-linear retraction to $U$ and we are done.
\end{proof}
When one only considers one prime $p$ at a time and considers a (virtually) pro-$p$ group one can dispense with some of the difficulty of the finiteness conditions for modules.
\begin{prop}[Proposition 4.2.3 of \cite{SW00}]\label{FPnForVirtProP}
Let $G$ be a virtually pro-$p$ group and $M\in\Gmod[p]{C}$. Then $M$ is of type FP$_n$ if and only if $\Ext^k_G(M, F)$ is finite for every (simple) $F\in\Gmod[p]{F}$ and all $k\leq n$.
\end{prop} 
\begin{prop}
Let $G$ be a pro-$p$ group. Then $G$ has property FIM with respect to $\pi=\{p\}$.
\end{prop}
\begin{proof}
This follows from Proposition \ref{FPnForVirtProP} and the identity
\[\Hom_G(\ind^S_G(M), \F_p) = \Hom_S(M,F_p)\]
noting that the only simple $p$-primary $G$ (or $S$) module is $\F_p$ (Lemma 7.1.5 of \cite{RZ00}). 
\end{proof}
\begin{prop}\label{IndAndFPMods2}
Let $G$ be a profinite group and $S$ a closed subgroup of $G$. Let $M\in\mathfrak{C}_\pi(S)$. If $M$ has type FP$_\infty$ as an $S$-module then $\ind^S_G(M)$ has type FP$_\infty$ as a $G$-module. If $(G,S)$ has property FIM with respect to $\pi$ then the converse also holds.
\end{prop}
\begin{proof}
The forwards direction follows immediately from Proposition \ref{IndAndFPProjs}. For the other direction proceed by induction. Suppose $\ind^S_G(M)$ has type FP$_\infty$ as a $G$-module. Suppose we have proved that $M$ has type FP$_{n}$ and take a projective resolution $P_n\to\cdots \to P_0\to M$ with each $P_i$ finitely generated. We must prove that $K=\ker(P_n\to P_{n-1})$ is finitely generated. Applying $\ind^S_G$ gives a partial resolution of $\ind^S_G(M)$ finitely generated in each dimension. Then by Proposition \ref{Schanuel} the kernel $\ker(\ind^S_G(P_n)\to \ind^S_G(P_{n-1}))$ is a finitely generated $G$-module. By exactness of $\ind^S_G(-)$ this kernel is precisely $\ind^S_G(K)$ so by property FIM we are done.
\end{proof}

\begin{prop}\label{DirSumsAndFPn}
Let $M$ be a finite direct sum of $G$-modules $M_1,\ldots, M_r\in\Gmod[\pi]{C}$. Then $M$ has type FP$_\infty$ if and only if each $M_i$ does. 
\end{prop}
\begin{proof}
If all $M_i$ are of type FP$_\infty$ then taking a direct sum of projective resolutions of the $M_i$, finitely generated in each degree, gives the required resolution of $M$. 

So suppose $M$ has type FP$_\infty$. Without loss of generality we will prove that $M_1$ has type FP$_\infty$; and by combining the other summands we may assume that $r=2$. We will prove by induction that both $M_1$ and $M_2$ have type FP$_\infty$. Suppose we know that they both have type FP$_n$ and let $P_n\to\cdots \to P_0\to M_1$ and $Q_n\to\cdots \to Q_0\to M_2$ be partial resolutions with all $P_k$ and $Q_k$ finitely generated projectives. The direct sum $(P_k\oplus Q_k)_{0\leq k\leq n}$ is a partial resolution of $M$ finitely generated in each dimension. Hence by Proposition \ref{Schanuel} the kernel $P_n\oplus Q_n\to P_{n-1}\oplus Q_{n-1}$ is finitely generated. But this kernel is simply the sum of the kernels $P_n\to P_{n-1}$ and $Q_n\to Q_{n-1}$ which are therefore finitely generated and we are done.
\end{proof}

\begin{prop}\label{SESandFPn}
Let $0\to M'\to M\to M''\to 0$ be a short exact sequence in \Gmod[\pi]{C}. If $M'$ has type FP$_\infty$ then $M$ has type FP$_\infty$ if and only if $M''$ does.
\end{prop}
\begin{proof}
Suppose that $M'$ and $M''$ have type FP$_\infty$ and take projective resolutions $P'_\bullet$ and $P''_\bullet$ of $M'$ and $M''$ which are finitely generated in each dimension. By the Horseshoe Lemma (\cite{Weibel95}, Lemma 2.2.8) there is a projective resolution of $M$ whose objects are $P'_n\oplus P''_n$, hence is finitely generated in each dimension. So $M$ has type FP$_\infty$ as required.

Suppose that $M'$ and $M$ have type FP$_\infty$. Take projective resolutions $P_\bullet$ and $Q_\bullet$ of $M'$ and $M$ which are finitely generated in each dimension and lift the map $M'\to M$ to a chain map $f_\bullet\colon P_\bullet\to Q_\bullet$. The mapping cone $C_\bullet$ of $f_\bullet$ is a chain complex with modules $C_n=P_{n-1}\oplus Q_n$ and is therefore projective and finitely generated in each dimension. By the long exact sequence for the mapping cone $C_\bullet$ is exact at each degree $n\geq 1$ and has homology $M''$ in degree zero. So the mapping cone is the required projective resolution. Refer to \cite{Weibel95}, Section 1.5 for the mapping cone construction. 
\end{proof}

\subsection{Double complexes}\label{SecDblComplexes}
We will now set up our conventions for tensor products of chain complexes. Take two chain complexes $P_\bullet$ and $Q_\bullet$, with $P_r=0$ for all $r<r_0$ and $Q_s=0$ for all $s<s_0$ for some integers $r_0, s_0$, in an abelian monoidal category with monoidal product $\boxtimes$, for instance \Gmod[\pi]{C} with product $\widehat{\otimes}_{\Z[\pi]}$ (and with diagonal actions on the product). Their tensor product $P_\bullet\boxtimes Q_\bullet$ will be the double complex $R_{\bullet \bullet}$ with objects $R_{rs}=P_r\boxtimes Q_s$ and with differentials
\[d^{\rm hor}_{rs} = d^P_r\boxtimes\id\colon P_r\boxtimes Q_s\to P_{r-1}\boxtimes Q_s \]
\[d^{\rm ver}_{rs} = (-1)^r\id\boxtimes\, d^Q_s\colon P_r\boxtimes Q_s\to P_r\boxtimes Q_{s-1} \] 
The total complex ${\rm tot}(R_{\bullet\bullet})$ is of course the complex $C_\bullet$ with 
\[C_n = \bigoplus_{r+s=n}R_{rs}, \quad d^C_n= \bigoplus_{r+s=n} (d^{\rm hor}_{rs}+d^{\rm ver}_{rs}) \] 
The theory of double complexes is fairly standard and the definition was included here only to note the sign convention. We will however recall the following standard fact. Here $\hotimes$ denotes $\hotimes[{\Z[\pi]}]$ in accordance with our standard conventions.
\begin{prop}\label{DblResolutions}
Let $G$ be a profinite group and let $\pi$ be a set of primes. Let $P_\bullet\to \Z[\pi]$ and $Q_\bullet\to\Z[\pi]$ be projective resolutions of $\Z[\pi]$ in \Gmod[\pi]{C}. Then ${\rm tot}(P_\bullet\hotimes Q_\bullet)$ with augmentation $P_0\hotimes Q_0\to \Z[\pi]\hotimes \Z[\pi] = \Z[\pi]$ is a projective resolution of $\Z[\pi]$. 
\end{prop}
\begin{proof}
This follows immediately from (the proof of) Theorem 2.7.2 of \cite{Weibel95}, noting that $-\hotimes \Z[\pi]$ is exact (indeed it is the identity functor).
\end{proof}
\subsection{Continuously indexed families of subgroups}\label{SecCtslyIndFams}
\begin{defn}\label{DefCtsIndex}
Let $G$ be a profinite group, and $\famS=\{S_x\}_{x\in X}$ a family of subgroups of $G$ indexed by a profinite space $X$. We say that \famS{} is {\em continuously indexed by $X$} if whenever $U$ is an open subset of $G$ the set \[\left\{x\in X \mid S_x\subseteq U\right\}\] is open in $X$. An equivalent definition is that the set
\[\left\{(g,x)\in G\times X\mid g\in S_x \right\} \]
is a closed subset of $G\times X$. 
\end{defn}
\begin{prop}
Let $G$ be a profinite group and $\famS=\{S_x\}$ a family of subgroups continuously indexed by $X$. Consider the equivalence relation on $X\times G$ given by
\[(x,g)\sim_\famS (x',g') \Longleftrightarrow x=x'\text{ and } g^{-1}g'\in S_x \]
Then the quotient space $G/\famS=X\times G/\sim_\famS$, equipped with the quotient topology, is a profinite space. 
\end{prop}
\begin{proof}
The quotient is of course compact, so we must show that it is Hausdorff and totally disconnected. Note also that the map $X\times G\to X$ factors through $G/\famS$. Take two distinct points $[(x_1, g_1)]$ and $[(x_2, g_2)]$ in $G/\famS$, where we use square brackets to denote an equivalence class under $\sim_\famS$. If $x_1\neq x_2$ then take preimages of $X$-clopen sets which separate $x_1$ and $x_2$; these give clopen sets in $G/\famS$ separating our two points. 

So suppose $x_1 = x_2=x$ so that $g_1 S_x\cap g_2 S_x=\emptyset$. Then since $S_x$ is a closed subgroup of $G$ there is an open normal subgroup $W$ of $G$ such that $Wg_1^{-1} g_2 \cap W S_x =\emptyset$. Since $\famS$ is continuously indexed by $X$ there is a clopen neighbourhood $Y$ of $x$ in $X$ such that $S_y\subseteq WS_x$ for all $y\in Y$. The map 
\[W\times W\to G, \quad (w_1, w_2)\mapsto (g_1w_1)^{-1}(g_2 w_2)\]
being continuous, and mapping $(1,1)$ to a point outside $W$, there is an open normal subgroup $U$ of $G$ such that $(g_1u_1)^{-1}(g_2 u_2)$ is not in $W$---and hence not in any $S_y$ for $y\in Y$---for all $u_1, u_2\in U$. Thus the clopen sets $Y\times g_1U$ and $Y\times g_2 U$ of $X\times G$ are disjoint upon passing to the quotient by $\sim_\famS$, and provide disjoint clopen neighbourhoods of $[(x,g_1)]$ and $[(x, g_2)]$ as required.
\end{proof}

\begin{prop}\label{FamiliesUpToConjugacy}
Let $G$ be a profinite group and let $\famS=\{S_x\}_{x\in X}$ be a family of subgroups of $G$ continuously indexed over a profinite set $X$. Let $\gamma\colon X\to G$ be a continuous function. Let $\famS'$ be the family of subgroups 
\[\famS' = \left\{S_x^{\gamma(x)}\mid x\in X \right\} \]
Then $\famS'$ is continuously indexed by $X$ and there is homeomorphism
\[G/\famS\iso G/\famS'\]
compatible with the left $G$-actions on these spaces.
\end{prop}
\begin{proof}
Consider the homeomorphism
\[X\times G\to X\times G,\quad (x,g)\mapsto (x,g\gamma(x)) \]
This takes the closed subset
\[\left\{(g,x)\in G\times X\mid g\in S_x \right\} \]
to the corresponding one for $\famS'$, thus proving that $\famS'$ is continuously indexed by $X$. Furthermore the homeomorphism takes the equivalence relation $\sim_\famS$ to $\sim_{\famS'}$, hence induces the claimed homeomorphism of quotient spaces. 
\end{proof}

\subsection{Sheaves of modules}
It will unfortunately prove necessary for certain applications to consider infinite collections of subgroups, or more generally infinite collections of compact modules. However there is no sensible compact topology on the abstract direct sum of infinitely many compact $G$-modules, so one must use a slightly modified notion of `profinite direct sum of a sheaf of modules'. For the most part this has similar properties to the abstract direct sum, so the casual reader may choose to omit this section.  In this section all modules will be {\em compact}.

\begin{defn}
Let $R$ be a profinite ring. A {\em sheaf of $R$-modules} consists of a triple $({\cal M}, \mu, X)$ with the following properties.
\begin{itemize}
\item ${\cal M}$ and $X$ are profinite spaces and $\mu\colon{\cal M}\to X$ is a continuous surjection.
\item Each {\em `fibre'} ${\cal M}_x=\mu^{-1}(x)$ is endowed with the structure of a compact $R$-module such that the maps
\[R\times{\cal M} \to {\cal M}, \quad (r,m)\to r\cdot m \]
\[{\cal M}^{(2)}=\{(m,n)\in{\cal M}^2\mid \mu(m)=\mu(n)\} \to {\cal M}, \quad (m,n)\to m+n \]
are continuous.
\end{itemize} 

A {\em morphism of sheaves} $(\alpha, \bar\alpha)\colon ({\cal M}, \mu, X) \to ({\cal M'}, \mu', X') $ consists of continuous maps $\alpha:{\cal M}\to {\cal M'}$ and $\bar\alpha\colon X\to X'$ such that $\mu'\alpha = \bar\alpha\mu$ and such that the restriction of $\alpha$ to each fibre is a morphism of $R$-modules ${\cal M}_x\to{\cal M'}_{\bar\alpha(x)}$.
\end{defn}

We often contract `the sheaf $({\cal M}, \mu, X)$' to simply `the sheaf $\cal M$'. Regarding an $R$-module as a sheaf over the one-point space one may talk of a sheaf morphism from a sheaf to an $R$-module.

\begin{defn}
A {\em profinite direct sum} of a sheaf $\cal M$ consists of an $R$-module $\bigboxplus_X\cal M$ and a sheaf morphism $\omega\colon {\cal M}\to \bigboxplus_X \cal M$ (sometimes called the `canonical morphism' such that for any $R$-module $N$ and any sheaf morphism $\beta\colon{\cal M}\to N$ there is a unique morphism of $R$-modules $\tilde\beta\colon\bigboxplus_X{\cal M}\to N$ such that $\tilde\beta\omega=\beta$.
\end{defn}
We will sometimes call this an `external direct sum'. Note that as any compact module is an inverse limit of finite modules it is sufficient to verify this universal property for finite $N$. We may also denote the profinite direct sum as $\bigboxplus_{x\in X} {\cal M}_x$. 

In \cite{Ribes17} this sum is simply denoted with $\bigoplus$. We prefer to use a different notation to remind the reader that the notion is in a certain sense `more rigid' than a traditional direct sum. For example an identification of each fibre ${\cal M}_x$ with some module $N_x$, depending on some choices, need not give a sensible identification of the entire direct sum if the $N_x$ do not form a sheaf in any natural way such that the identifications depend continuously on $X$. A more precise difference from the classical direct sum is that there may no longer be a projection to each fibre ${\cal M}_x$ which vanishes on all the other fibres.  Slightly weaker statements, for instance Corollary \ref{WeakProjections}, must be used instead.

With certain exceptions such as this, the properties of the profinite direct sum are fairly intuitive and mirror those of the abstract direct sum in those respects which will be useful to us. Therefore we will leave the statements of properties of the direct sum that we will use, and their proofs, to \ref{AppSheaves}.

\section{Basic properties of relative cohomology}\label{SecRelCohDef}
\subsection{Definitions and long exact sequence}\label{SecBasicDefs}
\begin{defn}
Let $G$ be a profinite group and let $\famS=\{S_x\}_{x\in X}$ be a family of (closed) subgroups of $G$ continuously indexed by a non-empty profinite space $X$. We will often abbreviate this to `let $(G,\famS)$ be a profinite group pair'. Recall from Section \ref{SecCtslyIndFams} that there is a natural topology on the set 
\[G/\famS = \bigsqcup_{x\in X} G/S_x \]
making it into a profinite space such that the natural map $X\times G\to G/\famS$ is continuous. Here the symbol $\bigsqcup$ denotes `disjoint union of sets'. Let \ZGS\ be the free $\Z[\pi]$-module on $G/\famS$, viewed as a \ZG{}-module via the natural $G$-action on $G/\famS$. Note that by Proposition \ref{DirSumsAndFreeMods} another way of expressing \ZGS\ is as the profinite internal direct sum
\[\ZGS =\bigboxplus_{x\in X}\ZG{/S_x} \]
 
Now consider the augmentation map $\ZGS \to \Z[\pi]$. The kernel, an object of \Gmod[\pi]{C}, will be denoted $\Delta$, or $\Delta_{G,\famS}$ when more precision is desired. This will be the crucial actor in the development of the theory. Note that it is topologically generated by the elements $gS_x-S_y$ for $g\in G$ and $x,y\in X$, as may be readily seen by an inverse limit argument from the finite case. For modules $A\in \Gmod[\pi]{D}, M\in\Gmod[\pi]{C}$ define 
\[H_k(G,\famS;M)=H_{k-1}(G; \Delta\hotimes M),\quad H^k(G,\famS;A)=H^{k-1}(G;\Hom(\Delta,A))\]
Note the dimension shift. Here $\Delta\hotimes M$ and $\Hom(\Delta, A)$ are equipped with diagonal actions, i.e.
\[g\cdot(\delta\otimes m)=(g\delta)\otimes (gm),\quad (g\cdot f)(\delta) = gf(g^{-1}\delta)\]
where $g\in G,m\in M, \delta \in \Delta, f\in \Hom(\Delta, A)$. 
\end{defn}

There is another characterisation of these functors, for which we require the following lemma. 
\begin{lem}
$\Delta$ is projective as a \Z[\pi]-module.
\end{lem}
\begin{proof}
Consider the $p$-primary component of $\Delta$---this is the kernel of the augmentation map $\Zpof{G/\famS}\to \Z[p]$. This is a submodule of the free module $\Zpof{G/\famS}$, hence is $p$-torsion-free and hence $\Z[p]$-free by Proposition \ref{TorFreeImpliesFree}. So each $p$-primary component of $\Delta$ is a projective $\Z[p]$-module; hence $\Delta$ is a projective \Z[\pi]-module by Proposition \ref{pPrimProjectives}. 
\end{proof}
\begin{prop}\label{relequalsext}
There are natural isomorphisms of functors
\[ \Tor^G_\bullet(\Delta^\perp, M)\iso H_{\bullet+1}(G,\famS; M)\iso \Tor^G_\bullet(M^\perp, \Delta), \quad \Ext_G^\bullet(\Delta, M)\iso H^{\bullet+1}(G,\famS; M)\]
where $\Delta^\perp$ denotes $\Delta$ with the canonical right $G$-action $\delta\cdot g= g^{-1}\delta$.
\end{prop}
\begin{proof}
We will prove the statement for cohomology. By the previous lemma, $\Delta$ is projective over $\Z[\pi]$. Thus the functor $\Hom(\Delta,-)$ is an exact functor on \Z[\pi]-modules, and hence also on \ZG{}-modules. Therefore the right derived functors of $\Hom_G(\Z[\pi],\Hom(\Delta,-))$ are precisely 
$\Ext_G^\bullet(\Z[\pi], \Hom(\Delta, -))$. But since there is an isomorphism
\[\Hom_G(\Delta, -)\iso \Hom_G(\Z[\pi], \Hom(\Delta, -)) \]
these are also the right derived functors of $\Hom_G(\Delta, -)$. So we have the required isomorphisms
\begin{equation*} 
\Ext_G^\bullet(\Delta, -)\iso \Ext_G^\bullet(\Z[\pi], \Hom(\Delta, -)) \eqqcolon H^{\bullet+1}(G, \famS; -)  
\end{equation*}
The homology statement is similar, or may be deduced using Pontrjagin duality.
\end{proof}

If $\Delta$ happens to be a module of type FP$_\infty$ then we may extend these definitions to \Gmod[\pi]{P}-functors
\[{\bf H}_{\bullet+1}(G,\famS; M)= \bTor^G_\bullet(\Delta^\perp, M), \quad {\bf H}^{\bullet+1}(G,\famS; M)= \bExt_G^\bullet(\Delta, M)\]
Judicious use of Propositions \ref{FptoPp1} and Proposition \ref{FptoPp2} allows us to extend the various propositions we will prove to the case of these `bold' functors when all required modules have type FP$_\infty$.

\begin{prop}
Let $(G,\famS)$ be a profinite group pair. There is a natural long exact sequence
\begin{equation}
\cdots \to  H^k(G, A)\to H^k(\famS, A) \to  H^{k+1}(G,\famS; A)\to  H^{k+1}(G,A)\to \cdots\label{RelCohLES}
\end{equation} 
where $A\in\Gmod[\pi]{D}$. Similarly in homology.
\end{prop}
\begin{rmk}
For brevity we have defined
 \[H_k(\famS; M) = H_k(G, \ZGS\hotimes M), \quad H^k(\famS; A) = H^k(G, \Hom(\ZGS, A))\]
When \famS\ is a finite collection these are simply the direct sums of the homology and cohomology of the $S_i$. Note that due to the dimension shift inherent in the definition of relative cohomology, the `connecting homomorphisms' in this sequence are actually the maps from $H^\bullet(G,\famS)$ to $H^\bullet (G)$. 
\end{rmk}
\begin{proof}
Let $A\in\Gmod[\pi]{D}$. The short exact sequence
\[\begin{tikzcd} 0\ar{r}& \Delta\ar{r} & \ZGS\ar{r} & \Z[\pi]\ar{r} & 0 \end{tikzcd}\]
splits as an exact sequence of \Z[\pi]-modules, hence remains exact when the functor $\Hom(-,A)$ is applied. So we have a short exact sequence of $G$-modules
\[\begin{tikzcd} 0&\ar{l} \Hom(\Delta, A) &\ar{l} \Hom(\ZGS, A) &\ar{l} A& 0 \ar{l} \end{tikzcd}\]
and applying the functor $H^\bullet(G,-)$ gives the familiar long exact sequence for relative homology.
\end{proof}
\begin{rmk}
Via Proposition \ref{relequalsext}---and a similar one for the module \ZGS---one may also see this as the long exact sequence deriving from the application of the functor $\Ext_G^\bullet(-, A)$ to the short exact sequence
\begin{equation*}\label{DeltaDef}\begin{tikzcd} 0\ar{r}& \Delta\ar{r} & \ZGS\ar{r} & \Z[\pi]\ar{r} & 0 \end{tikzcd}
\end{equation*}
\end{rmk}
\begin{defn}
Let $G$ be a profinite group and let $\famS=\{S_x\}_{x\in X}$ be a family of (closed) subgroups of $G$ continuously indexed by a non-empty profinite space $X$ and let $H$ be a profinite group with a family $\mathcal{T}=\{T_y\}_{y\in Y}$ of closed subgroups, continuously indexed over a profinite set $Y$. A {\em map of profinite group pairs} $(\phi,f)\colon(H,{\cal T})\to (G,\famS)$ consists of a group homomorphism $\phi\colon H\to G$ and a continuous function $f\colon Y\to X$ such that $\phi(T_x)\subseteq S_{f(y)}$ for each $y\in Y$. 
\end{defn}
\begin{prop}\label{GrpPairNatural}
The relative (co)homology functors, and the long exact sequence \eqref{RelCohLES}, are natural with respect to maps of group pairs $(\phi,f)\colon(H,{\cal T})\to (G,\famS)$. 
\end{prop}
\begin{proof}
One sees immediately from the definitions that there is a commuting diagram of $H$-modules
\[\begin{tikzcd}
0\ar{r}&  \Delta_{H,\mathcal{T}}\ar{r}\ar{d} & \Z[\pi][\![H/\mathcal{T}]\!]\ar{r}\ar{d} & \Z[\pi]\ar{r}\ar{d}  & 0\\
0\ar{r}& \Delta_{G,\famS}\ar{r} & \ZGS\ar{r} & \Z[\pi]\ar{r}& 0
\end{tikzcd}\]
and applying the usual functoriality of homology in the group variable gives maps in relative cohomology 
\[H^\bullet(G, \famS; A)\to H^\bullet(H,\mathcal{T}; A)\]
when $A\in \Gmod[\pi]{D}$ is regarded as an $H$-module via $\phi$. Indeed, we obtain maps of the entire long exact sequence \eqref{RelCohLES}. Similarly for homology.
\end{proof}

This is of course not the only `relative cohomology sequence'. We mention in particular that there is a long exact sequence which allows one to carry out certain inductions when $S$ is a finite family.
\begin{prop}
Suppose that $G$ is a profinite group and $\famS_i$ is a family of subgroups of $G$ continuously indexed over a non-empty profinite set $X_i$ for $i=1,2$. Let $\famS=\famS_1\sqcup \famS_2$ be the natural family of subgroups of $G$ continuously indexed over the disjoint union of $X_1$ and $X_2$. Then there is a natural long exact sequence
\[\cdots \to  H^k(G,\famS_1; A)\to H^k(\famS_2; A) \to  H^{k+1}(G,\famS; A)\to  H^{k+1}(G, \famS_1;A)\to \cdots \]
$A\in\Gmod[\pi]{D}$. Similarly in homology.
\end{prop}
\begin{proof}
Apply the functor $\Ext^\bullet_G(-,A)$ to the top row of the commuting diagram with exact rows and columns
\begin{equation}\label{InductionOnPairs}\begin{tikzcd}
\Delta_{G, \famS_1}\ar[hookrightarrow]{r}\ar[hookrightarrow]{d} & \Delta_{G, \famS} \ar[twoheadrightarrow]{r}\ar[hookrightarrow]{d} & \Zpiof{G/\famS_2}\ar[equal]{d} \\
\Zpiof{G/\famS_1} \ar[hookrightarrow]{r}\ar[twoheadrightarrow]{d} & 
\Zpiof{G/\famS} \ar[twoheadrightarrow]{r}\ar[twoheadrightarrow]{d} &\Zpiof{G/\famS_2} \\
\Z[\pi]\ar[equal]{r} & \Z[\pi] & 
\end{tikzcd} \end{equation}
This exact sequence also of course has the expected naturality properties.
\end{proof} 
\begin{prop}
Pontrjagin duality holds for relative cohomology. That is, for $M\in\Gmod[]{C}$ there are natural isomorphisms
\[H_n(G,\famS;M)^\ast\iso H^n(G,\famS;M^\ast) \]
\end{prop}
\begin{proof}
For $M\in\Gmod[]{C}$ we have
\begin{align*} 
H_n(G,\famS;M)^\ast & = H_{n-1}(G, \Delta\hotimes M)^\ast \\ &= H^{n-1}(G, (\Delta\hotimes M)^\ast)\\ &= H^{n-1}(G, \Hom(\Delta, M^\ast))\\ & = H^n(G, \famS; M^\ast)
\end{align*}
where the manipulation from the second to the third lines is an application of Proposition \ref{TensorHom}.
\end{proof}

We also remark that the relative cohomology is `invariant up to conjugacy of \famS' in the following manner.
\begin{prop}\label{DeltaUpToConjugacy}
Let $G$ be a profinite group and let $\famS=\{S_x\}_{x\in X}$ be a family of subgroups of $G$ continuously indexed over a profinite set $X$. Let $\gamma\colon X\to G$ be a continuous function. Let $\famS'$ be the family of subgroups 
\[\famS^\gamma = \left\{S_x^{\gamma(x)}\mid x\in X \right\} \]
Then $\famS^\gamma$ is continuously indexed by $X$ and there is a commuting diagram of isomorphisms
\[\begin{tikzcd} 
0\ar{r}& \Delta_{G,\famS}\ar{r}\ar{d}{\iso} & \ZGS\ar{r}\ar{d}{\iso} & \Z[\pi]\ar{r}\ar{d}{\id} & 0 \\
0\ar{r}& \Delta_{G,\famS^\gamma} \ar{r} & \Zpiof{G/\famS^\gamma}\ar{r} & \Z[\pi]\ar{r} & 0 
\end{tikzcd}\]
and hence isomorphisms of functors
\[H_\ast(G,\famS;-)\iso H_\ast(G,\famS^\gamma;-),\quad H^\ast(G,\famS;-)\iso H^\ast(G,\famS^\gamma;-) \]
\end{prop}
\begin{proof}
This follows immediately from Proposition \ref{FamiliesUpToConjugacy}.
\end{proof}

\subsection{Restriction to subgroups}\label{SecSubgpPairs}
Let $G$ be a profinite group and let \famS\ be a family of subgroups continuously indexed by the profinite space $X$. Let $H$ be a closed subgroup of $G$. Fix a continuous section $\sigma\colon H\lqt G\to G$ of the quotient map $G\to H\lqt G$. Such a section exists by Proposition 2.2.2 of \cite{RZ00}. There is a family of subgroups 
\[\famS^H_\sigma = \left\{ H\cap \sigma(y)S_x \sigma(y)^{-1} \mid x\in X, y\in H\lqt G/S_x\right\} \]
indexed over the profinite set \[H\lqt G/\famS=\bigsqcup_{x\in X} H\lqt G/S_x \] This indexing set has a natural topology making it a profinite space such that the natural quotient $G/\famS\to H\lqt G/\famS$ is continuous. Here we have abused notation by writing $\sigma$ for the section $H\lqt G/\famS\to G/\famS$ induced by $\sigma\colon H\lqt G\to G$.

Before we proceed further we must check that $\famS^H_\sigma$ is continuously indexed by $H\lqt G/\famS$. This is true provided that the subset
\[\left\{(HgS_x, h)\in H\lqt G/\famS\times H \mid h\in H\cap \sigma(y)S_x \sigma(y)^{-1} \right\} \]
is a closed subset of the constant sheaf $H\lqt G/\famS \times H$. This subset is the preimage of the closed subset $\{1S_x\mid x\in X\}\subseteq G/\famS$ under the continuous map
\[H\lqt G/\famS \times H\to G/\famS, \quad (HgS_x, h)\mapsto \sigma(HgS_x)^{-1} h \sigma(HgS_x)S_x \]
hence is closed. 

If $\sigma'\colon H\lqt G\to G$ is another section then we have a continuous function $\gamma(Hg) = \sigma(Hg)\sigma'(Hg)^{-1}$ from $H\lqt G$ to $H$. Furthermore for every $y\in H\lqt G$ we have 
\[H\cap \sigma'(y)S_x \sigma'(y)^{-1} = \left(H\cap \sigma(y)S_x \sigma(y)^{-1}\right)^{\gamma(y)}\]
Hence by Proposition \ref{DeltaUpToConjugacy} the module $\Delta_{H,\famS^H_\sigma}$ does not depend on $\sigma$ (up to canonical isomorphism). We will usually therefore drop $\sigma$ from the notation.

So we may consider the relative (co)homology of $H$ relative to $\famS^H$. To compare this to the relative (co)homology of $G$ we note the following important result.
\begin{prop}\label{SubgpPairDelta}
Let $G$ be a profinite group, \famS\ a family of subgroups continuously indexed by the profinite space $X$. Let $H$ be a closed subgroup of $G$. Then for any continuous section $\sigma\colon H\lqt G\to G$ we have a canonical isomorphism of $H$-modules
\[\Delta_{H,\famS^H_\sigma} \iso \res^G_H(\Delta_{G,\famS}) \]
\end{prop}
\begin{proof}
Consider the continuous map
\[H/\famS^H_\sigma \to G/\famS, \quad h(H\cap \sigma(y)S_x \sigma(y)^{-1})\mapsto h\sigma(y)S_x\] which is easily checked to be a bijection, hence a homeomorphism as both spaces are profinite. It is also compatible with the left $H$-action, so there is an isomorphism of $H$-modules
\[\Zpiof{H/\famS^H_\sigma}\iso\Zpiof{G/\famS} \] 
This isomorphism commutes with the augmentation maps to \Z[\pi] so the kernels of these maps are isomorphic $H$-modules. This is precisely the statement of the proposition.
\end{proof}

\begin{clly}\label{FPnAndFISubGps}
If $U$ is an open subgroup of $G$ then $\Delta_{G,\famS}$ is of type FP$_\infty$ as a $G$-module if and only if $\Delta_{U,\famS^U}$ is of type FP$_\infty$ as a $U$-module.
\end{clly} 
\begin{proof}
This now follows immediately from Proposition \ref{FPnModsAndSubgps}.
\end{proof}

As a result of Proposition \ref{SubgpPairDelta}, we obtain a `relative Shapiro lemma'.
\begin{prop}[`Relative Shapiro Lemma']
Let $(G,\famS)$ be a group pair and $H$ a closed subgroup of $G$. Then for any $M\in \mathfrak{C}_\pi(H)$ and any $A\in\mathfrak{D}_\pi(H)$ there are natural isomorphisms
\[H_\bullet(H, \famS^H; M)\iso H_\bullet(G,\famS; \ind^H_G(M)), \quad H^\bullet(H, \famS^H; A)\iso H^\bullet(G,\famS; \coind^H_G(A)) \]
\end{prop}

\begin{proof}
Let $\Delta=\Delta_{G,\famS}$, so that $\Delta_{H,\famS^H}=\res^G_H(\Delta)$. There is an isomorphism of left $G$-modules
\begin{eqnarray*}\ZG{}\hotimes[H](\res^G_H(\Delta)\hotimes M) &\iso& \Delta\hotimes[](\ZG{}\hotimes[H] M) \\
g\otimes \delta\otimes m &\mapsto& (g\delta)\otimes g\otimes m
\end{eqnarray*}
Here, as in Section \ref{SecUsefulIds}, the tensor products over $\Z[\pi]$ are given the diagonal action and the tensor products over $H$ have the left $G$-action given by the left action on \ZG{}.
From this isomorphism and the usual Shapiro isomorphisms (Proposition \ref{Shapiro1}) we have the first isomorphism from the statement of the proposition. The second follows from the isomorphism
\begin{eqnarray*}\Hom_H(\ZG{},\Hom(\res^G_H(\Delta), A)) &\iso& \Hom(\Delta, \Hom_H(\ZG{}, A)) \\
f  &\mapsto& \big(\delta\mapsto \left(g\mapsto f(g)(g\delta) \right) \big)
\end{eqnarray*}
and the absolute Shapiro lemma.
\end{proof}
\begin{rmk}
Suppose $U$ is open in $G$ and $\Delta_{G,\famS}$ is of type FP$_\infty$ as a $G$-module. Then as for the absolute Shapiro lemma one may take limits to find natural isomorphisms
\[{\bf H}_\bullet(G,\famS;\bind_G^U(M))\iso {\bf H}_\bullet(U,\famS^U; M), \quad {\bf H}^\bullet(G,\famS;\bcoind_G^U(M))\iso {\bf H}^\bullet(U,\famS^U; M) \]
for all $M\in \mathfrak{P}_\pi(U)$.
\end{rmk} 

We may now define restriction and corestriction maps.
\begin{defn}
Let $(G,\famS)$ be a profinite group pair and $H$ a closed subgroup of $G$. For $M\in\Gmod[\pi]{C}$ the {\em restriction map on homology} is given by the map
\[\res^H_G\colon H_\bullet(H,\famS^H; M)\iso H_\bullet(G,\famS; \ind_G^H(M))\to H_\bullet(G,\famS; M)\]
corresponding to the map 
\[ \ZG{}\hotimes[H] M\to M, \quad g\otimes m\mapsto gm\]
For $A\in \Gmod[\pi]{D}$ the {\em restriction map on cohomology} is the map
\[\res^G_H\colon H^\bullet(G,\famS; A)\to H^\bullet(G,\famS; \coind^H_G(A)) \iso H^\bullet(H, \famS^H; A)\]
induced by the module map
\[A\to\Hom_H(\ZG{}, A), \quad a\mapsto \left(g\mapsto ga\right) \]
\end{defn}
One may verify that these agree with the functorial maps arising from the obvious map of group pairs $(H,\famS^H)\to(G,\famS)$.
\begin{defn}
Let $(G,\famS)$ be a profinite group pair and $U$ a open subgroup of $G$. For $M\in\Gmod[\pi]{C}$ the {\em corestriction map on homology} is given by the map
\[\cor^G_U\colon H_\bullet(G,\famS; M) \to H_\bullet(G,\famS; \ind_G^U(M))\iso H_\bullet(U,\famS^U; M)\]
corresponding to the map 
\[ M \to \ZG{}\hotimes[U]M , \quad m\mapsto \sum_i g_i\otimes (g_i^{-1}m)\]
where $\{g_i\}$ is a complete set of coset representatives of $G/U$. The map above does not depend on the choice of these representatives.

For $A\in \Gmod[\pi]{D}$ the {\em corestriction map on cohomology} is the map
\[\cor^U_G\colon H^\bullet(U, \famS^U; A)\iso H^\bullet(G,\famS; \coind_G^U(A))\to   H^\bullet(G,\famS; A)\]
induced by the module map
\[\Hom_U(\ZG{}, A)\to A, \quad \left(f\mapsto \sum g_i f(g_i^{-1}) \right)\]
where again $\{g_i\}$ is an (irrelevant) choice of coset representatives of $G/U$.
\end{defn}
The corestriction map is also often known as a {\em transfer map}. The following proposition follows immediately from the definitions.
\begin{prop}\label{corres}
For a profinite group pair $(G,\famS)$ and an open subgroup $U$ of $G$, the composition \[\cor^U_G\circ \res^G_U\colon H^\bullet(G,\famS; A)\to H^\bullet(G,\famS; A) \]
is multiplication by $[G:U]$ for any $A\in \Gmod[\pi]{D}$. Similarly for any $M\in\Gmod[\pi]{C}$ the composition  
\[\res^U_G\circ \cor^G_U\colon H_\bullet(G,\famS; M)\to H_\bullet(G,\famS; M) \]
is multiplication by $[G:U]$.
\end{prop}
\begin{prop}\label{IntersectionsOfSubGps}
Let $(G,\famS)$ be a profinite group pair and let $K_i$ be a nested descending sequence of closed subgroups of $G$. For each $i$ let $A_i\in\mathfrak{D}_\pi(K_i)$ and let $A_i\to A_{i+1}$ be a sequence of maps compatible with the inclusions $K_{i+1}\to K_i$. For each $i$ let  $M_i\in\mathfrak{C}_\pi(K_i)$ be a $K_i$-module and let $M_{i+1}\to M_{i}$ be a sequence of maps compatible with the inclusions $K_{i+1}\to K_i$. Then if $L=\bigcap K_i$ we have
\[H_\bullet(L, \famS^L; \varprojlim M_i) = \varprojlim H_\bullet(K_i,\famS^{K_i}; M_i), \quad H^\bullet(L, \famS^L; \varinjlim A_i) = \varinjlim H^\bullet(K_i,\famS^{K_i}; A_i) \]
\end{prop}
\begin{proof}
There is a natural isomorphism of $G$-modules
\[ \varinjlim \Hom_{K_i}(\Zpiof{G}, A_i) \stackrel{\iso}{\longrightarrow} \Hom_L(\Zpiof{G}, \varinjlim A_i)\]
and the cohomology statement follows from the relative Shapiro Lemma. The homology statement follows by Pontrjagin duality.
\end{proof}
\subsection{Cohomological Dimension of a Pair}
Let $p$ be a prime. By analogy with the absolute case one may define
\[ \hd_p(G,\famS) = \max\left\{n\mid \exists M\in\Gmod[p]{C} \text{ such that }H_n(G,\famS; M)\neq 0\right\}\] 
\[ {\rm cd}_p(G,\famS) = \max\left\{n\mid \exists A\in\Gmod[p]{D} \text{ such that }H^n(G,\famS; A)\neq 0\right\}\] 
These are in fact equal by the Pontrjagin duality of the last section, and we will use them interchangeably. By convention we set the dimension to be zero if the (co)homology vanishes in every dimension for all modules, and $\infty$ if the defining sets are unbounded.

Of course by the definition of relative cohomology we have $\cd_p(G,\famS)\leq \cd_p(G)+1$. Furthermore from the long exact sequence in relative homology we see that if $\cd_p(S_x)<\cd_p(G)-1$ for all $x\in X$ then $\cd_p(G,\famS)\leq \cd_p(G)$. 
\begin{prop}
Let $(G,\famS)$ be a profinite group pair and $H$ a closed subgroup of $G$. Then
\[\cd_p(H,\famS^H)\leq \cd_p(G,\famS) \]
\end{prop}
\begin{proof}
This follows immediately from the relative Shapiro lemma of the last section.
\end{proof}

The next propositions are simply relative forms of Propositions 21 and $21'$ of \cite{Serre13}, and the proofs are much the same. We will give a sketch here for the convenience of the reader, but the details will be left out. 
\begin{prop}\label{CDofProP}
For a profinite group pair $(G,\famS)$, assume that $G$ is a pro-$p$ group. Then the following are equivalent:
\begin{enumerate}[(1)]
\item $\cd_p(G,\famS)\leq n$
\item $H^{n+1}(G, \famS; \F_p)=0$
\end{enumerate}
\end{prop}
\begin{proof}
This follows from a standard d\'evissage argument. All non-simple modules in \Gmod[\pi]{F} are successive extensions by smaller finite modules, and the only simple module is $\F_p$ (see Lemma 7.1.5 of \cite{RZ00}). Now apply the long exact sequence for $\Ext^\bullet_G(\Delta, -)$ to find $H^{n+1}(G,\famS; M)=0$ for all $M\in\Gmod[\pi]{F}$. Taking a direct limit gives the result.
\end{proof}

\begin{prop}\label{cdintermsofFp}
The following are equivalent, for a profinite group pair $(G,\famS)$.
\begin{enumerate}[(1)]
\item $\cd_p(G,\famS)\leq n$
\item $H^{n+1}(H, \famS^H; \F_p)=0$ for all closed subgroups $H$ of $G$
\item $H^{n+1}(U, \famS^U; \F_p)=0$ for all open subgroups $U$ of $G$
\end{enumerate}
\end{prop}
\begin{proof}
That (1) implies (2) is an immediate consequence of Proposition \ref{CDofProP}. That (2) implies (3) is trivial. That (3) implies (2) follows from Proposition \ref{IntersectionsOfSubGps} applied to the trivial module $\F_p$ and a descending sequence of open subgroups of $G$ intersecting in $H$.

Finally assume that (2) holds and let $A\in\Gmod[p]{D}$. Taking $H$ to be a $p$-Sylow subgroup of $G$, the previous proposition implies that $\cd_p(H,\famS^H)\leq n$. So $H^{n+1}(H,\famS^H; A)=0$. Now write $H$ as the intersection of a decreasing sequence of open subgroups $U_i$. By Proposition \ref{corres} the composition 
\[\cor^{U_i}_G\circ \res^G_{U_i}\colon H^{n+1}(G,\famS;A)\to H^{n+1}(G,\famS;A) \]
is multiplication by $[G:U_i]$, hence is an injective map since $[G:U_i]$ is coprime to $p$. Hence for every $i$ the restriction map \[H^{n+1}(G,\famS;A)\to H^{n+1}(U_i,\famS^{U_i};A)\] is injective. Hence the restriction map
\[H^{n+1}(G,\famS;A)\to \varinjlim H^{n+1}(U_i,\famS^{U_i};A) = H^{n+1}(H,\famS^{H};A)= 0\]
is also injective and we are done. 
\end{proof}

We note one intuitively clear result for future use.
\begin{lem}\label{LemReductionOfFamily}
Suppose that the family $\famS$ is such that at most one subgroup $S_0$ is non-trivial. Then for every $k> 1$ and every $M\in \Gmod[\pi]{C}, A\in\Gmod[\pi]{D}$, we have 
\[H_k(G,\famS; M)=H_k(G, \{S_0\}; M), \quad H^k(G,\famS; A)=H^k(G, \{S_0\};A)  \]
In particular, $\cd_p(G,\famS)=\cd_p(G,\{S_0\})$ for all $p\in \pi$ except possibly if one  dimension is 1 and the other 0.
\end{lem}
\begin{proof}
Note that for all $k$
\[ H_k(\famS; M) = \Tor^{G}_k(\Z[\pi], \Zpiof{G/\famS}\hotimes[] M)=\bigboxplus_{S_x\in\famS} \Tor^{G}_k(\Z[\pi], \ind^{S_x}_G(M))  = \Tor^{G}_k(\Z[\pi], \ind^{S_0}_G(M)) = H_k(S_0; M) \]
for $k\geq 1$. The second equality follows from Proposition \ref{DirectSumsAndTor}. Using the absolute Shapiro Lemma all the groups \[\Tor^{G}_k(\Zhat, \ind^{S_0}_G(M))= \Tor^{G}_k(\Zhat, \ind^{1}_G(M))=H_k(1, M) \]
vanish for $S_x\neq S_0$ so the third equality follows from Lemma \ref{LemSheafOfTrivs}.

 The result on homology now follows from the 5-Lemma and the map of long exact sequences in relative homology corresponding to the map of pairs $(G, \{S_0\})\to (G, \famS)$. The cohomology statement follows by Pontrjagin duality.
\end{proof}
\begin{rmk}
The discrepancy in the cohomological dimension statement is just about possible---for example $\cd_p(1, \{1\})=0$ but $\cd_p(1, \{1\}\sqcup\{1\})=1$.
\end{rmk}
\subsection{Spectral sequence}
We briefly mention that there is a spectral sequence associated to an `extension of group pairs' in a suitable sense. Recall first that the Lyndon-Hochschild-Serre spectral sequence exists in the world of profinite groups.
\begin{prop}[see Theorem 7.2.4 of \cite{RZ00}]\label{LHS1}
Let $G$ be a profinite group and $N$ a closed normal subgroup of $G$ with $G/N=Q$. For $A\in\Gmod[\pi]{D}$ there is a natural convergent first quadrant cohomological spectral sequence
\[E_2^{rs}=H^r(Q, H^s(N; A))\Rightarrow H^{r+s}(G,A) \]
\end{prop}

\begin{prop}
Let $(G,\famS)$ be a profinite group pair where $\famS$ is continuously indexed by a profinite space $X$ and let $N$ be a closed normal subgroup of $G$ contained in $S$ for all $S\in\famS$. Let $Q=G/N$ and let ${\cal T}=\{S/N\mid S\in\famS\}$ be a family of subgroups of $Q$; then $\cal T$ is continuously indexed by $X$. The group $N$ acts trivially on $\Delta_{G,\famS}$ and $\Delta_{G,\famS}$, regarded as a $Q$-module, is naturally isomorphic to $\Delta_{Q,\cal T}$.
\end{prop}
\begin{proof}
The statement that $\cal T$ is continuously indexed follows immediately from the definition. For $x\in G$, $n\in N$, $S\in\famS$ the formula
\[n\cdot xS=xx^{-1}nxS=xS \]
shows that $N$ acts trivially on $\Zpiof{G/\famS}$, hence on $\Delta_{G,\famS}$. Finally the natural continuous bijections $G/S\iso (G/N)/(S/N)$ show that $\Zpiof{G/\famS}$ is naturally isomorphic as a $G$-module to $\Zpiof{Q/\cal T}$ in a way compatible with the augmentation to \Z[\pi]. Hence  $\Delta_{G,\famS}$ is naturally isomorphic to $\Delta_{Q,\cal T}$.
\end{proof}

\begin{prop}\label{LHS2}
Let $(G,\famS)$ be a profinite group pair and let $N$ be a closed normal subgroup of $G$ contained in $S$ for all $S\in\famS$. Let $Q=G/N$ and let ${\cal T}=\{S/N\mid S\in\famS\}$. For $A\in\Gmod[\pi]{D}$ there is a natural convergent first quadrant cohomological spectral sequence
\[E_2^{rs}=H^r(Q,{\cal T}; H^s(N; A))\Rightarrow H^{r+s}(G,\famS; A) \]
\end{prop}
\begin{proof}
This follows from the spectral sequence in Proposition \ref{LHS1} via the natural isomorphism
\[H^s(N; \Hom(\Delta, A)) = \Hom(\Delta, H^s(N, A)) \]
where $\Delta=\Delta_{G,\famS}$. This isomorphism holds because $\Delta$, viewed as an $N$-module, is simply a free \Z[\pi]-module with trivial $N$-action. Hence the functor $\Hom(\Delta, -)$ is an exact functor with
\[ \Hom_N(P, \Hom(\Delta, A)) = \Hom(\Delta, \Hom_N(P,A))\]  
(as $G$-modules) for every $P\in \mathfrak{C}_\pi(N)$ and the cohomology isomorphisms follow.
\end{proof}

\section{Cup products}\label{SecCupProds}
\begin{rmk}\label{nocapprods}
The classical theory of Poincar\'e duality pairs in \cite{BE77} makes extensive use of cap products, particularly with respect to a `fundamental class' in the top-dimensional homology. The cap product in its most general form consists of a family of homomorphisms
\[H_{r+s}(G,C)\otimes H^s(G,A)\to H_r(G,C\otimes A)\]
However here we run into an issue. Without additional conditions on $G$, group homology is defined only with compact coefficients, and cohomology with discrete torsion coefficients. So to make sense of the above homomorphisms in general one needs to make sense of the tensor product of a compact module $C$ with a discrete module $A$ in such a way that $C\otimes A$ is a compact module. There is no way to do this in the generality required, so we must consider cap products to be largely unavailable to us. Therefore we will develop a duality theory only making use of cup products---and therefore only using discrete modules. 

One may note in passing that such a cap product is well defined and behaves as one would expect in the case when $A$ is finite. This is not really sufficient for our purposes, but does suffice to treat absolute duality groups. The reader interested in the theory of absolute profinite duality groups developed in this way is directed to \cite{Pletch80a, Pletch80b}.
\end{rmk}
\subsection{Definition and basic properties}
Let $C, A, B\in \Gmod[\pi]{D}$ and $\fkB\in\Hom_G(C\otimes A, B)$. We define a cup product as follows. Let $P_\bullet\to\Z[\pi]$ be a projective resolution of \Z[\pi] by left $G$-modules. Then the complex $Q_\bullet= {\rm tot}(P_\bullet \hotimes P_\bullet)$ is also a projective resolution of \Z[\pi]. For $\zeta\in \Hom_G(P_r,C)$ and $\xi\in\Hom_G(P_s,A)$ we may define a cup product 
\[\zeta\smile\xi = \zeta\smile_\fkB \xi \colon Q_{r+s} \stackrel{\rm pr}{\longrightarrow} P_r\hotimes P_s \stackrel{\zeta\otimes\xi}{\longrightarrow} C\otimes A\stackrel{\fkB}{\longrightarrow} B\] 
at the level of cochains. One may readily check that this descends to a map on cohomology using the formula 
\[d(\zeta\smile\xi)=d\zeta\smile \xi +(-1)^r\zeta\smile d\xi \]
Standard arguments from homological algebra show that this definition is independent of the chosen resolution $P_\bullet$. 

We will mostly be utilising the cup product to explore duality notions, and hence require a species of adjoint cup product. In the classical case one does this by taking cap products with respect to a chosen class in $H_n$. Here we use the dual notion. Again given $C, A, B\in \Gmod[\pi]{D}$ and $\fkB\in\Hom_G(C\otimes A, B)$, choose also a `coclass' $e\colon H^n(G,B)\to I_\pi$. Now define the `adjoint cup product'
\[\Upsilon_e=\Upsilon_{e,\fkB}\colon H^k(G,C)\to H^{n-k}(G,A)^\ast \]
by means of the composition
\[H^k(G,C)\otimes H^{n-k}(G,A)\stackrel{\smile}{\longrightarrow} H^n(G,B)\stackrel{e}{\longrightarrow} I_\pi \] 
One basic example of a cup product that we will encounter frequently arises from the pairing 
\begin{eqnarray*}
\fkB\colon \Hom(\Delta,C)\otimes A &\to & \Hom(\Delta, C\otimes A) \\
f\otimes a&\mapsto & (\delta\mapsto f(\delta)\otimes a) 
\end{eqnarray*}
which yields, by definition of relative cohomology, a cup product
\[H^r(G,\famS; C)\otimes H^s(G, A)\to H^{r+s}(G,\famS; C\otimes A)\]

The cup product is natural with respect to group homomorphisms in the following sense. 
\begin{prop}\label{GrpPairCupProd1}
Let $G, H$ be profinite groups. Let $C, A, B\in \Gmod[\pi]{D}$ and $\fkB\in\Hom_G(C\otimes A, B)$. Let $C', A', B'\in \mathfrak{D}_\pi(H)$ and $\fkB'\in\Hom_H(C'\otimes A', B')$. Let $f\colon H\to G$ be a group homomorphism, and let $c\colon C\to C'$, $a\colon A\to A'$, $b\colon B\to B'$ be continuous group homomorphisms compatible with $f$ in the obvious way, and assume that $b\circ \fkB = \fkB'\circ(c\otimes a)$. Then the diagram 
\[\begin{tikzcd}
H^r(G,C)\otimes H^s(G,A) \ar{r}\ar{d}{\smile_\fkB} & H^r(H,C')\otimes H^s(H,A')\ar{d}{\smile_{\fkB'}}\\
H^{r+s}(G,B)\ar{r} & H^{r+s}(H,B')
\end{tikzcd}\]
commutes, so that cup products are natural with respect to group homomorphisms.  
\end{prop}
\begin{proof}
 One may readily check this on the level of cochains.
\end{proof} 
We will of course be requiring maps not just of single cohomology groups, but of entire exact sequences. Given short exact sequences of modules in \Gmod[\pi]{D}
\begin{eqnarray*}
0\to C_1\stackrel{\gamma_1}{\longrightarrow} C_2\stackrel{\gamma_2}{\longrightarrow} C_3\to 0\\
0\to A_3\stackrel{\alpha_2}{\longrightarrow} A_2\stackrel{\alpha_1}{\longrightarrow} A_1\to 0
\end{eqnarray*}
and pairings $\fkB_i\in \Hom_G(C_i\otimes A_i, B)$, we say that these pairings are {\em compatible} with the maps in the short exact sequences if the diagrams
\[\begin{tikzcd} 
C_1\otimes A_2 \ar{r}{\id\otimes \alpha_1}\ar{d}{\gamma_1\otimes\id} & C_1\otimes A_1 \ar{d}{\fkB_1} & & C_2\otimes A_3 \ar{r}{\id\otimes \alpha_2}\ar{d}{\gamma_2\otimes\id} & C_2\otimes A_2 \ar{d}{\fkB_2} \\
C_2\otimes A_2\ar{r}{\fkB_2} & B & & C_3\otimes A_3\ar{r}{\fkB_3} & B
\end{tikzcd}\]
commute. 
\begin{theorem}\label{BaseCupProdLES}
Consider short exact sequences of modules in \Gmod[\pi]{D}\ and compatible pairings as above. Assume also that these short exact sequences split as sequences of \Z[\pi]-modules (i.e.\ as abelian groups). Choose a coclass $e\colon H^n(G,B)\to I_\pi$. Then the diagram of long exact sequences
\begin{equation}\label{cupprodLES}\begin{tikzcd}[cramped]
H^k(G,C_1) \ar{r}\ar{d}{\Upsilon_{e,\fkB_1}} & H^k(G, C_2)\ar{r}\ar{d}{\Upsilon_{e,\fkB_2}} &H^k(G,C_3)\ar{r}\ar{d}{\Upsilon_{e,\fkB_3}} & H^{k+1}(G, C_1)\ar{d}{\Upsilon_{e,\fkB_1}}\\
H^{n-k}(G,A_1)^\ast \ar{r} & H^{n-k}(G, A_2)^\ast\ar{r} &H^{n-k}(G,A_3)^\ast\ar{r}{(-1)^k} & H^{n-k-1}(G, A_1)^\ast
\end{tikzcd}\end{equation} 
commutes where the lower sequence is the dual of the coefficient sequence for the $A_i$, with an occasional sign change in the connecting map as shown. 
\end{theorem}
\begin{rmk}
The restriction that the short exact sequences are $\Z[\pi]$-split will not hinder the results of the rest of the paper, but is necessary here.
\end{rmk}
\begin{proof}
Take a projective resolution $P_\bullet\to\Z[\pi]$ and let $Q_\bullet={\rm tot}(P_\bullet \hotimes P_\bullet)$. Equipped with the augmentation map $P_0\hotimes P_0 \to \Z[\pi]\hotimes \Z[\pi]=\Z[\pi]$ we know that $Q_\bullet$ is also a projective resolution of $\Z[\pi]$ by Proposition \ref{DblResolutions}. Commutativity of diagrams such as
\[\begin{tikzcd}[cramped]
\Hom(P_r,C_1)\otimes \Hom(P_s, A_2) \ar{dd}{\id\otimes\alpha_1} \ar{rr}{\gamma_1\otimes\id}\ar{dr}{\smile} & &\Hom(P_r,C_2)\otimes\Hom(P_s,A_2)\ar{d}{\smile}  \\
& \Hom(Q_{r+s}, C_1\otimes A_2) \ar{d}{\id\otimes\alpha_1} \ar{r}{\gamma_1\otimes\id} & \Hom(Q_{r+s}, C_2\otimes A_2)\ar{d}{\fkB_2} \\
\Hom(P_r,C_1)\otimes\Hom(P_s,A_1) \ar{r}{\smile}  &\Hom(Q_{r+s}, C_1\otimes A_1) \ar{r}{\fkB_1} & \Hom(Q_{r+s}, B)
\end{tikzcd}\]
follows immediately from the definitions. Taking homology gives a commuting diagram
\begin{equation}\label{CupProductSquare}
\begin{tikzcd} H^r(G,C_1)\otimes H^s(G,A_2)\ar{r}{\gamma_1 \otimes\id} \ar{d}{\id\otimes\alpha_1} & H^r(G,C_2)\otimes H^s(G,A_2)\ar{d}{\smile_{\fkB_2}} \\ 
H^r(G,C_1)\otimes H^s(G,A_1)\ar{r}{\smile_{\fkB_1}}& H^{r+s}(G,B) \end{tikzcd}
\end{equation}
Passing to the adjoint cup product, this diagram guarantees that the first square in \eqref{cupprodLES} does indeed commute. Similarly for the second square.

To deal with the square involving the connecting homomorphisms we must show that the diagram 
\[\begin{tikzcd}[cramped, column sep= small]
H^r(G,C_3)\otimes H^s(G,A_1) \ar{dd}{(-1)^r\id\otimes\delta} \ar{rr}{\delta\otimes\id}\ar{dr}{\smile} & &H^{r+1}(G,C_1)\otimes H^{s}(G,A_1)\ar{d}{\smile}  \\
& H^{r+s}(G, C_3\otimes A_1) \ar{d}{\delta} \ar{r}{\delta} & H^{r+s+1}(G, C_1\otimes A_1)\ar{d}{\fkB_1} \\
H^{r}(G,C_3)\otimes H^{s+1}(G,A_3) \ar{r}{\smile}  &H^{r+s+1}(G, C_3\otimes A_3) \ar{r}{\fkB_3} & H^{r+s+1}(G, B)
\end{tikzcd}\]
commutes, where the various symbols $\delta$ denote connecting homomorphisms. That the uppermost and leftmost quadrilaterals commute may be deduced very quickly from the definitions of cup products and connecting maps. The sign in the leftmost quadrilateral arises since the relevant part of the connecting homomorphism on the right vertical side of the quadrilateral derives via the snake lemma from the differentials
\[P_r\hotimes P_{s+1} \to P_r\hotimes P_s\]
in the double complex, which are given by $(-1)^r\id\otimes\, d^P_{s+1}$. 

The lower-right square requires considerably more effort. Viewing the given short exact sequences of modules as chain complexes with $A_2$ and $C_2$ in degree zero, consider the double complex below. All rows and columns are exact by the splitness assumption. The signs on the arrows are to remind the reader of the sign convention from Section \ref{SecDblComplexes}.
\[\begin{tikzcd}
& 0\ar{d} & 0\ar{d} &0\ar{d} & \\
0\ar{r} &C_1\otimes A_3 \ar{r}\ar{d}{-} &C_2\otimes A_3 \ar{r}\ar{d} &C_3\otimes A_3 \ar{r}\ar{d}{-} & 0\\
0\ar{r} &C_1\otimes A_2 \ar{r}\ar{d}{-} &C_2\otimes A_2 \ar{r}\ar{d} &C_3\otimes A_2 \ar{r}\ar{d}{-} & 0\\
0\ar{r} &C_1\otimes A_1 \ar{r}\ar{d} &C_2\otimes A_1 \ar{r}\ar{d} &C_3\otimes A_1 \ar{r}\ar{d} & 0\\
&0&0&0&
\end{tikzcd}\]
One may verify by a direct diagram chase that there is a commuting diagram in which both rows are exact 
\begin{equation}\begin{tikzcd}\label{kerd1seq}
0\ar{r}& C_1\otimes A_1\oplus C_3\otimes A_3 \ar{r} &C_2\otimes A_1\oplus C_3\otimes A_2 \ar{r} & C_3\otimes A_1\oplus C_3\otimes A_1 \ar{r} & 0\\
0\ar{r}& C_1\otimes A_1\oplus C_3\otimes A_3 \ar{r}\ar[equal]{u} &{\rm ker }( d_{-1}) \ar{r}\ar[hookrightarrow]{u} &C_3\otimes A_1 \ar{r}\ar[hookrightarrow]{u}{\rm diag} & 0
\end{tikzcd}\end{equation}
where ${\rm ker }( d_{-1})$ is the kernel of the differential map of the total complex, and is mapped to $C_3\otimes A_1$ via the map $C_2\otimes A_1\oplus C_3\otimes A_2 \to C_2\otimes A_1 \to C_3\otimes A_1$. There is also a commuting diagram
\begin{equation*}\begin{tikzcd}[cramped]
0\ar{r}& C_1\otimes A_1\oplus C_3\otimes A_3 \ar{r} \ar[equal]{d} &{\rm ker }( d_{-1}) \ar{r}\ar{d}{\iso} &C_3\otimes A_1 \ar{r}\ar[equal]{d} & 0\\
0\ar{r}& C_1\otimes A_1\oplus C_3\otimes A_3 \ar{r}\ar{d}{\fkB_1 + \fkB_3} &\left(C_1\otimes A_1\oplus C_2\otimes A_2\oplus C_3\otimes A_3\right) / {\rm im}(d_0)\ar{r}\ar{d}{\fkB_1 + \fkB_2 + \fkB_3} & C_3\otimes A_1 \ar{r}\ar{d} & 0\\
0\ar{r}&B\ar[equal]{r} & B \ar{r} & 0 \ar{r} & 0
\end{tikzcd}\end{equation*}
where the bottom left square commutes because of the compatibility of the pairings with the maps in the short exact sequences. Applying $H^\bullet(G,-)$ to the lower two rows gives us a commuting diagram
\begin{equation}\label{Diag37}
\begin{tikzcd} H^{r+s}(G,C_3\otimes A_1)\ar{r}{\delta}\ar{d} & H^{r+s+1}(G,C_1\otimes A_1\oplus C_3\otimes A_3) \ar{d}{\fkB_1\oplus\fkB_3}  \\
H^{r+s}(G, 0)\ar{r}  & H^{r+s+1}(G, B)
\end{tikzcd}
\end{equation}
Finally note that due to the long exact sequences associated to the diagram \eqref{kerd1seq} and additivity of connecting homomorphisms, the top arrow in diagram \eqref{Diag37} is simply the difference (because of the signs in the double complex) between the connecting homomorphisms 
\[H^{r+s}(G,C_3\otimes A_1) \to H^{r+s+1}(G,C_1\otimes A_1) \quad\text{and}\quad H^{r+s}(G,C_3\otimes A_1) \to H^{r+s+1}(G,C_3\otimes A_3) \]
Since the composition of this difference with the map induced by $\fkB_1+\fkB_3$ vanishes, this gives the commuting square 
\[\begin{tikzcd}
H^{r+s}(G, C_3\otimes A_1) \ar{d}{\delta} \ar{r}{\delta} & H^{r+s+1}(G, C_1\otimes A_1)\ar{d}{\fkB_1} \\
H^{r+s+1}(G, C_3\otimes A_3) \ar{r}{\fkB_3} & H^{r+s+1}(G, B)
\end{tikzcd}\]
that we required.
\end{proof}
One particular sequence that we will need later is the following.
\begin{prop}\label{leswithwedge}
Let $J\in\Gmod[\pi]{D}$ be such that $J^\ast$ is finitely generated and free as a \Z[\pi]-module, and let $j\colon H^n(G,\famS; J)\to I_\pi$ be a coclass. Then for any short exact sequence
\[0\to C_1\to C_2\to C_3\to 0 \]
in \Gmod[\pi]{D} the cup product induces a map of long exact sequences 
\[\begin{tikzcd}[cramped]
H^k(G,C_1) \ar{r}\ar{d}{\Upsilon_{j,{\rm ev}_1}} & H^k(G, C_2)\ar{r}\ar{d}{\Upsilon_{j,{\rm ev}_2}} &H^k(G,C_3)\ar{r}\ar{d}{\Upsilon_{j,{\rm ev}_3}} & H^{k+1}(G, C_1)\ar{d}{\Upsilon_{j,{\rm ev}_1}}\\
H^{n-k}(G,C_1^\vee)^\ast \ar{r} & H^{n-k}(G, C_2^\vee)^\ast\ar{r} &H^{n-k}(G,C_3^\vee)^\ast\ar{r}{(-1)^k} & H^{n-k-1}(G, C_1^\vee)^\ast
\end{tikzcd} \]
where $C_i^\vee = \bHom(C_i, J)$ and ${\rm ev}_i\colon C_i\otimes C_i^\vee \to J$ is the evaluation map.
\end{prop}
\begin{proof}
Firstly note that since $J^\ast$ is finitely generated as a \Z[p]-module, $\bHom(M,J)$ is finite for any finite module $M$, and hence $\bHom(-,J)$ is a well-defined functor from \Gmod[p]{D} to itself. Since $J^\ast$ is a free \Z[p]-module, this functor is exact. Hence the lower of the long exact sequences is well-defined. The map of long exact sequences commutes by Theorem \ref{BaseCupProdLES} when the short exact sequence $C_\bullet$ is \Z[p]-split. In this situation we have not just a map of one pair of long exact sequences, but a natural transformation of functors 
\[H^\bullet(G,-)\to H^{n-\bullet}(G,(-)^\vee)^\ast \]
which commutes with connecting homomorphisms for split short exact sequences. Proposition 3.5.2 of \cite{SW00} then guarantees that the diagram of long exact sequences commutes even when the sequence does not split.
\end{proof}

\subsection{The relative cohomology sequence}\label{SubSecRelCohomSeq}
The key example to which we apply the long exact sequence of (adjoint) cup products is the exact sequence for relative cohomology. We will primarily be interested in using cup products in connection with Poincar\'e duality pairs. As we shall see in Section \ref{SecFPPairs} in this case this will force $\famS$ to be a finite family of subgroups, and therefore we will make this simplifying assumption here.

Let $C, A\in \Gmod[\pi]{D}$ and consider the short exact sequences
\begin{eqnarray*}
0\to C\stackrel{\gamma_1}{\longrightarrow} \Hom(\ZGS, C)\stackrel{\gamma_2}{\longrightarrow} \Hom(\Delta, C)\to 0\\
0\to A\stackrel{\alpha_2}{\longrightarrow} \Hom(\ZGS, A)\stackrel{\alpha_1}{\longrightarrow} \Hom(\Delta, A)\to 0
\end{eqnarray*}
arising from applying Hom functors to the short exact sequence
\[ 0\to \Delta\longrightarrow\ZGS\longrightarrow\Z[\pi]\to 0\]
which splits as a sequence of \Z[\pi]-modules. Define pairings 
\begin{eqnarray*}
\fkB_1\colon \Hom(\Delta,C)\otimes A &\to & \Hom(\Delta, C\otimes A) \\
f\otimes a&\mapsto & (\delta\mapsto f(\delta)\otimes a) \\
\fkB_3\colon C\otimes \Hom(\Delta,A) &\to & \Hom(\Delta, C\otimes A) \\
c\otimes f&\mapsto & (\delta\mapsto  c\otimes f(\delta)) 
\end{eqnarray*}
and define 
\[\fkB_2\colon \Hom(\ZGS,C)\otimes \Hom(\ZGS, A) \to \Hom(\Delta, C\otimes A)\]
to be the map
\begin{eqnarray*}
{\rm coev}\colon \Hom(\ZGS,C)\otimes \Hom(\ZGS, A) &\to &  \Hom(\ZGS, C\otimes A)\\
f\otimes g &\mapsto & (xS_i \mapsto f(xS_i)\otimes g(xS_i)) 
\end{eqnarray*}
where $x\in G$ followed by the map
\[\Hom(\ZGS,C\otimes A) \to \Hom(\Delta, C\otimes A)\]
induced by inclusion. Note that the first map is only given by this formula on the (\Z[\pi]-)basis vectors $xS_i$ of \ZGS, where $x\in G$ and is extended to the whole module linearly. In particular a generator $xS_i-S_y$ of $\Delta$ is mapped to
\[xS_i-S_y\mapsto f(xS_i)\otimes g(xS_i)-f(S_y)\otimes g(S_y)\]
One may now readily check that these pairings are compatible with the maps in the short exact sequence. Hence given a coclass \[e\colon H^{n-1}(G,\Hom(\Delta, C\otimes A))\to I_\pi\] we obtain a map of long exact sequences
\begin{equation}\begin{tikzcd}[]\label{LESRelCohCup1}
\cdots \ar{r}& H^k(G,C) \ar{r}\ar{d}{\Upsilon_{e,\fkB_1}} & H^k(G,\Hom(\ZGS,C))\ar{r}\ar{d}{\Upsilon_{e,\fkB_2}} &H^{k+1}(G,\famS; C)\ar{d}{\Upsilon_{e,\fkB_3}}\ar{r}&\cdots\\
\cdots \ar{r}& H^{n-k}(G,\famS;A)^\ast \ar{r} & H^{n-k-1}(G,\Hom(\ZGS, A))^\ast\ar{r} &H^{n-k-1}(G,A)^\ast \ar{r}& \cdots
\end{tikzcd}\end{equation} 
Now the middle vertical map is identified with a map
\[\bigoplus_i H^k(S_i, C) \to \bigoplus_i H^{n-k-1}(S_i, A)^\ast \]
which the reader is no doubt expecting to correspond with the cup product on each $S_i$. This follows immediately from the following proposition. When this proposition is used $M$ will be either $\Z[\pi]$ or $\Delta_{S, \cal T}$ for some family of subgroups $\cal T$ of $S$.
\begin{prop}\label{ShapiroAndCupProd}
Let $G$ be a profinite group and $S$ a closed subgroup of $G$. Let $C, A\in\Gmod[\pi]{D}$ and $M\in\mathfrak{C}_\pi(S)$. Let ${\rm coev}$ be the map
\begin{eqnarray*}
{\rm coev}\colon \Hom(\ind^S_G(M),C)\otimes \Hom(\Zpiof{G/S}, A) &\to &  \Hom(\ind^S_G(M), C\otimes A)\\
f\otimes g &\mapsto & (x\otimes m \mapsto f(x\otimes m)\otimes g(xS)) 
\end{eqnarray*}
where $x\in G$ and $m\in M$. Let $\tau_r$ be the Shapiro isomorphism 
\[\tau_r\colon H^r(S,-) \stackrel{\iso}{\longrightarrow} H^r(G, \coind^S_G(-))\]
and recall from Proposition \ref{IndAndCoindAsGMods} that there is a natural identification of $G$-modules
\[\sigma\colon  \coind^S_G(\Hom(M,C)) \iso \Hom(\ind^S_G(M), C)\]
Then there is a commuting diagram
\[\begin{tikzcd}[column sep = large]
H^r(S, \Hom(M,C)) \otimes H^s(S, A) \ar{r}{(\sigma\tau_r)\otimes (\sigma\tau_s)}[swap]{\iso} \ar{d}{\smile_{\id_{C\otimes A}}} & H^r(G,\Hom(\ind^S_G(M),C)) \otimes H^s(G,\Hom(\Zpiof{G/S},A))\ar{d}{\smile_{\rm coev}} \\
H^{r+s}(S, \Hom(M,C\otimes A))\ar{r}{\sigma\tau_{r+s}}[swap]{\iso} & H^{r+s}(G,\Hom(\ind^S_G(M),C\otimes A))
\end{tikzcd}\]
In particular if $B\in \Gmod[\pi]{D}$ and $\fkB$ is any pairing 
\[\Hom(\ind^S_G(M),C)\otimes \Hom(\Zpiof{G/S}, A) \to B \]
which factors as $h\circ{\rm coev}$ where $h$ is a module map 
\[h\colon \Hom(\ind^S_G(M), C\otimes A) \to B \]
and $e\colon H^n(G,B)\to I_\pi$ is a coclass then there is a commuting diagram of adjoint cup products
\[\begin{tikzcd}
H^r(S, \Hom(M,C)) \ar{r}{\sigma\tau_r}[swap]{\iso} \ar{d}{\Upsilon_{\bdy e, \id_{C\otimes A}}} & H^r(G, \Hom(\ind^S_G(M), C) \ar{d}{\Upsilon_{e, \fkB}} \\
H^{n-r}(S, A)^\ast \ar{r}{\sigma\tau_{n-r}}[swap]{\iso} & H^{n-r}(G, \Hom(\Zpof{G/S}, A)) ^\ast
\end{tikzcd} \]
where $\bdy e$ is the coclass defined by 
\[H^n(S, \Hom(M,C\otimes A))\stackrel{\sigma\tau_{n}}{\longrightarrow} H^n(G, \Hom(\ind^S_G(M), C\otimes A) \xrightarrow{e\circ\fkB} I_\pi \]
\end{prop}
\begin{proof}
Let $P_\bullet\to \Z[\pi]$ be a projective resolution of \Z[\pi] by left $G$-modules. From the proofs in Section \ref{SecUsefulIds}, the Shapiro isomorphism is given by the map induced on cohomology by the chain isomorphism
\[\tau_r\colon \Hom_{S_i}(P_r, E)\to  \Hom_G(P_r, \Hom_S(\Zpiof{G},E)) \]
given by 
\[ \tau_r(f)(p)(x) = f(xp) \qquad(f\in \Hom_{S_i}(P_r, C), p\in P_r, x\in G)\] 
where $E\in\mathfrak{D}_\pi(S)$. Furthermore $\sigma$ is given by
\[\sigma(f)(x\otimes m) = xf(x^{-1})(m) \]
for $f\in \Hom_S(\Zpiof{G},\Hom(M, C)), x\in G$ and $m\in M$. One may now readily check that the first diagram in the statement commutes at the level of cochains. The second statement follows immediately.
\end{proof}
Thus we have the familiar chain of adjoint cup products for relative cohomology.
\begin{prop}\label{LESRelCohCup2}
Let $C, A\in \Gmod[\pi]{D}$ and let $e\colon H^{n-1}(G,\famS;C\otimes A)\to I_\pi$ be a coclass. Then there is a commutative diagram of long exact sequences and adjoint cup product maps
\[\begin{tikzcd}[cramped]
\cdots H^k(G,C) \ar{r}\ar{d}{\Upsilon_{e}} & H^k(\famS,C)\ar{r}\ar{d}{\Upsilon_{\bdy e}} &H^{k+1}(G,\famS; C)\ar{d}{\Upsilon_{e}}\ar{r}&H^{k+1}(G,C) \ar{d}{\Upsilon_{e}}\cdots\\
\cdots H^{n-k}(G,\famS;A)^\ast \ar{r} & H^{n-k-1}(\famS, A)^\ast\ar{r} &H^{n-k-1}(G,A)^\ast \ar{r}{(-1)^k}& H^{n-k-1}(G,\famS;A)^\ast \cdots
\end{tikzcd}\]
where $\bdy e=\bdy_i e\colon \bigoplus_i H^{n-1}(S_i; C\otimes A)\to I_\pi$. 

Let $(\phi,f)\colon(H,{\cal T})\to (G,\famS)$ be a map of group pairs, let $c',A'\in \mathfrak{D}_\pi(H)$ and let $e_H\colon H^{n-1}(H, {\cal T}; C'\otimes A')\to I_\pi$ be a coclass. Assume that the coclass $e$ on $H^{n-1}(G,\famS)$ is the composition
\[H^{n-1}(G,\famS;C\otimes A) \longrightarrow  H^{n-1}(H, {\cal T}; C'\otimes A')\stackrel{e_H}{\longrightarrow} I_\pi\] 
Then the commuting diagram above is natural with respect to $(\phi,f)$ in the obvious sense.
\end{prop}
The final naturality statement follows from fact that these sequences are induced from applying functors to coefficient sequences, together with the naturality statements in Propositions \ref{GrpPairNatural} and \ref{GrpPairCupProd1}.

We also note that putting this result together together with the commuting square \eqref{CupProductSquare} one acquires a commuting pentagon for each $i$
\begin{equation}\label{CupProdPentagon}
\begin{tikzcd}[column sep = small]
H^r(S_i, C)\otimes H^s(G,A)\ar{rr} \ar{d} && H^r(S_i, C)\otimes H^s(S_i, A) \ar{d}{\smile_{S_i}} \\ 
H^{r+1}(G,\famS; C)\otimes H^s(G,A) \ar{rd}{\smile_G} && H^{r+s}(S_i, C\otimes A) \ar{ld}\\
& H^{r+s+1}(G,\famS; C\otimes A) &
\end{tikzcd}
\end{equation}

\section{Actions on pro-$\pi$ trees}\label{SecTrees}
\subsection{Trees and graphs of groups}
Let $\cal C$ be an variety of finite groups closed under taking isomorphisms, subgroups, quotients and extensions. Let $\pi(\cal C)$ be the set of primes which divide the order of some finite groups in $\cal C$. A {\em pro-$\cal C$ group} is a profinite group which is an inverse limit of groups in $\cal C$. 

The full development of the theory of profinite graphs and trees is well beyond the scope of this paper. The material here is mostly to be found in \cite{ZM89}. The full theory of profinite trees may be found in \cite{Ribes17}, or distributed around various papers in the literature, mainly by Gildenhuys, Ribes, Zalesskii and Mel'nikov. The theory for pro-$p$ groups is given in \cite{RZ00p}. We will adopt the following definition.
\begin{defn}
An {\em abstract graph} $T$ is a set with a distinguished subset $V(T)$ and two retractions $d_0,d_1\colon T\to V(T)$. Elements of $V(T)$ are called vertices, and elements of $E(T)=T\smallsetminus V(T)$ are called edges. Note that a graph comes with an orientation on each edge.

If an abstract graph is in addition a profinite space (that is, an inverse limit of finite discrete topological spaces), $V(T)$ and $E(T)$ are closed and $d_0,d_1$ are in addition continuous, then $T$ is called a {\em profinite graph}. 
 A morphism of profinite graphs $T$, $T'$ is a continuous function $f\colon T\to T'$ such that $d_i f = fd_i$ for each $i$. An action of a profinite group on a graph is a continuous action by graph morphisms. 
 For a set of primes $\pi$, a profinite graph is a {\em pro-$\pi$ tree} if the chain complex 
\begin{equation}\label{TreeDef}
 0\to \Zpiof{E(T)}\stackrel{d_1-d_0}{\longrightarrow} \Zpiof{V(T)}\stackrel{\epsilon}{\longrightarrow}\Z[\pi]\to 0 
\end{equation}
is exact, where $\epsilon$ is the augmentation.
\end{defn}
\begin{rmk}
The most general definition of a profinite graph does not require that $E(T)$ is closed. Our restricted definition simplifies the exposition, but does not materially alter the results we will state. Moreover the cases that usually arise in applications have $E(T)$ closed. Therefore we will make this simplification. To develop the theory with $E(T)$ not closed one must work with pointed profinite spaces and the free modules over them; specifically in various places $E(T)$ must be replaced with the pointed profinite space $(T/V(T), V(T)/V(T))$.
\end{rmk}

The theory of profinite graphs of groups can be defined over general profinite graphs; we shall only consider finite graphs here as this considerably simplifies the theory and is sufficient for our needs. 
\begin{defn}
A {\em finite graph of pro-${\cal C}$ groups} ${\cal G}=(Y, G_\bullet)$ consists of a connected finite graph $Y$, a pro-${\cal C}$ group $G_y$ for each $y\in Y$, and (continuous) monomorphisms $\bdy_i\colon G_y\to G_{d_i(y)}$ for $i=0,1$ which are the identity when $y\in V(Y)$.
\end{defn}
\begin{defn}
Given a finite graph of pro-${\cal C}$ groups $(Y, G_\bullet)$, choose a maximal subtree $Y_0$ of $Y$. A {\em pro-${\cal C}$ fundamental group} of the graph of groups with respect to $Y_0$ consists of a pro-$\pi$ group $\Delta$, and a map
\[ \phi\colon \coprod_{y\in Y} G_y \amalg \coprod_{e\in E(Y)} \overline{\gp{t_e}} \to H \]
such that    
\[ \phi(t_e) = 1 \text{ for all }e\in E(Y_0)\] and 
\[\phi( t_e^{-1} \bdy_0(g) t_e ) = \phi(\bdy_1(g)) \text{ for all } e\in E(Y), g\in G_e\]
and with $(H,\phi)$ universal with these properties, within the category of pro-${\cal C}$ groups. The pro-${\cal C}$ group $H$ will be denoted $\Pi_1(\cal G)$ or $\Pi_1(Y,G_\bullet)$. Here $\coprod$ denotes the free profinite product; see \cite{RZ00}, Chapter 9.
\end{defn}
The group so defined exists and is independent of the maximal subtree $Y_0$ (see Section 3 of \cite{ZM89}, Section 6.2 of \cite{Ribes17}). Note that in the category of discrete groups this universal property is precisely the same as the classical definition of $\pi_1\cal G$ as group with a certain presentation. Also notice that free products are a special case of a graph of groups in which all edge groups are trivial. 

We use the notation $G=G_1\amalg_L G_2$ for a pro-${\cal C}$ amalgamated free product, i.e.\ the fundamental pro-${\cal C}$ group of a graph of groups with two vertex groups $G_1$ and $G_2$ and one edge group $L$, which is a common subgroup of the two vertex groups. We by convention orient the edge from $G_1$ to $G_2$. 

If we have a pro-${\cal C}$ group $G_1$ and two subgroups $L$ and $L'$ which are isomorphic via an isomorphism $\tau$ then we denote by $G_1\amalg_{L,\tau}$ the pro-${\cal C}$ HNN extension; that is, the fundamental pro-${\cal C}$ group of the graph of groups with vertex group $G_1$, edge group $L$ and monomorphisms \[\bdy_0=\id\colon L\to L\subseteq G_1\text{ and }\bdy_1=\tau\colon L\to L'\subseteq G_1\]
\begin{rmk}
The above notion of fundamental group does of course depend on the variety $\cal C$. For instance a pro-$p$ amalgamated free product of pro-$p$ groups is very different from the profinite amalgamated free product of those same groups viewed as profinite groups. One could introduce $\pi$ into the notation, but this would clutter it rather. All the theorems in this paper will make clear which is meant, although it is likely that the context would be sufficient to tell which category is in use.
\end{rmk}
In the classical Bass-Serre theory, a graph of discrete groups $(Y,G_\bullet)$ gives rise to a fundamental group $\pi_1(Y,G_\bullet)$ and an action on a certain tree $T$ whose vertices are cosets of the images $\phi(G_v)$ of the vertex groups in $\pi_1(Y,G_\bullet)$ and whose edge groups are cosets of the edge groups. Putting a suitable topology and graph structure on the corresponding objects in the profinite world and proving that the result is a profinite tree, is rather more involved than the classical theory; however the conclusion is much the same. We collate the various results into the following theorem. 
\begin{theorem}[Proposition 3.8 of \cite{ZM89}, Section 6.3 of \cite{Ribes17}]
Let ${\cal G}=(Y,G_\bullet)$ be a finite graph of pro-${\cal C}$ groups. Let $\Pi = \Pi_1(\cal G)$ and set $\Pi(y) = {\rm im}(G_y\to \Pi)$. Then there exists an (essentially unique) pro-$\pi({\cal C})$ tree $T(\cal G)$, called the {\em standard tree} of $\cal G$, on which $\Pi$ acts with the following properties. 
\begin{itemize}
\item The quotient graph $\Pi\backslash T({\cal G})$ is isomorphic to $Y$.
\item The stabiliser of a point $t\in T(\cal G)$ is a conjugate of $\Pi(\zeta(t))$ in $\Pi$, where $\zeta\colon T({\cal G})\to Y$ is the quotient map.
\end{itemize}
\end{theorem}
In fact $T$ satisfies a stronger property, of being ${\cal C}(\pi)$-simply connected. Conversely (see Section 6.6 of \cite{Ribes17}) an action of a profinite group on a ${\cal C}(\pi)$-simply connected profinite tree with quotient a finite graph gives rise to a decomposition as a finite graph of profinite groups. In particular open subgroups of fundamental groups of graphs of groups, which act on the standard tree, are themselves fundamental groups of graphs of groups formed in a way closely analogous to the discrete theory. However no analogous results hold when the quotient graph is infinite. 

In the classical theory one tacitly identifies each $G_y$ with its image in the fundamental group $\pi_1(Y,G_\bullet)$ of a graph of groups. In general in the world of profinite groups the maps $\phi_y\colon G_y\to \Pi_1(\cal G)$ may not be injective, even for simple cases such as amalgamated free products. We call a graph of groups {\em proper} if all the maps $\phi_y$ are in fact injections. 

\subsection{Relative homology sequence of an action on a tree}
Let $G$ be a profinite group and let $\famS=\{S_x\mid x\in X\}$ be a family of subgroups of $G$ continuously indexed by a profinite space $X$, and let $\pi$ be a set of primes. Suppose $G$ acts on a pro-$\pi$ tree $T$ on the right. Let $\eta\colon T\to T/G$ be the quotient map. Note that by Proposition \ref{DirSumsAndFreeMods} we have natural identifications 
\[\Zpiof{V(T)} = \bigboxplus_{\bar v\in V(T)/G} \Zpiof{\eta^{-1}(\bar v)}, \quad \Zpiof{E(T)} = \bigboxplus_{\bar e\in E(T)/G} \Zpiof{\eta^{-1}(\bar e)}\]
Hence applying the functor $\Tor^G_\bullet(-, \Delta_{G,\famS}\hotimes M)$ to the short exact sequence \eqref{TreeDef}, where $M\in\Gmod[{\pi}]{C}$, and applying Proposition \ref{DirectSumsAndTor} gives a long exact sequence
\begin{multline*} 
\cdots\to \bigboxplus_{\bar e\in E(T)/G}\Tor^G_n(\Zpiof{\eta^{-1}(\bar e)}, \Delta_{G,\famS}\hotimes M) \to \bigboxplus_{\bar v\in V(T)/G}\Tor^G_n(\Zpiof{\eta^{-1}(\bar v)}, \Delta_{G,\famS}\hotimes M) \\ \to \Tor^G_n(\Z[\pi], \Delta_{G,\famS}\hotimes M)\to\bigboxplus_{\bar e\in E(T)/G}\Tor^G_{n-1}(\Zpiof{\eta^{-1}(\bar e)}, \Delta_{G,\famS}\hotimes M)\to \cdots
\end{multline*} 
Now for each $\bar v\in V(T)/G$, given a choice of lift to $v\in V(T)$ the map $g\mapsto vg$ gives an identification 
\[\Zpiof{\eta^{-1}(\bar v)}]\iso \Zpiof{G_v\lqt G} \]
where $G_v$ is the stabiliser of $v$; hence using Proposition \ref{Shapiro2} and Proposition \ref{SubgpPairDelta} we have identifications
\begin{equation}\Tor^G_n(\Zpiof{\eta^{-1}(\bar v)}, \Delta_{G,\famS}\hotimes M) \iso \Tor^{G_v}_n(\Z[\pi], \Delta_{G_v,\famS^{G_v}}\hotimes M) = H_{n+1}(G_v, \famS^{G_v}; M) \label{TreeIdents}\end{equation} 
Similarly for $\bar e\in E(T)/G$. This immediately yields:
\begin{prop}\label{TreeActionHD}
Let $G$ be a profinite group acting on a pro-$\pi$ tree from the right, and let $\famS=\{S_x\mid x\in X\}$ be a family of subgroups of $G$ continuously indexed by a profinite space $X$. Let $G_t$ denote the stabiliser of $t\in T$. Then 
\[\hd_p(G,\famS)\leq \max_{v\in V(T), e\in E(T)}\big\{\hd_p(G_v, \famS^{G_v}), \hd_p(G_e, \famS^{G_e})+1 \big\}\]
for every $p\in\pi$.
\end{prop} 
\begin{proof}
If the right hand side is finite, denote it by $n$. If it is infinite we have nothing to prove. Then for any $M\in\Gmod[p]{C}$ and any $k> n$ we have, for every $v\in V(T)$ and $e\in E(T)$:
\[H_{k}(G_v, \famS^{G_v}; M)=0, \quad H_{k-1}(G_e, \famS^{G_e}; M)=0 \]
Since an arbitrary profinite sum of zero modules is zero by Lemma \ref{LemSheafOfTrivs}, the long exact sequence above immediately forces $H_k(G,\famS; M)=0$.
\end{proof}
The reader may be wondering why we have not simply applied the identifications \eqref{TreeIdents} to the long exact sequence and written for example $\bigboxplus H_n(G_v, \famS^{G_v})$. The answer is that the latter groups do not form a particularly well-defined sheaf in the greatest generality. More precisely the identifications in \eqref{TreeIdents} required a choice of lift $v\in V(T)$, so that the subgroups $G_v$ may not be continuously indexed by $V(T)/G$. This was not an issue for the above proposition since there is no difficulty manipulating a sheaf all of whose fibres are the zero module. 

If one has a continuous section $\sigma\colon T/G\to T$ of $\eta$ then one can indeed make the required identifications in a continuous manner and recover the expected long exact sequence. 
\begin{multline*}
\cdots\to \bigboxplus_{\bar e\in E(T)/G}  H_{n+1}(G_{\sigma({\bar e})}, \famS^{G_{\sigma({\bar e})}}; M) \to \bigboxplus_{\bar v\in V(T)/G}  H_{n+1}(G_{\sigma({\bar e})}, \famS^{G_{\sigma({\bar v})}}; M) \\ \to H_{n+1}(G, \famS; M)\to\bigboxplus_{\bar e\in E(T)/G}  H_{n}(G_{\sigma({\bar e})}, \famS^{G_{\sigma({\bar e})}}; M)\to \cdots
\end{multline*}
In particular if $T/G$ is finite such a section exists.
\subsection{Mayer-Vietoris sequences and excision}
Let $\cal C$ be an variety of finite groups closed under taking isomorphisms, subgroups, quotients and extensions. Let $\pi=\pi(\cal C)$ be the set of primes which divide the order of some finite groups in $\cal C$.

In this section we will record several long exact sequences associated with injective graphs of profinite groups. Given the set-up in the previous sections, these derivations are generally similar those in Section 3 of \cite{BE77} and we shall not reproduce them all here. An exception is Theorem \ref{RibesDirectSumThm} which is a theorem (though not proof) related to Proposition 2.3 of \cite{Ribes69} and does not appear in \cite{BE77}.  

In all cases the plan is much the same. In the case of a proper pro-${\cal C}$ amalgamated free product $G=G_1\amalg_L G_2$ one starts from the short exact sequence
\[
0\to \Zpiof{G/L} \xrightarrow{(-\res, \res)} \Zpiof{G/G_1}\oplus \Zpiof{G/G_2} \to \Z[\pi] \to 0 
\] 
arising from the last section (using a left action on the tree this time). Having decided upon a theorem to prove, one finds a suitable short exact sequence of \Z[\pi]-free modules, applies an appropriate Ext or Tor functor, and uses Shapiro isomorphisms to translate between the cohomology of the various groups involved. 

Suppose we have a proper pro-${\cal C}$ HNN extension $G=G_1\amalg_{L,\tau}$ where $L$ is a subgroup of $G$ and $\tau\colon L\to L'$ is an isomorphism to another subgroup $L'$ of $G$. In this case one takes as a starting point the sequence   
\[
0\to \Zpiof{G/L} \xrightarrow{\res\tau_\ast-\res} \Zpiof{G/G_1} \to \Z[\pi] \to 0
\] 
In these exact sequences `$\res$' denotes maps like 
\[\Zpiof{G/L}\to \Zpiof{G/G_1}, \quad gL\mapsto gG_1\]
and $\tau_\ast$ denotes the map
\[\Zpiof{G/L}\to \Zpiof{G/L'}, \quad gL\mapsto t^{-1}gtL \]
where $t$ is some stable letter for the HNN extension. One can check that when this last map is translated via a Shapiro isomorphism into a map on the (co)homologies of $L$ and $L'$ it does in fact agree with the map induced functorially from $\tau$.
\begin{theorem}[Theorem 3.2 of \cite{BE77}]\label{FirstMV}
Let $G=G_1\amalg_L G_2$ be a proper pro-${\cal C}$ amalgamated free product. Let $\famS_i$ be a family of subgroups of $G_i$ continuously indexed by $X_i$ for each $i$, where $X_i$ may possibly be empty. Let $\famS$ be the family of subgroups $\famS_1\sqcup \famS_2$ be the family of subgroups of $G$ continuously indexed by $X_1\sqcup X_2$. Then there is a natural long exact sequence
\[\cdots H^{k-1}(L)\to H^k(G,\famS) \xrightarrow{(\res, \res)} H^k(G_1, \famS_1)\oplus H^k(G_2, \famS_2) \xrightarrow{(-\res)\oplus (\res)} H^k(L)\cdots \]
with coefficients in an arbitrary $A\in\Gmod[\pi]{D}$. Here relative cohomology with respect to an empty family should be interpreted as absolute cohomology. Similarly for homology.
\end{theorem}
\begin{proof}
In the case when $\famS_1$ and $\famS_2$ are non-empty this derives from the commuting diagram of short exact sequences 
\begin{equation*}\begin{tikzcd}[cramped]
\ind^{G_1}_G(\Delta_{G_1, \famS_1})\oplus \ind^{G_2}_G(\Delta_{G_2, \famS_2})\ar[equal]{d}\ar[hookrightarrow]{r} & \Delta_{G,\famS} \ar[hookrightarrow]{d}\ar[twoheadrightarrow]{r} & \Zpiof{G/L} \ar[hookrightarrow]{d}\\
\Zpiof{G}\hotimes[G_1]\Delta_{G_1, \famS_1}\oplus \Zpiof{G}\hotimes[G_2]\Delta_{G_2, \famS_2}\ar[hookrightarrow]{r} & \Zpiof{G/\famS_1} \oplus \Zpiof{G/\famS_2} \ar[twoheadrightarrow]{r}\ar[twoheadrightarrow]{d} & \Zpiof{G/G_1} \oplus \Zpiof{G/G_2}\ar[twoheadrightarrow]{d}\\
& \Z[\pi] \ar[equal]{r} & \Z[\pi]
\end{tikzcd} \end{equation*} 
For the other cases see \cite{BE77}.
\end{proof}
\begin{theorem}[Theorem 3.3 of \cite{BE77}]\label{FirstMVHNN}
Let $(G_1, \famS)$ be a profinite group pair with $G_1$ a pro-${\cal C}$ group and $\famS$ possibly empty. Let $L, L'$ be subgroups of $G_1$ isomorphic via an isomorphism $\tau$. Suppose $G=G_1\amalg_{L,\tau}$ is a proper pro-${\cal C}$ HNN extension. Then one has a natural long exact sequence
\[\cdots H^{k-1}(L)\to H^k(G,\famS) \stackrel{\res}{\longrightarrow} H^k(G_1, \famS_1) \xrightarrow{(-\res)\oplus (\tau^\ast\circ\res')} H^k(L)\cdots  \]
with respect to any $A\in \Gmod[\pi]{D}$. Here $\res'$ is the restriction map $H^k(G_1,\famS_1)\to H^k(L')$. Similarly for homology.
\end{theorem}
\begin{theorem}[``Excision'', Proposition 3.4 of \cite{BE77}]\label{Excision}
(a) Let $G=G_1\amalg_L G_2$ be a proper pro-${\cal C}$ amalgamated free product. Let $\famS_1$ be a family of subgroups of $G_1$ continuously indexed by $X_1$ which may possibly be empty. Then the map of pairs $(G_1, \famS_1\sqcup L)\to (G, \famS_1\sqcup G_2)$ induces isomorphisms
\[H_\bullet(G_1, \famS_1\sqcup L; -)\iso H_\bullet(G, \famS_1\sqcup G_2; -), \quad H^\bullet(G, \famS_1\sqcup G_2; -)\iso H^\bullet(G_1, \famS_1\sqcup L; -) \]

(b) Let $G_1$ be a pro-${\cal C}$ group, and let $\famS_1$ be a family of subgroups of $G_1$ continuously indexed by $X_1$ which may possibly be empty. Let $L, L'$ be subgroups of $G_1$ isomorphic via an isomorphism $\tau$. Suppose $G=G_1\amalg_{L,\tau}$ is a proper pro-${\cal C}$ HNN extension. Then there are natural isomorphisms
\[H_\bullet(G_1, \famS_1\sqcup L\sqcup L'; -)\iso H_\bullet(G, \famS_1\sqcup L; -), \quad H^\bullet(G, \famS_1\sqcup L; -)\iso H^\bullet(G_1, \famS_1\sqcup L\sqcup L'; -)  \]
induced by the obvious maps of pairs.
\end{theorem}
\begin{theorem}\label{RibesDirectSumThm}
Let $G=G_1\amalg_L G_2$ be a proper pro-${\cal C}$ amalgamated free product, and $\famS_i$ a family of subgroups of $G_i$ continuously indexed by a possibly empty profinite set $X_i$ for each $i$. Consider the family of subgroups $\famS = \famS_1\sqcup \{L\}\sqcup\famS_2$  of $G$ continuously indexed by $X_1\sqcup \{\ast\}\sqcup X_2$. Then there are natural isomorphisms
\[H^\ast(G,\famS_1\sqcup L\sqcup \famS_2; A)\iso H^\ast(G_1,\famS_1\sqcup L; A)\oplus H^\ast(G_2,\famS_2\sqcup L; A) \] 
induced by the maps of pairs $(G_i, \famS_i\sqcup L)\to (G,\famS\sqcup L)$.
\end{theorem}
\begin{proof}
Consider the commutative diagram below.
\[\begin{tikzcd}[]
   & \Zpiof{G/L} \ar[equal]{r}\ar[hookrightarrow]{d}{(-1,1)} & \Zpiof{G/L}\ar[hookrightarrow]{d} \\
  \bigoplus_{i=1,2}\Zpiof{G}\hotimes[G_i]\Delta_{G_i, \famS_i\sqcup L} \ar[hookrightarrow]{r}\ar{d} & \Zpiof{G/(\famS_1\sqcup \{L\}\sqcup \{L\}\sqcup\famS_2)}  \ar[twoheadrightarrow]{r}\ar[twoheadrightarrow]{d}{1\oplus 1} & \Zpiof{G/G_1}\oplus\Zpiof{G/G_2}\ar[twoheadrightarrow]{d} \\
  \Delta_{G, \famS_1\sqcup L\sqcup\famS_2} \ar[hookrightarrow]{r} & \Zpiof{G/(\famS_1\sqcup \{L\}\sqcup\famS_2)}  \ar[twoheadrightarrow]{r} & \Z[\pi] 
\end{tikzcd}\]
where the middle column arises from the short exact sequence
\[\begin{tikzcd}[]
  0\ar{r} & \Zpiof{G/L}\ar{r}{(-1,1)} &\Zpiof{G/L}\oplus\Zpiof{G/L}\ar{r}{1\oplus 1} & \Zpiof{G/L}\ar{r} & 0
 \end{tikzcd}\]
 by taking a direct sum with \Zpiof{G/\famS_1\sqcup\famS_2}. The middle row is the result of applying the exact functors $\Zpiof{G}\hotimes[G_i]-$ to the definitions of the $\Delta_{G_i, \famS_i\sqcup L}$. The bottom row and final column are already known to be exact. From this it follows that the leftmost vertical map is an isomorphism, whence the theorem. Note also that this isomorphism, via a Shapiro isomorphism, agrees with the inclusion map defined in Section \ref{SecBasicDefs}.
\end{proof}
\begin{rmk}
Using the theorems above there are several maps one can define from $H^{n-1}(L)$ to $H^n(G,\famS)$. One has the map from Theorem \ref{FirstMV}; one has the map induced by
\[\Delta_{G,\famS_1\sqcup \famS_2}\to \Delta_{\famS_1\sqcup G_2} \iso \ind^{G_1}_G(\Delta_{G_1,\famS_1\sqcup L}) \to \Zpiof{G/L} \]
where the second map is the inverse of the excision isomorphism (Theorem \ref{Excision}); and the map
\[\Delta_{G,\famS_1\sqcup \famS_2} \to \Delta_{G,\famS_1\sqcup L\sqcup \famS_2}\iso \bigoplus_{i=1,2}\ind^{G_i}_G\Delta_{G_i, \famS_i\sqcup L} \to \ind^{G_1}_G(\Delta_{G_1,\famS_1\sqcup L}) \to   \Zpiof{G/L}\]
where the second map is the inverse of the isomorphism in Theorem \ref{RibesDirectSumThm}. One may check by a simple diagram chase that these three maps agree. Denote this map by $\bdy_0$. Similarly for HNN extensions. 
\end{rmk}

The following result appears as Theorems 3.5 and 3.7 of \cite{BE77}. We will give a different proof.
\begin{theorem}\label{MVwithExcision}
Let $G=G_1\amalg_L G_2$ be a proper pro-${\cal C}$ amalgamated free product. Let $\famS_i$ be a family of subgroups of $G_i$ continuously indexed by $X_i$ for each $i$, where $X_i$ may possibly be empty. Let $\famS$ be the family of subgroups $\famS_1\sqcup \famS_2$ be the family of subgroups of $G$ continuously indexed by $X_1\sqcup X_2$. Then there is a natural long exact sequence
\[\cdots\to H^{k-1}(L)\to H^k(G_1, \famS_1\sqcup L)\oplus H^k(G_2, \famS_2\sqcup L) \stackrel{}{\longrightarrow}   H^k(G,\famS)\stackrel{}{\longrightarrow} H^k(L)\to\cdots \]
with coefficients in an arbitrary $A\in\Gmod[\pi]{D}$. Here relative cohomology with respect to an empty family should be interpreted as absolute cohomology. Similarly for homology.

Furthermore this sequence is natural with respect to cup products in the following sense. Let $C, A, B\in \Gmod[\pi]{D}$ and let $\fkB\colon C\otimes A\to B$ be a pairing. Choose a coclass $e\colon H^n(G, \famS; B)\to I_\pi$. Then the following diagram
\begin{equation*}\label{MVCupProdLES}
\begin{tikzcd}[cramped, column sep =small]
\cdots H^{k-1}(L,C)\ar{r}\ar{d} & H^k(G_1, \famS_1\sqcup L;C)\oplus H^k(G_2, \famS_2\sqcup L;C) \ar{r}\ar{d}  & H^k(G,\famS;C)\ar{r}\ar{d} & H^k(L;C)\ar{d}\cdots \\
\cdots H^{n-k}(L, A)^\ast\ar{r} & H^{n-k}(G_1, A)^\ast\oplus H^{n-k}(G_2, A)^\ast\ar{r} & H^{n-k}(G, A)^\ast\ar{r} & H^{n-k-1}(L; A)^\ast\cdots
\end{tikzcd}
\end{equation*}
commutes up to sign---the sign being a $(-1)^{k-1}$ in the third square if $\famS\neq \emptyset$ or in the second square if $\famS=\emptyset$. Here the first vertical map is the cup product map with respect to $\fkB$ and the coclass
\[\bdy e\colon H^{n-1}(L, B)\xrightarrow{-\bdy_0} H^n(G, \famS; B)\to I_\pi  \]
The second vertical map is the sum of the cup products with respect to coclasses
\[H^n(G_i, \famS_i\sqcup L;C)\to H^n(G, \famS; B)\to I_\pi\]
induced from $e$. The third vertical map is the usual cup product with respect to $e$.
\end{theorem}
\begin{proof}
In the case when $\famS$ is not empty the long exact sequence is derived, via Theorem \ref{RibesDirectSumThm}, from the top row of the commuting diagram of exact sequences below.
\[\begin{tikzcd}
\Delta_{G, \famS_1\sqcup\famS_2}\ar[hookrightarrow]{r}\ar[hookrightarrow]{d} & \Delta_{G, \famS_1\sqcup L\sqcup\famS_2} \ar[twoheadrightarrow]{r}\ar[hookrightarrow]{d} & \Zpiof{G/L}\ar[equal]{d} \\
\Zpiof{G/\famS_1}\oplus\Zpiof{G/\famS_2} \ar[hookrightarrow]{r}\ar[twoheadrightarrow]{d} & 
\Zpiof{G/(\famS_1\sqcup \{L\}\sqcup\famS_2)} \ar[twoheadrightarrow]{r}\ar[twoheadrightarrow]{d} &\Zpiof{G/L} \\
\Z[\pi]\ar[equal]{r} & \Z[\pi] & 
\end{tikzcd}\]
When $\famS=\emptyset$ the long exact sequence is derived from the long exact sequence for the pair $(G, \{L\})$ via Theorem \ref{RibesDirectSumThm}.

Next we derive naturality properties with respect to the cup product. Suppose first that $\famS\neq\emptyset$. Let $C, A, B\in \Gmod[\pi]{D}$ and let $\fkB\colon C\otimes A\to B$ be a pairing. Choose a coclass $e\colon H^n(G, \famS; B)\to I_\pi$.  We have (\Z[\pi]-split) exact sequences 
\[\begin{tikzcd}0\ar{r} & \Delta_{G, \famS_1\sqcup\famS_2}\ar{r}& \Delta_{G, \famS_1\sqcup L\sqcup\famS_2} \ar{r} & \Zpiof{G/L} \ar{r} & 0\end{tikzcd}\]
and 
\[0\to \Zpiof{G/L} \to \Zpiof{G/G_1}\oplus \Zpiof{G/G_2} \to \Z[\pi] \to 0 \] 
We may apply $\Hom(-,C)$ and $\Hom(-,A)$ and define the following pairings on the groups involved. 
\[\fkB_1\colon \Hom(\Zpiof{G/L}, C)\otimes\Hom(\Zpiof{G/L}, A) \to \Hom(\Delta_{G, \famS}, B) \]
is the pairing
\begin{eqnarray*}
{\rm coev}\colon \Hom(\Zpiof{G/L},C)\otimes \Hom(\Zpiof{G/L}, A) &\to &  \Hom(\Zpiof{G/L}, C\otimes A)\\
f\otimes g &\mapsto & (xL \mapsto f(xL)\otimes g(xL)) 
\end{eqnarray*}
followed by the natural map to $\Hom(\Delta_{G, \famS}, B)$ induced by $-\bdy_0$. 
By Proposition \ref{ShapiroAndCupProd}, via the Shapiro isomorphisms the cup product map induced on cohomology is simply the usual adjoint cup product on $H^\bullet(L)$ with co-class $\bdy e=e\circ(-\bdy_0)\colon H^{n-1}(L, B)\to I_\pi$ induced from $e$. As noted in the theorem statement this is the same as a map appearing in Theorem \ref{FirstMV}.

Next we have the pairing
\begin{multline*}\fkB_2\colon \Hom(\Delta_{G,\famS_1\sqcup L\sqcup \famS_2},C)\otimes \Hom(\Zpiof{G/G_1} \oplus\Zpiof{G/G_2}, A)\iso \\ \Hom(\ind^{G_1}_G(\Delta_{G_1,\famS_1\sqcup L}) \oplus \ind^{G_2}_G(\Delta_{G_2,\famS_2\sqcup L}),C)\otimes \Hom(\Zpiof{G/G_1}\oplus\Zpiof{G/G_2}, A) \\ \xrightarrow{\rm coev} \Hom(\Zpiof{G/G_1}\oplus\Zpiof{G/G_2}, C\otimes A)\to \Hom(\Delta_{G,\famS}, B)
\end{multline*}
where the first isomorphism is the inverse of the excision isomorphism in Theorem \ref{RibesDirectSumThm} and the second map is the direct sum of two coevaluation maps as defined in Proposition \ref{ShapiroAndCupProd}. Again via Proposition \ref{ShapiroAndCupProd} these maps agree with the standard cup product maps on the pairs $(G_i, \famS_i)$, with coclasses $H^n(G_i, \famS_i; B)\to I_\pi$ induced from the coclass $e$ via a map of pairs and the excision isomorphisms.

Finally,
\[\fkB_3\colon \Hom(\Delta_{G,\famS}, C)\otimes A \to\Hom(\Delta_{G,\famS}, B) \]
is the standard cup product pairing.
 
If one has sufficient tenacity one may check that these pairings are compatible with the maps in the short exact sequence. We can now apply Theorem \ref{cupprodLES} (noting that both short exact sequences terminate in a free module so split as \Z[\pi]-modules) to obtain the long exact sequence  of cup product maps as in the statement of the theorem, which commutes except for a sign $(-1)^{k-1}$ in the third square. 

In the case when $\famS=\emptyset$ the long exact sequence in the theorem is derived from the long exact sequence in cohomology for the group pair $(G, L)$ via Theorem \ref{RibesDirectSumThm}, and the cup product diagram is a translation of Proposition \ref{LESRelCohCup2}. This time however the diagram commutes up to a sign $(-1)^{k-1}$ in the second square, as the `connecting homomorphism' in this case is the second map.  
\end{proof}
Of course there is also a version for HNN extensions.
\begin{theorem}[Theorems 3.6 and 3.8 of \cite{BE77}]\label{HNNMVwithExcision}
Let $(G_1, \famS)$ be a profinite group pair with $G_1$ a pro-${\cal C}$ group and $\famS$ possibly empty. Let $L, L'$ be subgroups of $G_1$ isomorphic via an isomorphism $\tau$. Suppose $G=G_1\amalg_{L,\tau}$ is a proper pro-${\cal C}$ HNN extension. Then there is a natural long exact sequence
\[\cdots \to H^{k-1}(L)\to H^k(G_1, \famS\sqcup L\sqcup L') \stackrel{}{\longrightarrow}   H^k(G,\famS)\stackrel{}{\longrightarrow} H^k(L) \to\cdots \]
with coefficients in an arbitrary $A\in\Gmod[\pi]{D}$. Here relative cohomology with respect to an empty family should be interpreted as absolute cohomology. Similarly for homology.

Furthermore this sequence is natural with respect to cup products in the following sense. Let $C, A, B\in \Gmod[\pi]{D}$ and let $\fkB\colon C\otimes A\to B$ be a pairing. Choose a coclass $e\colon H^n(G, \famS; B)\to I_\pi$. Then the following diagram
\begin{equation*}
\begin{tikzcd}[cramped, column sep =small]
\cdots \ar{r}&H^{k-1}(L,C)\ar{r}\ar{d} & H^k(G_1, \famS\sqcup L\sqcup L';C) \ar{r}\ar{d}  & H^k(G,\famS;C)\ar{r}\ar{d} & H^k(L;C)\ar{d}\ar{r}&\cdots \\
\cdots \ar{r}&H^{n-k}(L, A)^\ast\ar{r} & H^{n-k}(G_1, A)^\ast\ar{r} & H^{n-k}(G, A)^\ast\ar{r} & H^{n-k-1}(L; A)^\ast\ar{r}&\cdots
\end{tikzcd}
\end{equation*}
commutes up to sign---the sign being a $(-1)^{k-1}$ in the third square if $\famS\neq \emptyset$ or in the second square if $\famS=\emptyset$. Here the first vertical map is the cup product map with respect to $\fkB$ and the coclass
\[\bdy e=e\circ(-\bdy_0)\colon H^{n-1}(L, B)\to H^n(G, \famS; B)\to I_\pi  \]
 The second vertical map is the cup product with coclass
\[H^n(G_1, \famS\sqcup L\sqcup L';C)\to H^n(G, \famS; B)\to I_\pi\]
induced from $e$. The third vertical map is the usual cup product with respect to $e$.
\end{theorem}

\section{Poincar\'e Duality Pairs}\label{SecFPPairs}
Let $p$ be a prime.
\subsection{Definitions and basic properties}
Recall that a $G$-module $M\in\Gmod[p]{C}$ is {\em of type \pFP[\infty]} if there exists a projective resolution of $M$ by finitely generated $\Zpof{G}$-modules. In this case both $\Tor^G_\bullet(M^\perp, -)$ and $\Ext^\bullet_G(M,-)$ are functors which take finite modules to finite modules. 
\begin{defn}
Let $G$ be a profinite group and $\famS$ a family of subgroups of $G$ continuously indexed over a set $X$. Then $G$ is {\em of type \pFP[\infty]} if \Z[p] is a module of type FP$_\infty$, and the pair $(G,\famS)$ is {\em of type \pFP[\infty]} if the module $\Delta_{G,\famS}$ is of type FP$_\infty$. 
\end{defn}
When the pair $(G,\famS)$ is of type \pFP[\infty] we may define \Gmod[p]{P}-functors
\[{\bf H}_\bullet(G,\famS; -) = \bTor^G_\bullet(\Delta^\perp, -), \quad  {\bf H}^\bullet(G,\famS; -) = \bExt_G^\bullet(\Delta, -)\]
These functors of course extend the functors we have been working with.
\begin{prop}\label{FPDelta}
Let $(G,\famS)$ be a profinite group pair and consider the following statements.
\begin{enumerate}[(1)]
\item $(G,\famS)$ is of type \pFP[\infty].
\item $\famS$ is a finite collection of subgroups and for each $S_x\in\famS$ the module $\Zpiof{G/S_x}$ is of type FP$_\infty$.
\item $G$ is of type \pFP[\infty].
\end{enumerate}
If (1) holds then (2) and (3) are equivalent. Note that if $G$ has property FIM then the second condition is equivalent to:
\begin{enumerate}
\item[(2')] $\famS$ is a finite collection of subgroups each of type \pFP[\infty].
\end{enumerate}
\end{prop}
\begin{proof}
We will consider the short exact sequence of modules 
\[0\to \Delta_{G,\famS}\to \Zpof{G/\famS} \to \Z[p]\to 0 \]
and note that the three conditions of the proposition are the conditions that the three modules in this sequence have type FP$_\infty$. This is tautologous except for $\Zpof{G/\famS}$. If (2) holds then by Proposition \ref{DirSumsAndFPn} $\Zpof{G/\famS}$ has type FP$_\infty$. On the other hand, if $\Zpof{G/\famS}$ is finitely generated then $\famS$ is a finite family by Proposition \ref{InfGenSheaves} and each summand is of type FP$_\infty$ by Proposition \ref{DirSumsAndFPn}. 

The result now follows from Proposition \ref{SESandFPn}.
\end{proof}

Now when ${\rm cd}_p(G)= n>0$ and $G$ is of type \pFP[\infty], then $H^n(G,-)^\ast$ is a representable functor. More precisely there exists a (non-zero) module $I_p(G)\in \Gmod[p]{D}$ and a map $e\colon H^n(G,I_p(G))\to I_p$ such that for any $M\in \Gmod[p]{F}$ the assignment 
\[ \Big(f\colon M\to I_p(G)\Big) \longmapsto \left(H^n(G,M)\stackrel{f}{\longrightarrow} H^n(G, I_p(G)) \stackrel{e}{\longrightarrow} I_p \right) \]
gives a natural isomorphism of \Gmod[p]{F}-functors
\[\Hom_G(-, I_p(G))\iso H^n(G, -)^\ast \]

The pair $(I_p(G), i)$ is unique up to unique isomorphism. See \cite{Serre13}, Section 3.5, Lemma 6 and Proposition I.17. 

Precisely the same arguments show that when $\Delta_{G, \famS}$ is of type \pFP[\infty]\ and $\cd_p(G,\famS)= n>0$ then there exists a unique {\em dualising module} $I_p(G,\famS)$ and a map $e\colon H^n(G,\famS; I_p(G,\famS))\to I_p$ which again gives an isomorphism of \Gmod[p]{F}-functors
\begin{equation}\label{DualisingModule} 
\Hom_G(-, I_p(G,\famS))\iso H^n(G,\famS; -)^\ast 
\end{equation}
Note that by uniqueness of the dualising module if $I_p(G,\famS)\neq 0$ then $\cd_p(G,\famS)=n$. We will also have cause to consider the {\em compact dualising module} $D_p(G,\famS):= I_p(G,\famS)^\ast$. 

Since $(G,\famS)$ is of type \pFP[\infty], the right hand side is a restriction of a continuous \Gmod[p]{P}-functor. The left hand side commutes with inverse limits in the first variable, so we in fact have isomorphisms of \Gmod[p]{C}-functors
\[\Hom_G(M, I_p(G,\famS))\iso {\bf H}^n(G,\famS; M)^\ast \]
for $M\in\Gmod[p]{C}$.

In particular one has the following of isomorphisms.
\begin{eqnarray*}
{\bf H}^n(G,\famS; \Zpof{G})^\ast &\iso & \Hom_G(\Zpof{G}, I_p(G,\famS))\\
&\iso & \Hom(\Z[p], I_p(G,\famS))= I_p(G,\famS)
\end{eqnarray*}
One may check that these are isomorphisms of $G$-modules when the first two modules are given the actions deriving from the right action of $G$ on $\Zpof{G}$. Hence we have an identification
\begin{equation}\label{HnZpG}
{\bf H}^n(G,\famS; \Zpof{G}) = D_p(G,\famS)
\end{equation}
Finally if $D_p(G,\famS)$ is a finitely generated $G$-module both sides of \eqref{DualisingModule} are co-continuous \Gmod[\pi]{P}-functors agreeing on \Gmod[\pi]{F}\ so this is in fact an isomorphism of \Gmod[\pi]{P}-functors. 

\begin{defn}
The pair $(G,\famS)$ is a {\em duality pair of dimension $n$ at the prime $p$} (or more briefly a {\em D$^n$ pair at $p$}) if $(G, \famS)$ is of type \pFP[\infty]\ and $\cd_p(G,\famS)\leq n$, and 
\[{\bf H}^k(G,\famS; \Zpof{G})= \begin{cases} D_p(G,\famS)\neq 0 & k=n
\\ 0 &  k\neq n
\end{cases}\]
where the compact dualising module $D_p(G,\famS)$ is isomorphic as a \Z[p]-module to a finitely generated free \Z[p]-module. If in addition $D_p(G,\famS)\iso\Z[p]$ as a \Z[p]-module then we say $(G,\famS)$ is a {\em Poincar\'e duality pair of dimension $n$ at the prime $p$}, or more briefly a {\em PD$^n$ pair at $p$}. We refer to the map $i$ from the definition of the dualising module as the {\em fundamental coclass} of the pair---this is the dual notion to the fundamental class in the top-dimensional homology.

For a PD$^n$ pair at the prime $p$ the {\em orientation character} of $(G,\famS)$ is the homomorphism $\chi\colon G\to\Aut(\Z[p])$ given by the action on $D_p(G,\famS)$ and (any) identification $D_p(G,\famS)\iso\Z[p]$. We say $(G,\famS)$ is {\em orientable} if $\im(\chi)$ is trivial, and {\em virtually orientable} if $\im(\chi)$ is finite. 
\end{defn}
The obvious analogous definitions for absolute cohomology give the definition of a profinite duality (or Poincar\'e duality) group.
\begin{prop}\label{FPforPDn}
If $(G,\famS)$ is a PD$^n$ pair at the prime $p$ then $G$ is of type \pFP[\infty]. Hence $\famS$ is a finite collection of subgroups and for each $S\in\famS$ the module $\Zpiof{G/S}$ is of type FP$_\infty$. If in addition $G$ has property FIM then each $S\in S$ has type \pFP[\infty].
\end{prop}
\begin{proof}
Take a resolution $P_\bullet$ of $\Delta$ by finitely generated projective $G$-modules. By \cite{brumer66}, Proposition 3.1, the kernel of $P_{n-1}\to P_{n-2}$ is actually projective (and is finitely generated, being the image of $P_n$) so we may truncate $P_\bullet$ and assume that $P_k=0$ for $k>n$.  Now apply the functor $\bHom_G(-,\Zpof{G})$. One may readily see that $\bHom_G(\Zpof{G}, \Zpof{G})=\Zpof{G}$ so that $\bHom_G(-, \Zpof{G})$ preserves the property of being finitely generated free, hence the property of being finitely generated projective. The complex $\bHom(P_{n-1-\bullet}, \Zpof{G})$ is a complex of finitely generated projective modules. Furthermore the homology of this complex is
\begin{equation*}
H_k(\bHom(P_{n-1-\bullet}, \Zpof{G}))={\bf H}^{n-k}(G, \famS; \Zpof{G}) =\begin{cases}
D_p(G,\famS) & \text{if $k=0$} \\ 0 & \text{if $k\neq 0$}
\end{cases} 
\end{equation*}
So this is a resolution of $D_p(G,\famS)$ by finitely generated free $G$-modules---that is, of \Z[p] with some $G$-action.  If $\chi\colon G\to \Aut(\Z[p])=\Z[p]^{\!\times}$ is the $G$-action, then let $\bar\chi(g) = \chi(g)^{-1}$ be the inverse $G$-action---which is well defined as $\Aut(\Z[p])$ is abelian. For a homomorphism $\rho\colon G\to\Aut\Z[p]$ let $\Z[p](\chi)$ denote \Z[p] with $G$-action given by $\rho$. Then $\Z[p](\chi)\hotimes[] \Z[p](\bar{\chi})$ (with diagonal action) is isomorphic to the trivial module $\Z[p]$. Finally note that, via the isomorphism
\[\Zpof{G}\hotimes[] \Z[p](\bar{\chi})\to\Zpof{G},\quad g\otimes 1\mapsto \bar\chi(g)^{-1}\cdot g \]
we see that $-\hotimes[] \Z[p](\bar{\chi})$ is an exact functor taking finitely generated projectives to finitely generated projectives. Hence $\bHom(P_{n-1-\bullet}, \Zpof{G})\hotimes\Z[p](\bar{\chi})$ is the required resolution of \Z[p] by finitely generated projective modules.

The statement about $\famS$ follows from Proposition \ref{FPDelta}.
\end{proof}

In the next two propositions there are maps of \Gmod[\pi]{P}-functors induced by the cup product. In Section \ref{SecCupProds} we defined cup product maps on the category \Gmod[\pi]{D}. These are extended to the maps in the theorem using Propositions \ref{FptoPp1} and \ref{FptoPp2}. 
\begin{theorem}[cf Theorem 4.4.3 of \cite{SW00}]\label{PDisduality1}
Suppose that $G$ has type \pFP$_\infty$. Let $J\in \Gmod[p]{D}$ be a non-zero module such that $J^\ast$ is finitely generated as a \Z[p]-module, and let \[j\colon H^m(G,J) \to I_p\] be a map. Then the following are equivalent:
\begin{enumerate}[(1)]
\item $G$ is a D$^n$ group at the prime $p$ with dualising module $J$ and fundamental coclass $j$.
\item For every $k$ there is an isomorphism of $\Gmod[p]{P}$-functors (equivalently of \Gmod[p]{F}-functors)
\[\Upsilon \colon {\bf H}^k(G,C) \to {\bf H}^{n-k}(G, \bHom(C, J))^\ast  \] 
induced by the cup product maps on \Gmod[\pi]{D} with coclass 
\[ {\bf H}^n(G, C\bhotimes[] \bHom(C,J))\stackrel{\rm ev}{\longrightarrow} {\bf H}^n(G, J) \stackrel{j}{\longrightarrow} I_p \]
\end{enumerate} 
\end{theorem}
We will not prove this proposition, as the proof is very similar to the proof of the next proposition.
\begin{theorem}\label{PDisduality2}
Suppose that $(G,\famS)$ has type \pFP$_\infty$ and $G$ has type \pFP[\infty]. Let $J\in \Gmod[p]{D}$ be a non-zero module such that $J^\ast$ is finitely generated as a \Z[p]-module, and let \[j\colon H^m(G,\famS; J) \to I_p\] be a map. Then the following are equivalent:
\begin{enumerate}[(1)]
\item $(G,\famS)$ is a D$^n$ pair at the prime $p$ with dualising module $J$ and fundamental coclass $j$.
\item For every $k$ the cup product map
\[\Upsilon \colon {\bf H}^k(G, \famS;C) \to {\bf H}^{n-k}(G, \bHom(C, J))^\ast  \] 
induced by the coclass
\[{\bf H}^n(G, \famS; C\bhotimes[] \bHom(C,J))\stackrel{\rm ev}{\longrightarrow} {\bf H}^n(G,\famS; J) \stackrel{j}{\longrightarrow} I_p \]
is an isomorphism of connected sequences of $\Gmod[p]{P}$-functors, or equivalently of \Gmod[p]{F}-functors. 
\end{enumerate} 
\end{theorem}
\begin{rmk}
Recall from Section \ref{SecModuleCats} that when $J^\ast$ is finitely generated as a \Z[p]-module, $\Hom(C, J)$ is finite for $C\in\Gmod[p]{F}$ and it makes sense to speak of $\bHom(-, J)$ as a co-continuous functor in \Gmod[p]{P}.
\end{rmk}
\begin{proof}
The fact that isomorphisms of \Gmod[p]{F}-functors are equivalent to isomorphisms of \Gmod[p]{P}-functors follows from Propositions \ref{FptoPp1} and \ref{FptoPp2}.

 First suppose that (1) holds. By Propositions \ref{FptoPp1}, \ref{FptoPp2} and  \ref{leswithwedge} the map $\Upsilon$ is a morphism of connected $\Gmod[p]{P}$-functors. For $C\in\Gmod[\pi]{F}$ note that \[\Hom_G(C, J) = \Ext_G^0(C , J) = \Ext_G^0(\Z[p], \Hom(C, J))\] where the last identification holds since $J^\ast$ is free over \Z[p] by assumption. Furthermore by  looking at the definition of cup product we see that under this identification the map 
\[\Hom_G(C, J)\to H^m(G,\famS; C)^\ast \] 
agrees with the cup product map 
\[ H^0(G, \Hom(C,J)) \to H^m(G,\famS; C)^\ast\] 
Then the map $\Upsilon$ of connected sequences of cohomological coeffaceable \Gmod[p]{P}-functors is an isomorphism on finite modules at degree zero. Therefore by Propositions \ref{FptoPp1} and \ref{FptoPp2} the functor $\Upsilon$ is an isomorphism for all $k$.

If (2) is true, then the $k=n$ case of the cup product isomorphism gives the defining property of the dualising module so $J$ is indeed the dualising module for the pair $(G,\famS)$. Now notice that
\begin{eqnarray*}
{\bf H}^{n-1}(G,\famS; \F_p\hotimes[] \Zpof{G}) & \iso & H^1(G, \Hom(\F_p\hotimes[] \Zpof{G}, J))\\
&\iso & H^1(G, \Hom(\Zpof{G}, \Hom(\F_p, J)))\\
&=& H^1(G, \coind^1_G(\F_p\otimes J)) = 0
\end{eqnarray*}
whence consideration of the short exact sequence
\[0\to \Zpof{G}\stackrel{p}{\longrightarrow} \Zpof{G}\to \F_p\hotimes[]\Zpof{G} \to 0  \]
and the isomorphism \eqref{HnZpG} shows that the map $J^\ast\to J^\ast$ given by multiplication by $p$ is an injection. Proposition \ref{TorFreeImpliesFree} implies that $J^\ast$ is actually a free \Z[p]-module. Furthermore we have
\[{\bf H}^{k}(G, \famS; \Zpof{G}) \iso {\bf H}^{n-k}(G, \Hom(\Zpof{G}, J)) = H^{n-k}(G, \coind^1_G(J)) = 0\]
for $k\neq n$. Furthermore the cup product isomorphisms immediately show $\cd_p(G,\famS)\leq n$. So $(G,\famS)$ is a D$^n$ pair as required.
\end{proof}

\begin{prop}\label{PDisduality3}
Suppose that $(G,\famS)$ has type \pFP$_\infty$ and $G$ has type \pFP[\infty]. Let $J\in \Gmod[p]{D}$ be a $G$-module isomorphic to $I_p$ as an abelian group and let \[j\colon H^m(G,\famS; J) \to I_p\] be a map. Then the following are equivalent:
\begin{enumerate}[(1)]
\item $(G,\famS)$ is a PD$^n$ pair at the prime $p$ with dualising module $J$ and fundamental coclass $j$.
\item For every $k$ the cup product map
\[\Upsilon \colon {\bf H}^k(G,C) \to {\bf H}^{m-k}(G, \famS; \bHom(C, J))^\ast  \] 
induced by the coclass
\[{\bf H}^n(G, \famS; C\bhotimes[] \bHom(C,J))\stackrel{\rm ev}{\longrightarrow} {\bf H}^n(G,\famS; J) \stackrel{j}{\longrightarrow} I_p \]
is an isomorphism of connected sequences of $\Gmod[p]{P}$-functors.   
\end{enumerate} 
\end{prop}
\begin{proof}
By Proposition \ref{FPforPDn}, $G$ also has type \pFP[\infty]. Since $J$ is isomorphic to $I_p$ as an abelian group, then a form of `Pontrjagin duality with $G$-action'---that is, 
\[C\iso \bHom(\bHom(C,J),J) \]
as $G$-modules, where Hom groups have diagonal actions---implies that this condition is equivalent to (2) of Proposition \ref{PDisduality2}.
\end{proof}

\begin{clly}
If $(G,\famS)$ is a PD$^n$ pair at the prime $p$ then
\[n=\cd_p(G,\famS)=\cd_p(G)+1 \]
\end{clly}
\begin{proof}
The only non-trivial part is that $\cd_p(G)\leq n-1$, which follows easily from the duality in the last proposition.
\end{proof}

\begin{prop}\label{PeripheralsOfPDPairs}
Suppose that $(G,\famS)$ is a PD$^n$ pair at the prime $p$ with dualising module $J$ and fundamental coclass $j$. Assume that either $S$ is of type \pFP[\infty] for all $S_x\in\famS$ or that $G$ has property FIM. Then each subgroup $S_x\in\famS$ is a PD$^{n-1}$ group at the prime $p$ with dualising module $\res^G_{S_x}(J)$ and fundamental coclass given by 
\[\bdy j \colon H^{n-1}(S_x, J)\to H^{n-1}(\famS, J) \to H^n(G,\famS; J)\stackrel{j}{\longrightarrow} I_p\] 
\end{prop}
\begin{proof}
This follows immediately from Theorems \ref{PDisduality1} and \ref{PDisduality2} and Proposition \ref{PDisduality3} combined with Proposition \ref{LESRelCohCup2} and the Five Lemma. The assumptions in the second sentence of the statement are needed to guarantee that each $S_x$ has type \pFP[\infty]---when $G$ has property FIM this follows from Proposition \ref{FPforPDn}.
\end{proof}
\begin{prop}\label{FISubOfPD}
Let $(G,\famS)$ be a pair of type \pFP[\infty]\ and $\cd_p(G,\famS)=n$. Then for any open subgroup $U$ of $G$, the pair $(U,\famS^U)$ is of type \pFP[\infty]\ and
\[I_p(U,\famS^U) = \res^G_U I_p(G,\famS) \] 
\end{prop}
\begin{proof}
The first statement is Corollary \ref{FPnAndFISubGps}. Now let $M\in \mathfrak{F}_p(U)$. Recall that by Proposition \ref{IndEqualsCoind}, since $U$ is open in $G$ the induced and coinduced modules on $M$ agree. Therefore we have, using the relative Shapiro Lemma, 
\begin{eqnarray*}
H^n(U,\famS^U; M)^\ast & \iso & H^n(G,\famS; \coind^U_G(M))^\ast\\
& \iso & \Hom_G(\coind^U_G(M), I_p(G,\famS))\\
& \iso & \Hom_G(\ind^U_G(M), I_p(G,\famS))\\
& \iso & \Hom_U(M, I_p(G,\famS) )
\end{eqnarray*}
where the last isomorphism follows from Corollary \ref{HomIndAndCoind} and Proposition \ref{preShapiro}. Hence the result. 
\end{proof}
\begin{prop}
Let $(G,\famS)$ be a profinite group pair with $G$ a $p$-torsion-free profinite group and let $U$ be an open subgroup of $G$. Then $(G,\famS)$ is a PD$^n$ pair at the prime $p$ if and only if $(U,\famS^U)$ is a PD$^n$ pair at the prime $p$, and their dualising modules agree as $U$-modules.
\end{prop}
\begin{proof}
The fact that $(G,\famS)$ has type \pFP[\infty] if and only if $(U,\famS^U)$ has type \pFP[\infty] is Proposition \ref{FPnAndFISubGps}. Next we prove that cohomological dimension $n$ if and only if $(U,\famS^U)$ does. Given Proposition \ref{PDisduality2}, if $(U,\famS^U)$ is a PD$^n$ pair at the prime $p$ then 
\[\cd_p(G,\famS)\geq \cd_p(U,\famS^U)=\cd_p(U)+1= \cd_p(G)+1 \geq\cd_p(G,\famS)  \]
where the equality $\cd_p(U)=\cd_p(G)$ derives from Serre's Theorem (Proposition $14'$ of \cite{Serre13}). If $(G,\famS)$ is a PD$^n$ pair at the prime $p$ then $\cd_p(U,\famS^U)=n$ by Lemma \ref{CorisSurj}. 

Finally we have isomorphisms
\[{\bf H}^k(G,\famS; \Zpof{G}) = {\bf H}^k(G,\famS; \ind^U_G(\Zpof{U})) = {\bf H}^k(G,\famS; \bcoind^U_G(\Zpof{U}))={\bf H}^k(U,\famS^U; \Zpof{U}) \]
using Proposition \ref{IndEqualsCoind} and the relative Shapiro Lemma. This concludes the proof.
\end{proof}

\begin{lem}\label{CorisSurj}
If $(G,\famS)$ is a profinite group pair of cohomological dimension $n$ and $U$ is an open subgroup of $G$ then for any $A\in\Gmod[p]{D}$ the corestriction map
\[\cor^U_G\colon H^n(U,\famS^U; A)\to H^n(G,\famS; A)\]
is a surjection. In particular $\cd_p(U,\famS^U)=\cd_p(G,\famS)$.
\end{lem}
\begin{proof}
Consider the map 
\[\Sigma\colon\Hom_{U}(\Zpof{G}, A)\to A\]
which defines the corestriction, which is a surjection. The short exact sequence 
\[0\to \ker(\Sigma)\to \Hom_{U}(\Zpof{G}, A)\to A\to 0  \]
gives an exact sequence
\[ H^{n}(U, \famS^{U}; A)\stackrel{\cor}{\longrightarrow} H^{n}(G, \famS; A) \to H^{n+1}(G, \famS; \ker(\Sigma)) = 0\]
so that the corestriction map is a surjection. This shows that $\cd_p(G,\famS)\leq\cd_p(U,\famS^U)$; the other inequality follows from the relative Shapiro Lemma.
\end{proof}

\begin{prop}\label{InfIndexOfPD}
Suppose that $(G,\famS)$ is a PD$^n$ pair at $p$ and that $H$ is a closed subgroup of $G$ with $p^\infty \mid [G:H]$ (see Section I.1.3 of \cite{Serre13} for the definition of this index). Then 
\[\cd_p(H,\famS^H)<\cd_p(G,\famS) \]
\end{prop}
\begin{proof}
Let $n=\cd_p(G,\famS)$. By Proposition \ref{cdintermsofFp} it suffices to show that for any $H$ with $p^\infty \mid [G:H]$ we have $H^{n}(H,\famS^H; \F_p)=0$. Since $p^\infty \mid [G:H]$ we may find a descending chain $\{U_i\}_{i\geq 0}$ of open subgroups of $G$ with $p\mid [U_i, U_{i+1}]$ for all $i$. Since by Proposition \ref{IntersectionsOfSubGps} we have
\[H^{n}(H,\famS^H; \F_p) \iso \varinjlim H^{n}(U_i, \famS^{U_i}; \F_p) \]
it suffices to show that each restriction map
\[ H^{n}(U_i, \famS^{U_i}; \F_p) \to H^{n}(U_{i+1}, \famS^{U_{i+1}}; \F_p)\]
vanishes for each $i$. Since for all open subgroups $U$ of $G$ we have that $(U,\famS^U)$ are PD$^n$ pairs at the prime $p$ by Proposition \ref{FISubOfPD}, both the domain and codomain of the corestriction are isomorphic to subgroups of $\Hom(\F_p, \Q_p/\Z[p])\iso\F_p$, hence either one of them is zero (so that the restriction map vanishes) or both are isomorphic to $\F_p$. In this last case consider the corestriction map $\cor^{U_{i+1}}_{U_i}$ in the other direction, which is a surjection by Lemma \ref{CorisSurj}. Hence the corestriction is an isomorphism. 

Finally recall that the map 
\[\cor^{U_{i+1}}_{U_i}\circ \res^{U_{i}}_{U_{i+1}} \colon H^{n}(U_i, \famS^{U_i})\to H^{n}(U_i, \famS^{U_i})\]
is multiplication by $[U_i:U_{i+1}]$. Here since $p\mid [U_i:U_{i+1}]$ this map is zero. Since the corestriction is an isomorphism, the restriction map is zero as required.
\end{proof}
\begin{clly}\label{FixedVertex}
Let $(G,\famS)$ be a pro-$\pi$ group pair which is a PD$^n$ pair at every prime $p\in \pi$. Suppose that $G$ acts on a pro-$\pi$ tree $T$. Suppose that for every edge $e$ of $T$ we have $\cd_p(G_e, \famS^{G_e})<n-1$ for all $p\in\pi$, where $G_e$ denotes the stabiliser of $e$. Then $G$ fixes a vertex of $T$.
\end{clly}
\begin{proof}
By Proposition 2.4.12 of \cite{Ribes17} one may pass to a minimal invariant subtree of $T$ so that $G$ acts irreducibly. We aim to prove that there exists $p\in\pi$ such that $p^\infty\mid [G:G_x]$ for all $x\in T$, for then the result now follows from Propositions \ref{TreeActionHD} and \ref{InfIndexOfPD}.

By factoring out the kernel of the action one may assume that $G$ acts faithfully on $T$. Then by Theorem 4.2.10 of \cite{Ribes17} if $G$ does not fix a vertex then either $G$ admits a non-abelian free pro-$p$ subgroup acting freely on $T$---which forces $p^\infty\mid [G:G_x]$ for all $x\in T$---or for some $p\in\pi$, there is an abelian normal subgroup of $G$ isomorphic to $\Z[p]$. By Lemma 4.2.6(c) of \cite{Ribes17} this subgroup acts freely so again $p^\infty\mid [G:G_x]$ for all $x\in T$ as required.
\end{proof}

Finally we conclude with the result that for pro-$p$ groups the property of being a PD$^n$ pair is determined by the behaviour of the cohomology with coefficients in $\F_p$. Recall that for pro-$p$ group pairs having type \pFP[\infty] is equivalent to having finite cohomology by Proposition \ref{FPnForVirtProP}.
\begin{theorem}
Let $(G,\famS)$ be a pro-$p$ group pair with $\cd_p(G,\famS)=n$. Then the following are equivalent. 
\begin{enumerate}[(1)]
\item $(G,\famS)$ is a PD$^n$ pair at the prime $p$.
\item $H^k(G,\famS; \F_p)$ is finite for all $k$, $\dim_{\F_p}(H^n(G,\famS;\F_p)) =1$ and the pairing 
\[H^k(G,\famS;\F_p) \otimes H^{n-k}(G, \F_p)\to H^n(G,\famS;\F_p) \]
induced by the cup product (with respect to the multiplication pairing $\F_p\otimes\F_p\to \F_p$) is non-degenerate for all $1\leq k\leq n$.
\end{enumerate} 
\end{theorem}
\begin{proof}
Assume first that $(G,\famS)$ is a PD$^n$ pair. Then 
\[ H^n(G,\famS; \F_p)\iso \Hom_G(\F_p, I_p(G,\famS))^\ast=\F_p\]
noting that since $G$ is a pro-$p$ group it must act trivially on the submodule $\gp{p^{-1}}$ of $I_p$ under any $G$-action, so the natural map $\F_p\to I_p(G,\famS)$ is $G$-linear. The non-degeneracy follows from Theorem \ref{PDisduality2}, noting that the multiplication pairing coincides with the evaluation pairing $\F_p\otimes \Hom(\F_p, I_p(G,\famS))\to I_p(G,\famS)$ under any identification of $\F_p$ with the submodule of $I_p(G,\famS)$ killed by $p$. 

Now assume that $G$ satisfies (2). Let $\F_p\Gmod[p]{C}$ be the subcategory of modules annihilated by $p$. Consider, for $C\in\F_p\Gmod[p]{C}$, the map 
\[\Upsilon_k\colon {\bf H}^k(G,\famS; C)\to H^{n-k}(G, C^\ast)^\ast  \]
induced by the cup product with respect to the evaluation pairing $C\otimes C^\ast\to \F_p$ and a choice of identification coclass $H^n(G,\famS; \F_p)\to \F_p$. Note that since $C$ is in $\F_p\Gmod[p]{C}$ the evaluation map on $C\otimes C^\ast$ does indeed have image in $\F_p$. By Proposition \ref{leswithwedge} this map gives a map of connected sequences of continuous $\F_p\Gmod[p]{F}$-functors. The unique simple $G$-module in $\F_p\Gmod[p]{F}$ is $C=\F_p$ (\cite{RZ00}, Lemma 7.1.5). Using the long exact sequence and induction starting from $\F_p$ we find that $\Upsilon_k$ is an isomorphism for all $C\in \F_p\Gmod[p]{F}$. Both sides commute with inverse limits, so $\Upsilon_k$ is an isomorphism of $\F_p\Gmod[p]{C}$-functors.

In particular we have
\begin{equation*}{\bf H}^{n-k}(G,\famS; \Fpof{G})\iso H^k(G, \Fpof{G}{}^\ast)^\ast =H_k(G, \Fpof{G})=H_k(G, \ind^1_G(\F_p))=\begin{cases} \F_p & \text{if $k=0$}\\ 0& \text{if $k\neq 0$} \end{cases} 
\end{equation*}
Then considering the long exact sequence associated to the short exact sequence
\[0\to \Zpof{G}\stackrel{p}{\to} \Zpof{G}\to  \Fpof{G}\to 0 \]
and recalling that by \eqref{HnZpG} we have ${\bf H}^n(G,\famS; \Zpof{G})=D_p(G,\famS)$, we find that the sequence
\[0\to D_p(G,\famS)\stackrel{p}{\to} D_p(G,\famS)\to \F_p\to 0 \]
is exact. That is, $D_p(G,\famS)$ is $p$-torsion-free and hence a free abelian pro-$p$ group by Proposition \ref{TorFreeImpliesFree}, and is also of rank 1. So $I_p(G,\famS)$ is isomorphic to $I_p$ as an abelian group. Furthermore the rest of this long exact sequence shows that for every $k\neq n$ the map
\[{\bf H}^k(G,\famS; \Zpof{G}) \stackrel{p}{\to} {\bf H}^k(G,\famS; \Zpof{G})\]
is an isomorphism. For a compact $p$-primary module multiplication by $p$ is only an isomorphism when the module is trivial. This concludes the proof.
\end{proof}

\subsection{Graphs of PD$^n$ pairs}\label{SecGraphsOfPDn}
Let $\cal C$ be an variety of finite groups closed under taking isomorphisms, subgroups, quotients and extensions. Let $\pi(\cal C)$ be the set of primes which divide the order of some finite groups in $\cal C$. 
\begin{theorem}
Let $G=G_1\amalg_L G_2$ be a proper pro-$\cal C$ amalgamated free product. Assume that $L$ does not equal either $G_1$ or $G_2$. Let $\famS_i$ be a (possibly empty) finite family of subgroups of $G_i$. Let $\famS$ be the family of subgroups $\famS_1\sqcup \famS_2$ of $G$, which is continuously indexed by $X_1\sqcup X_2$. Then the following hold.
\begin{enumerate}[(1)]
\item If each pair $(G_i, \famS_i\sqcup L)$ is a PD$^n$ pair at the prime $p$, so is $(G,\famS)$.
\item Suppose $(G,\famS)$ is a PD$^n$ pair at the prime $p$ and $L$ is a PD$^{n-1}$ group at the prime $p$. Assume either that $G_i$ and $(G_i, \famS_i\sqcup L)$ have type \pFP[\infty] for each $i$ or that $G$ has property FIM. Then each pair $(G_i, \famS_i\sqcup L)$ is a PD$^n$ pair at the prime $p$.
\end{enumerate}
Moreover the dualising modules are all isomorphic to appropriate restrictions of $I_p(G,\famS)$. In particular $(G,\famS)$ is orientable if and only if both $(G_i, \famS_i\sqcup L)$ are orientable.
\end{theorem}
\begin{proof}
We first note that in each case all necessary groups and group pairs have type \pFP$_\infty$. In case (1) the conditions we know that $G_1$ and $G_2$ and are of type \pFP$_\infty$ by Proposition \ref{FPforPDn}, and furthermore that $\Zpof{G_i/L}$ has type \pFP[\infty] as a $G_i$-module. The short exact sequence 
\[ 0\to \Zpof{G/L}\to \Zpof{G/G_1}\oplus\Zpof{G/G_2}\to\Z[p]\to 0\] 
combined with Propositions \ref{SESandFPn} and \ref{IndAndFPMods2} now forces $G$ to have type \pFP$_\infty$---note that $\Zpof{G/L}=\ind^{G_1}_G(\Zpof{G_1/L})$. Furthermore applying Propositions \ref{FPforPDn} and \ref{IndAndFPMods2} to the short exact sequence
\begin{equation*}\begin{tikzcd}
0\ar{r}&\ind^{G_1}_G(\Delta_{G_1, \famS_1})\oplus \ind^{G_2}_G(\Delta_{G_2, \famS_2}) \ar[]{r} & \Delta_{G,\famS} \ar[]{r} & \Zpiof{G/L} \ar{r}&0
\end{tikzcd} \end{equation*} 
appearing in Theorem \ref{FirstMV} shows that $(G,\famS)$ also has type \pFP$_\infty$.

In the case (2) we must show that each pair $(G_i, \famS_i\sqcup L)$  and group $G_i$ have type \pFP[\infty], assuming that $G$ has property FIM. We know that $G$ and $L$ and all the finitely many groups in $\famS$ have type \pFP$_\infty$. The short exact sequence 
\begin{equation*}\begin{tikzcd}
0\ar{r}& \Delta_{G, \famS}\ar[]{r}& \Delta_{G, \famS\sqcup L} \ar[]{r} & \Zpiof{G/L}\ar{r} &0
\end{tikzcd} \end{equation*}
from equation \ref{InductionOnPairs} together with Propositions \ref{SESandFPn} and \ref{IndAndFPMods2} forces the pair $(G,\famS\sqcup L)$ to be of type \pFP[\infty]. The isomorphism 
\[  \bigoplus_{i=1,2}\Zpiof{G}\hotimes[G_i]\Delta_{G_i, \famS_i\sqcup L} \iso \Delta_{G, \famS_1\sqcup L\sqcup\famS_2} \]
from Theorem \ref{RibesDirectSumThm} and Proposition \ref{DirSumsAndFPn} show that, given property FIM, the group pairs $(G_i,\famS_i)$ have type \pFP$_\infty$. It now follows that $G_1$ and $G_2$ also have type \pFP[\infty] by Proposition \ref{FPDelta}.

We now move on to the main part of the theorem. First suppose the conditions of (1) hold. By Proposition \ref{PeripheralsOfPDPairs} the group $L$ is a PD$^{n-1}$ group whose dualising module $I_p(L)$ is isomorphic to $\res^{G_i}_L(I_p(G_i, \famS_i\sqcup L))$ for both $i=1,2$, hence the actions $G_i\to \Aut(\Q_p/\Z[p])$ agree on $L$ and by the universal property of amalgamated free products there is a natural $G$-module $J$ with underlying abelian group $\Q_p/\Z[p]$ so that $\res^G_{G_i}(J)=I_p(G_i,\famS_i\sqcup L)$ for each $i$.

Since any fundamental coclass $e_i$ of $(G_i, \famS_i\sqcup L)$ restricts to a fundamental coclass of $L$, choose coclasses $e_i$ which restrict to the same fundamental coclass of $L$. There is also an induced coclass $e\colon H^n(G,\famS; J)\to I_p$ given by 
\[e_2-e_1\colon H^n(G_1, \famS_1; I_p(G_1, \famS_1\sqcup L))\oplus H^n(G_1, \famS_1; I_p(G_2, \famS_2\sqcup L)\to I_p \]
Note that $e$ vanishes on $H^{n-1}(L)$ since the $e_i$ both restrict to a fundamental coclass on $H^{n-1}(L, I_p(L))$. Therefore by the long exact sequence in Theorem \ref{MVwithExcision} we do indeed have a well-defined coclass on $H^n(G,\famS; J)$.

We may now use Theorems \ref{MVwithExcision} and \ref{PDisduality2} and  the Five Lemma applied to the sign-commutative diagram of long exact sequences \eqref{MVCupProdLES} with $A=\Hom(C,J)$ and coclass $e$ to prove that $(G,\famS)$ is a PD$^n$ pair at the prime $p$ with dualising module $J$ as required.

Next assume the conditions of part (2) and assume also that the dualising module for $L$ is in fact $\res^G_L(I_p(G,\famS))$. We will prove this later. Let $I_p(G,\famS)=J$. It follows from the Five Lemma and \eqref{MVCupProdLES} that for each $i$ and $k$ and for any $C\in\Gmod[p]{P}$ the cup product induces an isomorphism
\[\Upsilon \colon {\bf H}^k(G_i, \famS_i;\res^G_{G_i}(C)) \to {\bf H}^{m-k}(G_i, \bHom(\res^G_{G_i}(C), J))^\ast  \]   

Let $G_i\lqt G\to G$ be a continuous section of the quotient map. This yields an identification of $G_i$-modules $\Zpof{G}=\Zpof{G_i}[\![G_i\lqt G]\!]$ and a $G_i$ linear epimorphism:
$q\colon \Zpof{G}=\Zpof{G_i}[\![G_i\lqt G]\!] \to \Zpof{G_i}$
sending each element $G_ig$ of $G_i\lqt G$ to 1. Furthermore we have the natural $G_i$-linear map
\[i\colon \Zpof{G_i} \to \Zpof{G}\]
such that $qi=\id_M$. Now the commuting diagram
\[\begin{tikzcd}
{\bf H}^k(G_i,\famS_i; \Zpof{G_i}) \ar[shift left]{r}{i}\ar{d}{\Upsilon} & {\bf H}^k(G_i,\famS_i; \res^G_{G_i}(\Zpof{G})) \ar[shift left]{l}{q}\ar{d}{\iso}\\
{\bf H}^{n-k}(G_i, \Hom(\Zpof{G_i}, J))^\ast \ar[shift left]{r}{i} & {\bf H}^{n-k}(G_i,\Hom(\res^G_{G_i}(\Zpof{G}), J))^\ast \ar[shift left]{l}{q}
\end{tikzcd} \]
combined with the fact that applying various functors to $q$ and $i$ preserves the property that $qi$ is an identity map shows that the cup product map on the left is in fact an isomorphism. Then the natural map
\[{\bf H}^k(G_i,\famS_i;\Zpof{G_i})\iso {\bf H}^{n-k}(G_i, \coind^1_{G_i}(J))^\ast=\begin{cases} 0 &\text{ if $k\neq n$}\\ \res^G_{G_i}(D_p(G,\famS)) & \text{if $k=n$} \end{cases}\]
is an isomorphism as required. This is an isomorphism of $G_i$-modules by Proposition \ref{ShapiroWithAction}. 

We now prove that $I_p(L)=\res^G_L(I_p(G,\famS))$. Consider the map 
\[\bdy_0\colon H^{n-1}(L, I_p(G,\famS))\to H^n(G,\famS; I_p(G,\famS)) \]
from the long exact sequence in Theorem \ref{FirstMV}. By Poincar\'e duality for $L$ and $(G,\famS)$ both sides are isomorphic to (subgroups of) $\Q_p/\Z[p]$ as abelian groups. Therefore the map $\bdy$ is either an isomorphism or factors through the multiplication-by-$p$ map on $H^{n-1}(L, I_p(G,\famS))$. If $\bdy$ is an isomorphism then since the left hand side is isomorphic to $\Hom_L(I_p(L), I_p(G,\famS))$ and the right hand side is isomorphic to $\Hom_G(I_p(G, \famS), I_p(G,\famS))\iso I_p$ there exists a non zero $L$-linear map between $I_p(L)$ and $I_p(G,\famS)$. Since these are both isomorphic to $\Q_p/\Z[p]$, one may show that there must be an $L$-module isomorphism between the two modules and indeed $I_p(L)=\res^G_L(I_p(G,\famS))$.

We now turn our attention to the other case and show that it is impossible. The fundamental coclass $e$ of $(G,\famS)$ yields a coclass $\bdy e\colon H^{n-1}(L, I_p(G,\famS))\to I_p$ which factors through multiplication by $p$. If $C$ is a finite $G$-module killed by $p$ then the cup product map
\[\Upsilon_{\bdy e, {\rm ev}}\colon H^k(L,C)\otimes H^{n-1-k}(L, \Hom(C, I_p(G,\famS)))\to H^{n-1}(L, I_p(G,\famS))\to I_p \] 
(given by the evaluation pairing and the coclass $\bdy e$) is a map of modules killed by $p$ which factors through multiplication by $p$. That is, it is zero. This cup product is exactly the right-hand map appearing in the following piece of the map of long exact sequences \eqref{MVCupProdLES}:
\[\begin{tikzcd}
H^{n-1}(G,\famS; C)\ar{r}\ar{d}{\iso} & H^{n-1}(L, C) \ar{d}{0}\\
H^1(G, \Hom(C,I_p(G,S))^\ast \ar{r} & H^0(L, \Hom(C, I_p(G,\famS)))^\ast
\end{tikzcd}\]
Therefore the bottom map vanishes. Since $I_p(G,\famS)$ is isomorphic to $I_p$ as an abelian group, by Pontrjagin-duality-with-action the coefficient group $\Hom(C,I_p(G,S))$ of the bottom row ranges over all finite $G$-modules killed by $p$. Therefore, considering the rest of the long exact sequence on the bottom row and using Shapiro isomorphisms, for all finite modules $A$ killed by $p$ the sequence 
\[ \Hom_G(\Z[p], A)\to \Hom_G(\Zpof{G/G_1}\oplus \Zpof{G/G_2}, A) \to \Hom_G(\Zpof{G/L}, A)\to 0  \]
is exact. Taking inverse limits, which is a right-exact functor, and noting that for instance each $\Zpof{G/G_i}$ is finitely generated we obtain an exact sequence
\[ \bHom_G(\Z[p], M)\to \bHom_G(\Zpof{G/G_1}\oplus \Zpof{G/G_2}, M) \to \bHom_G(\Zpof{G/L}, M)\to 0  \]
for compact $G$-modules $M$ which are killed by $p$. In particular the natural quotient map from $\Zpof{G/L}\to \Fpof{G/L}$ extends to a $G$-linear map $\Zpof{G/G_1}\oplus \Zpof{G/G_2}\to \Fpof{G/L}$. That is, the sequence
\[0\to \Fpof{G/L}\to \Fpof{G/G_1}\oplus \Fpof{G/G_2}\to \F_p\to 0 \]
splits. Therefore since $\F_p$ is a trivial module there must be a $G$-invariant point of $\Fpof{G/G_i}$ for each $i$, which is non-zero for at least one $i$. This is impossible by Propositions \ref{NoFpSplittings} and \ref{VxGrpsAreInfIndex}. 
\end{proof}
In the classical case one would simply assert that unless $L=G_{1-i}$, the group $G_i$ has infinite index in $G=G_1\amalg_L G_2$ and therefore the existence of a $G$-invariant point of $\Fpof{G/G_i}$ is absurd. In our case we must work a little harder---for instance, absurd as it is, one may show that $\F_5[\![\Z[3]]\!]$ has a \Z[3]-invariant point. This analysis does not really belong here and we leave it to \ref{AppGofGs}.

Similarly to the above one has the analogous result for HNN extensions, relying upon Theorem \ref{HNNMVwithExcision}.
\begin{theorem}
Let $(G_1, \famS)$ be a profinite group pair with $G_1$ a pro-$\cal C$ group and $\famS$ a possibly empty finite family of subgroups of $G_1$. Let $L, L'$ be subgroups of $G_1$ isomorphic via an isomorphism $\tau$. Suppose $G=G_1\amalg_{L,\tau}$ is a proper pro-$\cal C$ HNN extension. Then the following hold.
\begin{enumerate}[(1)]
\item If $(G_1, \famS_1\sqcup L\sqcup L')$ is a PD$^n$ pair at the prime $p$, so is $(G,\famS)$.
\item Suppose $(G,\famS)$ is a PD$^n$ pair at the prime $p$ and $L$ is a PD$^{n-1}$ group at the prime $p$. Assume either that $G_1$ and $(G_1, \famS\sqcup L\sqcup L')$ have type \pFP[\infty] or that $G$ has property FIM. Then $(G_1, \famS_1\sqcup L\sqcup L')$ is a PD$^n$ pair at the prime $p$.
\end{enumerate}
Moreover the dualising modules are all isomorphic to appropriate restrictions of $I_p(G,\famS)$. In particular $(G,\famS)$ is orientable if and only if $(G, \famS\sqcup L\sqcup L')$ is orientable.
\end{theorem}
By a standard induction one deduces a theorem valid for finite graphs of pro-$\cal C$ groups with more than one edge. By a `reduced' graph of groups below we mean that any edge whose edge group coincides with the adjacent vertex groups is a loop. By collapsing any non-loops with this property one may always make a graph of groups reduced.
\begin{theorem}
Let ${\cal G}=(Y, G_\bullet)$ be a reduced proper finite graph of pro-$\cal C$ groups. For each $y\in V(Y)$ let $\famS_y$ be a possibly empty finite family of subgroups of $G_y$. Let ${\cal E}_y$ denote the set of subgroups of $G_y$ which are images of the edge groups $\bdy_i(G_e)$ for adjacent edges with $d_i(e)=y$. Finally let $G$ denote the fundamental pro-$\cal C$ group of $\cal G$ and let $\famS=\bigsqcup_{y\in V(Y)}\famS_y$. Then the following hold.
\begin{enumerate}[(1)]
\item If each pair $(G_y, \famS_y\sqcup {\cal E}_y)$ is a PD$^n$ pair at the prime $p$, so is $(G,\famS)$.
\item Suppose $(G,\famS)$ is a PD$^n$ pair at the prime $p$ and $G_e$ is a PD$^{n-1}$ group at the prime $p$ for each $e\in E(Y)$. Assume either that $G_y$ and $(G_y, \famS_y\sqcup {\cal E}_y)$ have type \pFP[\infty] for all $y\in V(Y)$ or that $G$ has property FIM. Then each pair $(G_y, \famS_y\sqcup {\cal E}_y)$ is a PD$^n$ pair at the prime $p$.
\end{enumerate}
Moreover the dualising modules are all isomorphic to appropriate restrictions of $I_p(G,\famS)$. In particular $(G,\famS)$ is orientable if and only if $(G_y, \famS_y\sqcup {\cal E}_y)$ is orientable for all $y\in V(Y)$.
\end{theorem}

\section{Relative goodness for discrete groups}\label{SecGoodness}
\subsection{Definition and basic properties}
Let $\Gamma$ be a discrete group and $G$ a profinite group. Let $\phi\colon \Gamma\to G$ be a group homomorphism. Let $M$ be a $\Gamma$-module and $N\in\Gmod[\pi]{C}$, and let $f\colon M\to N$ be a group homomorphism compatible with $\phi$ in the natural way, i.e.\ $ f(\gamma m)=\phi(\gamma)f(m)$ for $\gamma\in\Gamma, m\in M$.  
Let $A$ be a $\Gamma$-module and $B\in \Gmod[\pi]{D}$ and let $g\colon B\to A$ be a group homomorphism compatible with $\phi$ in the sense that $\gamma g(n)=g(\phi(\gamma)n)$ for $\gamma\in\Gamma, n\in N$. 
Take a projective resolution $P_\bullet$ of $M$ by $\Gamma$-modules and a projective resolution $Q_\bullet$ of $N$ in $\Gmod[\pi]{C}$. Viewing $Q_\bullet$ simply as a complex of $\Gamma$-modules via $\phi$, the map $f\colon M\to N$ lifts to a chain map $f_\bullet\colon P_\bullet\to Q_\bullet$, unique up to chain homotopy. Any continuous $G$-module map $Q_n\to B$ gives a $\Gamma$-module map $P_n\to A$ via $f_n$ and $g$, yielding a chain map
\[\Hom_{\ZG{}}(Q_\bullet, B) \to \Hom_{\Z \Gamma}(P_\bullet, A) \]
and thus, passing to cohomology, a map
\[\Ext^\bullet_{\ZG{}}(N, B)\to \Ext^\bullet_{\Z[\Gamma]}(M, A)  \]
which as usual is independent of the choices of $P_\bullet$ and $Q_\bullet$ and is natural with respect to maps of modules and maps of short exact sequences.

Similarly let $M_1$ be a right $\Gamma$-module and $M_2$ is a left $\Gamma$-module, and $N_1\in \Gmod[\pi]{C}^\perp$, $N_2\in \Gmod[\pi]{D}$. Then given homomorphisms $M_1\to N_1$, $M_2\to N_2$ which are compatible with $\phi$ there are canonical morphisms
\[\Tor^\Gamma_\bullet(M_1, M_2) \to \Tor^G_\bullet(N_1, N_2) \]

In our particular case of study, let $\Gamma$ be a discrete group and $\Sigma$ a finite family of subgroups of $\Gamma$. Let $(G,\famS)$ be a profinite group pair and assume $\phi$ is a map of group pairs (in the same sense as in Section \ref{SecBasicDefs}). There is then a commuting diagram of short exact sequences
\[\begin{tikzcd}
0\ar{r}&  \Delta_{\Gamma,\Sigma}\ar{r}\ar{d} & \Z[][\Gamma/\Sigma]\ar{r}\ar{d} & \Z\ar{r}\ar{d}  & 0\\
0\ar{r}& \Delta_{G,\famS}\ar{r} & \ZGS\ar{r} & \Z[\pi]\ar{r}& 0
\end{tikzcd}\]
where the top row is the defining exact sequence for relative homology of discrete groups, as in \cite{BE77}, Section 1. Let $A\in \Gmod[\pi]{D}$ and take (continuous) module morphisms from the above diagram to $A$---regarding $A$ as a $\Gamma$-module via $\phi$. Combining this with the canonical map $\Z\to \Z[\pi]$ and applying the above construction gives a commuting diagram of long exact sequences
\[\begin{tikzcd}
\cdots\ar{r} & H^k(G, A)\ar{r}\ar{d} & H^k(\famS, A)\ar{r}\ar{d} & H^k(G,\famS; A)\ar{r}\ar{d} & \cdots \\
\cdots\ar{r} & H^k(\Gamma, A)\ar{r} & H^k(\Sigma, A)\ar{r} & H^k(\Gamma,\Sigma; A)\ar{r} & \cdots 
\end{tikzcd}\]
where the bottom row is the relative cohomology sequence for discrete groups.
\begin{defn}
Let $\Gamma$ be a discrete group and $\Sigma$ be a finite family of subgroups of $\Gamma$. Let $G=\widehat\Gamma_{(\pi)}$ be the pro-$\pi$ completion relative to some set $\pi$ of primes, and let $i\colon \Gamma\to G$ be the canonical map. Let \famS\ be the finite family of subgroups $\overline{i(S)}$ where $S\in\Sigma$, where the bar denotes closure in $G$. We call the pair $(G,\famS)$ the {\em pro-$\pi$ completion} of the pair $(\Gamma,\Sigma)$. Note that the groups $\overline{i(S)}$ may not be the pro-$\pi$ completions of the discrete groups $S\in\Sigma$.

We say that $(\Gamma, \Sigma)$ is {\em (cohomologically) $\pi$-good} if for every $A\in \Gmod[\pi]{F}$ the map
\[H^\bullet(G,\famS; A)\to H^\bullet(\Gamma, \Sigma; A) \]
is an isomorphism. We also recall that $\Gamma$ is (cohomologically) $\pi$-good if for every $A\in \Gmod[\pi]{F}$ the map
\[H^\bullet(G, A)\to H^\bullet(\Gamma, A) \]
is an isomorphism.
\end{defn}
When $\pi$ is the set of all primes, any finite $\Gamma$-module is a $G$-module in a natural way and the definitions could be (and usually are) phrased in terms of `all finite $\Gamma$-modules'. When $\pi$ is not the set of all primes, not every finite $\pi$-primary $\Gamma$ module $M$ need be a $G$-module: one also requires that the image of the map $ \Gamma\to \Aut(M)$ is a $\pi$-group.

The definition is stated in terms of cohomology. Suppose $(\Gamma, \Sigma)$ is of type FP$_\infty$, in the sense that $\Delta_{\Gamma, \Sigma}$ has a resolution by finitely generated projectives $P_\bullet$. Then for any $M\in \Gmod[\pi]{F}$, we have
\[H_n(\Gamma, \Sigma; M)^\ast= H_n(P_\bullet^\perp\otimes_\Gamma M)^\ast= H^n((P_\bullet^\perp\otimes_\Gamma M)^\ast)=H^n(\Hom_\Gamma(P_\bullet, M^\ast))= H^n(\Gamma, \Sigma; M^\ast)  \]
 noting that $(-)^\ast = \Hom(-, I_\pi)$ is exact when applied to the sequences of finite $\pi$-modules $P_\bullet\otimes_\Gamma M$. Thus Pontrjagin duality holds for the discrete group pair. Since it holds for the profinite group pair as well we have the following proposition.
\begin{prop}\label{GoodHomol}
If $(\Gamma, \Sigma)$ is a $\pi$-good pair of type FP$_\infty$ and $(G,\famS)$ is its pro-$\pi$ completion then for every $M\in\Gmod[\pi]{F}$ the map of pairs $i\colon (\Gamma, \Sigma) \to (G,\famS)$ induces isomorphisms
\[H_\bullet(\Gamma, \Sigma; M) \iso H_\bullet(G,\famS; M) \] 
\end{prop}

A basic proposition that derives immediately from the above diagram of exact sequences and the 5-Lemma is the following.
\begin{prop}\label{LESGoodness}
Let $\Gamma$ be a discrete group and $\Sigma$ be a finite family of subgroups of $\Gamma$. Let $(G,\famS)$ be the pro-$\pi$ completion of $(\Gamma, \Sigma)$ and assume that for every $S\in\Sigma$ the natural map $\widehat{S}_{(\pi)}\to \overline{i(S)}$ is an isomorphism, where $\widehat{S}_{(\pi)}$ is the pro-$\pi$ completion. Assume further that each $S\in\Sigma$ is $\pi$-good. Then $\Gamma$ is $\pi$-good if and only if the pair $(\Gamma, \Sigma)$ is $\pi$-good. 
\end{prop}

We now proceed towards the expected result that the pro-$\pi$ completion of a PD$^n$ pair of discrete groups is a PD$^n$ pair at the prime $p$ where $p\in\pi$, under certain conditions on the pro-$\pi$ topology on $\Gamma$. 
\begin{defn}
Let $\Gamma$ be a discrete group and $G$ its profinite completion. For a $\Gamma$-module $P$, let $\widehat{P}$ be the $G$-module defined by
\[\widehat{P}_{(\pi)} = \varprojlim _{K, m} \Z/m\otimes_K P \]
where $K$ runs over the finite index normal subgroups of $\Gamma$ with $\Gamma/K$ a $\pi$-group and $m$ runs over those integers divisible only by primes from $\pi$. Here each module in the inverse limit acquires the $\Gamma/K$-action (and hence $G$-action) given by the left action on $P$. This is well-defined since $K$ is normal in $G$. By passing to a subsequence one may assume that this limit is indexed over a totally ordered set of pairs $(K,m)$. 
\end{defn}
\begin{lem}
The map $P\to \widehat{P}_{(\pi)}$ is an additive functor from the category of finitely generated $G$-modules to \Gmod[\pi]{C} taking finitely generated projective modules to finitely generated projective modules.
\end{lem}
\begin{proof}
The additivity statement follows immediately from the fact that an inverse limit of additive functors to \Gmod[\pi]{F} is an additive functor to \Gmod[\pi]{C}. By definition we have
\[\Zpiof{G} = \varprojlim_{K, m} \Z/m \otimes_K \Z{} [\Gamma]  \]
(see Section 5.3 of \cite{RZ00} for information on complete group rings). So our functor takes finitely generated free modules to finitely generated free modules. Since in both categories the finitely generated projectives are precisely the direct summands of the finitely generated free modules we are done.
\end{proof}
\begin{prop}
Suppose that the pair $(\Gamma, \Sigma)$ is $\pi$-good, and let $(G,\famS)$ be its pro-$\pi$ completion. Suppose that $P_\bullet{}_{(\pi)}$ is a resolution of $\Delta_{\Gamma, \Sigma}$ by finitely generated projective $\Gamma$-modules. Then the complex $\widehat{P}_\bullet{}_{(\pi)}$ is a resolution of $\Delta_{G,\famS}$ by finitely generated projective $\Zpiof{G}$-modules, where $(G,\famS)$ is the pro-$\pi$ completion of $(\Gamma, \Sigma)$. 
\end{prop}
\begin{proof}
We compute the homology of the chain complex $\widehat{P}_\bullet$ to verify that it is a resolution of $\Delta_{G,\famS}$.
We have
\begin{align*}
H_k(\widehat{P}_\bullet{}_{(\pi)}) & = H_k(\varprojlim \Z/m\otimes_K P_\bullet{}) && \\
& = \varprojlim H_k(\Z/m\otimes_K P_\bullet) && \text{by the Mittag-Leffler condition}\\
& = \varprojlim \Tor^K_k(\Z/m, \Delta_{K, \Sigma^K}) && \\
& = \varprojlim H_k(K, \Sigma^K; \Z/m) && \text{by Proposition \ref{relequalsext}}\\
& = \varprojlim H_k(\overline{K}, \famS^{\overline{K}}; \Z/m) && \text{by Proposition \ref{GoodHomol}}\\
& = H_k(1, \famS^1; \Z[\pi])&& \text{by Propostion \ref{IntersectionsOfSubGps}, since $\bigcap \overline{K} =1 $} \\
& = \Tor^1_k(\Z[\pi],\res^G_1(\Delta_{G,\famS})) && \\
& =\begin{cases} 0 & \text{if $k\neq 0$}\\ \Delta_{G,\famS} &\text{if $k=0$} \end{cases}
\end{align*}
where we have used several results from Section \ref{SecSubgpPairs} applied to discrete groups as they work just as well in the discrete case. We have further used the fact that, by definition of the pro-$\pi$ completion, the intersection of all the subgroups $\overline{K}\subseteq G$ (that is, of all the open subgroups) is trivial. For the Mittag-Leffler condition see Section 3.5 of \cite{Weibel95}. This is in fact a chain of isomorphisms of (right) $G$-modules when $\Delta$ and $\widehat{P}_\bullet{}_{(\pi)}$ are given the canonical right action dual to their left $G$-actions and the Tor groups have the right $G$-action induced as in Proposition \ref{ShapiroWithAction}.   
\end{proof}
\begin{clly}
Suppose that the pair $(\Gamma, \Sigma)$ is $\pi$-good, and let $(G,\famS)$ be its pro-$\pi$ completion. If $(\Gamma, \Sigma)$ has type FP$_\infty$ then $(G, \famS)$ has type \pFP[\infty] for all $p\in\pi$.
\end{clly}
\begin{proof}
The previous proposition shows that $\Z[\pi]$ is of type FP$_\infty$ in \Gmod[\pi]{C}. If we take a projective resolution of $\Z[\pi]$ which is finitely generated in each dimension then taking $p$-primary components (which is an exact functor) gives a resolution of $\Z[p]$ in \Gmod[p]{C} by finitely generated projectives by Proposition \ref{DiffPrimes}.
\end{proof}
\begin{clly}
Suppose $(\Gamma, \Sigma)$ is a $p$-good pair of discrete groups with $H^k(\Gamma, \Sigma; F)$ finite for all $p$-primary $\Gamma$-modules $F$ and all $k$. If $(G,\famS)$ is its pro-$p$ completion then $(G,\famS)$ has type \pFP[\infty].
\end{clly}
\begin{proof}
This follows immediately from Proposition \ref{FPnForVirtProP}.
\end{proof}
\begin{theorem}\label{GoodPDnPairs}
Suppose that the pair $(\Gamma, \Sigma)$ is of type FP$_\infty$  and let $(G,\famS)$ be its pro-$\pi$ completion. Suppose that $\Gamma$ and $(\Gamma, \Sigma)$ are $\pi$-good. Assume that $(\Gamma, \Sigma)$ is orientable if $2\notin \pi$. If $(\Gamma, \Sigma)$ is a PD$^n$ pair then $(G,\famS)$ is a PD$^n$ pair at the prime $p$ for every $p\in \pi$.
\end{theorem}
\begin{proof}
First note that by the previous proposition $(G,\famS)$ is of type \pFP[\infty] and by goodness $\cd_p(G,\famS)\leq \cd(\Gamma, \Sigma)=n$. Let $\widetilde \Z$ be the dualising module of $(\Gamma, \Sigma)$ (in the sense of Definition 4.1 of \cite{BE77}). Either the action on $\widetilde \Z$ is trivial or it factors through a map from $\Gamma$ to $\Z/2\Z$. Hence by the assumption on orientability the action of $\Gamma$ on $\widetilde \Z$ factors through $G$, so we may take the pro-$p$ completion to acquire a $G$-module $\widetilde{\Z[p]}$. The Pontrjagin dual of $\widetilde{\Z[p]}$ is a module $\widetilde I_p$ isomorphic to $I_p$ as an abelian group.

 For a finite module $M$ let $\widetilde M=\widetilde\Z\otimes M$ with the diagonal action---that is, the abelian group $M$ with an appropriately twisted action. Note that $\widetilde\Z\otimes\widetilde\Z=\Z$.
  
We have the following chain of isomorphisms induced by $\pi$-goodness and Poincar\'e duality for the pair $(\Gamma, \Sigma)$ and its finite index normal subgroups $K$ with $\Gamma/K$ a $\pi$-group (goodness is inherited by such subgroups by Proposition \ref{GoodFISubgps} below). We also note that by the analogue of Proposition \ref{FPforPDn} for discrete groups, $\Gamma$ is of type FP$_\infty$ and we may also use homological goodness as in Proposition \ref{GoodHomol}. The inverse limits are indexed over pairs $(U,m)$ where $U$ is an open normal subgroup of $G$ with $G/U$ a $\pi$-group and $m$ is a $\pi$-number.
\begin{eqnarray*}
{\bf H}^\ast(G,\famS; \Zpof{G}) & = & \varprojlim H^\ast(G,\famS; \Z/p^m[G/U])\\
& \iso & \varprojlim H^\ast(\Gamma, \Sigma; \Z/p^m[G/U])\\
& \iso & \varprojlim H^\ast(\Gamma, \Sigma; \Z/p^m[\Gamma/\Gamma\cap U])\\
& \iso & \varprojlim H^\ast(\Gamma, \Sigma; \ind^{\Gamma\cap U}_\Gamma(\Z/p^m))\\
& \iso & \varprojlim H^\ast(\Gamma, \Sigma; \coind^{\Gamma\cap U}_\Gamma(\Z/p^m))\\
& \iso & \varprojlim H^\ast(\Gamma \cap U, \Sigma^{\Gamma\cap U}; \Z/p^m) \\
& \iso & \varprojlim H_{n-\ast}(\Gamma\cap U, \widetilde{\Z/p^m}) \\
& \iso & \varprojlim H_{n-\ast}(U, \widetilde{\Z/p^m})  \\
& \iso & H_{n-\ast}(1, \widetilde{\Z}_p) \\
& = & \begin{cases} 0 & \text{if $\ast\neq n$} \\ \widetilde{\Z}_p & \text{if $\ast= n$} \end{cases}
\end{eqnarray*}
so $(G,\famS)$ is indeed a PD$^n$ pair. These are all isomorphisms of $G$-modules where the actions in the first four lines are derived from the right action of $G$ on $\Zpof{G}$ or on $\Z/p^m[G/U]$ and the actions in the remaining lines are given by the conjugation action of $G$. See Proposition \ref{ShapiroWithAction}. The second to last line follows by the absolute version of Proposition \ref{IntersectionsOfSubGps}.
\end{proof}
The absolute version of Theorem \ref{GoodPDnPairs} was first proved by Pletch \cite{Pletch80b}.
\subsection{Further properties of relative goodness}
Here we state and give brief proofs of some properties of goodness which mirror those of absolute goodness.
\begin{prop}\label{GoodFISubgps}
Let $(\Gamma, \Sigma)$ be a $\pi$-good pair, and let $K$ be a finite index normal subgroup of $\Gamma$ such that $\Gamma/K$ is a $\pi$-group. Then $(K, \Sigma^K)$ is also $\pi$-good.
\end{prop}  
\begin{proof}
Let $(G,\famS)$ be the pro-$\pi$ completion of $(\Gamma, \Sigma)$ and let $U=\overline{K}$ be the closure of $K$ in $G$. Then $(U, \famS^U)$ is the pro-$\pi$ completion of $(K, \Sigma^K)$ and for any $M\in \mathfrak{F}_\pi(U)$ there is a commuting diagram
\[ \begin{tikzcd}
H^r(U,\famS^U; M)\ar[equal]{r} \ar{d} & H^r(G, \famS; \Zpiof{G}\hotimes[U] M)\ar{d}\\
H^r(K, \Sigma^K; M) \ar[equal]{r} & H^r(\Gamma, \Sigma; \Z{}[\Gamma]\otimes_K M)
\end{tikzcd}\]
Since $G/K$ is a finite $\pi$-group the coefficient modules in the right-hand column are isomorphic and lie in $\Gmod[\pi]{F}$, hence the right vertical map is an isomorphism by hypothesis and we are done. 
\end{proof}
\begin{defn}
Let $\Gamma$ be a discrete group and let $\Lambda\leq \Gamma$. Let $\pi$ be a set of primes. Then we say that $\Lambda$ is {\em $\pi$-separable} in $\Gamma$ if for every $g\in\Gamma\smallsetminus \Lambda$ there is a map $\phi$ from $\Gamma$ to a finite $\pi$-group such that $\phi(g)\notin \phi(\Lambda)$.

We say that {\em $\Gamma$ induces the full pro-$\pi$ topology on $\Lambda$} if for every finite index normal subgroup $U$ of $\Lambda$ with $\Lambda/U$ a $\pi$-group there is a finite index normal subgroup $V$ of $\Gamma$ with $\Gamma/V$ a $\pi$-group and with $V\cap \Lambda\leq U$.

We say that $\Lambda$ is {\em fully $\pi$-separable} in $\Gamma$ if it is $\pi$-separable in $\Gamma$ and $\Gamma$ induces the full pro-$\pi$ topology on $\Lambda$.
\end{defn}
As usual we omit the symbol $\pi$ when $\pi$ is the set of all primes. An immediate consequence of full $\pi$-separability of a subgroup $\Lambda\leq \Gamma$ is that the natural map from the pro-$\pi$ completion of $\Lambda$ to the pro-$\pi$ completion of $\Gamma$ is an isomorphism to its image.
\begin{prop}
Let $1\to N\to E \to \Gamma\to 1$ be an extension of groups, such that $N$ is finitely generated, $E$ induces the full pro-$\pi$ topology on $N$, and $H^n(N, M)$ is finite for any $n$ and any finite $\pi$-primary $N$-module $M$. Let $\Pi$ be a finite family of subgroups of $E$ each containing $N$ and let $\Sigma$ be the image of this family in $\Gamma$. Assume $(\Gamma, \Sigma)$ is a $\pi$-good pair and that $N$ is $\pi$-good. Then $(E, \Pi)$ is $\pi$-good.
\end{prop}
\begin{proof}
The conditions of the theorem imply the existence of a short exact sequence of pro-$\pi$ completions 
\[1 \to \widehat{N}_\pi \to \widehat{E}_\pi \to \widehat{\Gamma}_\pi \to 1 \]
For any $M\in\mathfrak{C}_\pi(\widehat{E}_\pi)$ the natural maps of the discrete groups to their pro-$\pi$ completions induce a natural map from the relative Lyndon-Hochschild-Serre spectral sequence (Proposition \ref{LHS2}) to the analogous spectral sequence for the extension of discrete groups. By the goodness assumptions this map is an isomorphism 
\[H^r(\widehat{\Gamma}_\pi, \widehat{\Sigma}_\pi; H^s(\widehat{N}_\pi, M)) \to H^r(\Gamma, \Sigma; H^s(N, M)) \]
on the second page, hence gives an isomorphism
\[H^{r+s}(\widehat{E}_\pi, \widehat{\Pi}_\pi; M) \to H^{r+s}(E, \Pi; M) \] 
in the limit as required.
\end{proof}

We will need some terminology before moving on.
\begin{defn}
Let ${\cal G}= (X,\Gamma_\bullet)$ be a finite graph of finitely generated groups with fundamental group $\Gamma$. We say that ${\cal G}$ is {\em $\pi$-efficient} if $\Gamma$ is residually $\pi$-finite and each group $\Gamma_x$ is fully $\pi$-separable in $\Gamma$.
\end{defn}
One immediate consequence of $\pi$-efficiency is that the corresponding graph of pro-$\pi$ completions is proper and has pro-$\pi$ fundamental group equal to the pro-$\pi$ completion of $\Gamma$. 
\begin{prop}
Let $(X,\Gamma_\bullet)$ be a $\pi$-efficient finite graph of finitely generated groups with fundamental group $\Gamma$.  Let $\Sigma_x$ be a finite family of subgroups of $\Gamma_x$ for $x\in V(X)$ and let ${\cal E}_x$ be the set of edge groups incident to $\Gamma_x$. Let $\Sigma=\bigsqcup_{x\in V(X)}\Sigma_x$. Suppose that $(\Gamma_x, \Sigma_x\sqcup {\cal E}_x)$ is a $\pi$-good pair for each $x\in V(X)$ and that $G_e$ is $\pi$-good for each $e\in E(X)$. Then $(\Gamma, \Sigma)$ is $\pi$-good.
\end{prop}
\begin{proof}
This may be proved by induction from the cases of amalgamated free products and HNN extensions. The conditions of the theorem give a map from the Mayer-Vietoris sequences from Theorems \ref{MVwithExcision} and \ref{HNNMVwithExcision} to the analagous sequences for discrete groups. The result then follows from the Five Lemma. 
\end{proof}
\subsection{Examples of good PD$^n$ pairs}
In this section we give the expected examples of good group pairs and profinite PD$^n$ pairs. First we give a sufficient condition for a group pair to have type FP$_\infty$ which the author could not immediately find in the literature.
\begin{prop}\label{DiscFPn}
Suppose that $X$ is a connected aspherical cell complex with finitely many cells in each dimension with fundamental group $\Gamma$, and suppose that $\{Y_i\}_{1\leq i\leq n}$ $(n\geq 1)$ is a collection of connected disjoint subcomplexes. Suppose further that each $Y_i$ is an aspherical complex with fundamental group $S_i$ and that the inclusions of complexes induce injective maps $S_i\to \Gamma$. Then if $\Sigma=\{S_i\}_{1\leq i\leq n}$, the group pair $(\Gamma, \Sigma)$ is of type FP$_\infty$.
\end{prop}
\begin{rmk}
We have been somewhat sloppy where basepoints are concerned; however the basepoints are entirely irrelevant to the outcome as the module $\Delta_{\Gamma, \Sigma}$ is independent of the choices of the $S_i$ up to conjugacy by Proposition \ref{DeltaUpToConjugacy}.
\end{rmk}
\begin{proof}
We may assume that every vertex in $X$ lies in one of the $Y_i$ in the following manner. The 1-skeleton of $X$ is some connected finite graph and therefore admits a `maximal subforest relative to the $Y_i$'. That is, a collection of disjoint trees which together include all vertices of $X$ not in any $Y_i$ and each including exactly one vertex of some $Y_i$. If we enlarge each $Y_i$ by attaching those trees incident to it we do not change any of the homotopy-theoretic properties of the statement of the proposition. So we may proceed with this assumption.

For a cell complex $Z$ let $C_\bullet(Z)$ be its cellular chain complex. Consider the universal cover $\tilde X$ with projection map $p\colon \tilde X\to X$. The cellular chain complex $E_\bullet = C_\bullet(\tilde X)$ is a free resolution of $\Z$ by finitely generated free $\Z{} \Gamma$-modules. This chain complex has a subcomplex $D_\bullet= \bigoplus_{i=1}^n C_\bullet(p^{-1}(Y_i))$ deriving from the pre-images of the $Y_i$. This subcomplex is isomorphic as a complex of $\Gamma$-modules to
\[\bigoplus_{i=1}^n \Z \Gamma\otimes_{S_i} C_\bullet(\tilde Y_i)  \]
where $\tilde Y_i$ is the universal cover of $Y_i$. Now $D_\bullet$ is a free resolution of $\bigoplus_{i=1}^n \Z \Gamma\otimes_{S_i} \Z$ and the inclusion map $D_\bullet\to E_\bullet$ induces the augmentation map $\bigoplus_{i=1}^n \Z \Gamma\otimes \Z\to \Z$. Finally note that since $D_\bullet\to E_\bullet$ is induced by inclusions of subcomplexes, $E_\bullet / D_\bullet$ is a complex of free finitely generated modules. From the long exact sequence in homology deriving from the short exact sequence of chain complexes $D_\bullet \to E_\bullet\to E_\bullet/D_\bullet$ we find that $E_\bullet / D_\bullet$ is exact in degree at least 2 and that $H_1(E_\bullet/D_\bullet)\iso \Delta_{\Gamma, \Sigma}$. By the first part of the construction $E_0/D_0=0$. Thus $E_\bullet / D_\bullet$ provides a free resolution of $\Delta_{\Gamma, \Sigma}$ by finitely generated free modules as required.
\end{proof}
\begin{prop}
Let $X$ be a compact surface with boundary components $l_1, \ldots, l_n$ which is not a disc. Let $\Gamma=\pi_1 X$ and $\Sigma=\{i_\ast(\pi_1 l_i) \}_{i=1}^n$ where $i_\ast$ denotes an inclusion map. Let $\pi$ be a set of primes and assume that $X$ is orientable if $2\notin \pi$. Then the pair $(\Gamma, \Sigma)$ is $\pi$-good and the pro-$\pi$ completion of $(\Gamma, \Sigma)$ is a PD$^2$ pair at every prime $p\in \pi$. 
\end{prop}
\begin{proof}
Note that $\Gamma$ is a free group and therefore $\pi$-good for every $\pi$ (see for example Exercise 2.6.2 of \cite{Serre13}). Furthermore boundary subgroups are fully $\pi$-separable for every $\pi$ (see for instance Propositions 3.1 and 3.2 of \cite{Wilkes16} where this is done for $\pi=\{p\}$; much the same arguments work for general $\pi$). Thus we may apply Proposition \ref{LESGoodness} and Theorem \ref{GoodPDnPairs} to find the result.
\end{proof}

One cannot quite give such a sweeping statement as this in dimension 3, since the pro-$\pi$ topology on a 3-manifold group may be very poorly behaved. We will therefore restrict our attention to the set of all primes---that is, we will consider the profinite completion. We must first establish separability conditions for boundary components, or more generally fundamental groups of embedded surfaces. These are not entirely new results; see the historical note at the end of the section.

\begin{theorem}\label{GofGsFullSep}
Let ${\cal G}= (X,\Gamma_\bullet)$ and ${\cal L}= (Y,\Lambda_\bullet)$ be finite graphs of finitely generated groups. Let $f\colon Y\to X$ be a map of graphs and for every $y\in Y$ let $\phi_y\colon \Lambda_y\to \Gamma_{f(y)}$ be an injective group homomorphism compatible with the boundary maps in $\cal L$ and $\cal G$. Suppose further that the induced map on the fundamental groups $\phi\colon\Lambda=\pi_1({\cal L})\to \pi_1({\cal G})=\Gamma$ is injective. 

If $\cal G$ is an efficient graph of groups and for each $y\in Y$ the subgroup $\Gamma_{f(y)}$ induces the full profinite topology on $\phi_y(\Lambda_y)$ then $\Gamma$ induces the full profinite topology on $\phi(\Lambda)$.
\end{theorem}
\begin{proof}
Consider a finite index normal subgroup $U$ of $\Lambda$. Since each $\Lambda_y$ is fully separable in $\Gamma_{f(y)}$ there are finite index subgroups $V_x$ of the $\Gamma_x$ such that $V_x\cap \Lambda_y\leq U\cap \Lambda_y$ for every $y$ with $f(y)=x$. Since the graph of groups $\cal G$ is efficient we may assume, by passing to deeper subgroups $V_x$ if necessary, that there is a finite index normal subgroup $\widetilde\Gamma$ of $\Gamma$ such that $V_x = \widetilde\Gamma\cap \Gamma_x$ for all $x\in X$. 

Now $\widetilde\Gamma$ is the fundamental group of a graph of groups $\widetilde{\cal G}=(\widetilde X, \widetilde\Gamma_\bullet)$ whose vertex and edge groups are representatives of the conjugacy classes of the $V_x$ in $\Gamma$ and where $\widetilde X$ is some finite cover of the graph $X$. Similarly $\widetilde\Lambda =\widetilde\Gamma \cap \Lambda$ is the fundamental group of a graph of groups $\widetilde{\cal L}=(\widetilde Y, \widetilde\Lambda_\bullet)$ whose vertex and edge groups are representatives of the conjugacy classes of the $V_x\cap \Lambda_y=\widetilde\Gamma\cap \Lambda_y$ in $\Lambda$ and $\widetilde Y$ is some finite cover of the graph $Y$. Note that if we find a finite index subgroup $W$ of $\widetilde\Gamma$ with $W\cap\widetilde\Lambda\leq U\cap \widetilde \Lambda$ then we are done; for then $W$ is finite index in $\Gamma$ and
\[W\cap\Lambda=(W\cap\widetilde\Gamma )\cap \Lambda=W\cap \widetilde\Lambda\leq U\cap\widetilde\Lambda\leq U \]
Now by construction $U\cap\widetilde{\Lambda}$ contains all the vertex groups of $\widetilde{\cal L}$ which are (conjugates in $\Lambda$ of) $V_x\cap\Lambda_y\leq U\cap \Lambda_y$. Therefore the quotient map $f\colon \widetilde\Lambda\to \widetilde{\Lambda}/(U\cap\widetilde{\Lambda})$ factors through the map to the graph fundamental group $\pi_1(\widetilde Y)$. 
\[\begin{tikzcd}
\widetilde\Lambda \ar[equal]{r}\ar[hookrightarrow]{d} & \pi_1({\cal L})\ar[twoheadrightarrow]{r}\ar[hookrightarrow]{d} & \pi_1(\tilde Y) \ar[twoheadrightarrow]{r} \ar[hookrightarrow]{d} & \widetilde\Lambda/U\cap\widetilde\Lambda \ar[hookrightarrow]{d} \\
\widetilde\Gamma \ar[equal]{r} & \pi_1({\cal G})\ar[twoheadrightarrow]{r} & \pi_1(\widetilde X) \ar[twoheadrightarrow]{r} & \pi_1(\widetilde X)/W'
\end{tikzcd} \]
Now $\pi_1 \widetilde Y$ is a finitely generated subgroup of the free group $\pi_1 \widetilde X$, which is subgroup separable by a theorem of Marshall Hall \cite{Hall49}. Therefore there is a finite index normal subgroup $W'$ of $\pi_1 \widetilde X$ with $W'\cap \pi_1(\widetilde Y)$ contained in the kernel of the map $\pi_1(\widetilde Y)\to \widetilde{\Lambda}/(U\cap\widetilde{\Lambda})$. The preimage of $W'$ in $\widetilde\Gamma$ is the required subgroup $W$.
\end{proof}

\begin{theorem}\label{EmbSurfFullSep}
Let $M$ be a closed 3-manifold and let $L$ be an embedded $\pi_1$-injective surface in $M$. Then $\pi_1 L$ is fully separable in $\pi_1 M$. 
\end{theorem}
\begin{proof}
Separability is a theorem of Przytycki and Wise \cite{PW14}. We consider the profinite topology.

First note that we are free to pass to a finite index cover of $M$ so we may assume that both $M$ and $L$ are orientable. Consider the spheres in the Kneser-Milnor decomposition of $M$ and perturb them to make them transverse to $L$. Since $L$ is incompressible by the Loop Theorem, any intersection curve of $L$ with a sphere bounds a disc in $L$; so by performing surgeries on $L$ we may find a surface disjoint from all the spheres carrying the same fundamental group as $L$. Since the profinite topology on a free product of residually finite groups is efficient, hence induces the full profinite topology on any free factor, we have now reduced to the case when $M$ is an irreducible 3-manifold. 

The JSJ tori of $M$ induce an efficient graph of groups decomposition of $\pi_1 M$ (\cite{WZ10}, Theorem A). After possibly performing a small isotopy of the tori, the intersections of the tori with $L$ are a finite collection of simple closed curves which split $L$ as a finite graph of groups with finitely generated vertex groups. The inclusion of $F$ into $M$ induces exactly such a map of graphs of groups as in Theorem \ref{GofGsFullSep}. Each piece of the JSJ decomposition is either Seifert fibred or cusped hyperbolic, hence the fundamental groups are subgroup separable. In the Seifert fibred case this is a theorem of Scott \cite{scott78}. The cusped hyperbolic case is part of the recent seminal advances in 3-manifold theory pioneered by Wise, Agol, Przytycki and many others. The reader is directed to \cite{AFW15}, Section 5.2 for a complete account and the appropriate citations. Thus we may apply Theorem \ref{GofGsFullSep} and this completes the theorem.
\end{proof}

\begin{theorem}\label{FullSepWBdy}
Let $M$ be a compact orientable 3-manifold with incompressible boundary and let $L$ be a properly embedded, incompressible and boundary incompressible surface in $M$. Then $\pi_1 L$ is fully separable in $\pi_1 M$. 
\end{theorem}
\begin{proof}
Let $DM$ be the double of $M$ along its boundary and let $DL$ be the double of $L$ along its boundary, canonically embedded in $DM$. If $L$ is a closed surface let $DL=L$. Since the boundary of $M$ is incompressible $\pi_1 M$ injects into $\pi_1 DM$. Then $DL$ satisfies all the conditions of Theorem \ref{EmbSurfFullSep} and so $\pi_1 DL$ is fully separable in $\pi_1 DM$. If $g\in\pi_1 M \smallsetminus \pi_1 L$ then $g\in\pi_1 DM \smallsetminus \pi_1 DL$ and a homomorphism from $\pi_1 DM$ to a finite group separating $g$ from $\pi_1 DL$ separates $g$ from $\pi_1 L$ in $\pi_1 M$. So $\pi_1 L$ is separable. 

Now let $U$ be a finite index subgroup of $\pi_1 L$. The preimage $DU$ of $U$ under the `folding map' $\pi_1 DL\to \pi_1 L$ is a finite index subgroup of $\pi_1 DL$ meeting $\pi_1 L$ precisely in $U$. Since $\pi_1 DL$ is fully separable there exists a finite index subgroup $V$ of $\pi_1 DM$ such that $V\cap \pi_1 DL\leq DU$. Then $V\cap\pi_1 M $ is a finite index subgroup of $\pi_1 M$ such that 
\[ (V\cap \pi_1 M)\cap \pi_1 L = V\cap \pi_1 DL\cap \pi_1 L \leq  DU\cap \pi_1 L= U \]
as required.
\end{proof}
\begin{clly}\label{BdySep}
Let $M$ be a compact 3-manifold with $\pi_1$-injective boundary and let $L$ be a boundary component of $M$. Then $\pi_1 L$ is fully separable in $\pi_1 M$. 
\end{clly}
\begin{proof}
We are free to pass to a double cover so that $M$ is orientable. Then the boundary component $L$ is orientable, incompressible (by the Loop Theorem) and hence also boundary incompressible. Theorem \ref{FullSepWBdy} now applies and we are done.
\end{proof}
\begin{theorem}\label{relPD3}
Let $M$ be a compact aspherical 3-manifold with incompressible boundary components $\bdy M_1,\ldots \bdy M_r$. Let $\Gamma = \pi_1 M$, let $\Sigma= \{\pi_1 \bdy M_i\}_{1\leq i\leq r}$. Then $(\Gamma, \Sigma)$ is a good pair (with respect to the set of all primes) and its profinite completion is a PD$^3$ pair at every prime $p$.
\end{theorem}
\begin{proof}
Goodness follows immediately from Corollary \ref{BdySep} and Propositions \ref{LESGoodness} once we know that 3-manifold groups are good, which is well-known and may be found in work of various authors. See \cite{AFW15}, Section 5.2 for a full account.  That the profinite completion is a PD$^3$ pair at every prime now follows from Proposition \ref{DiscFPn} and Theorem \ref{GoodPDnPairs}.
\end{proof}

\begin{rmk}[Historical Note]
The full separability result above (Theorem \ref{FullSepWBdy}) has various precursors in the literature, although the result in its greatest strength seems to be new. The fact that boundary components are separable was first proved by Long and Niblo \cite{LN91}. In the case when the boundary is toroidal, full separability was established by Hamilton \cite{Ham01}. As cited in the proof, separability for embedded surface subgroups is a theorem of Przytycki and Wise \cite{PW14}.
\end{rmk}
\subsection{Decompositions of 3-manifold groups}
We conclude with a couple of simple applications of the theory of relative profinite duality groups, which parallel and extend the results of \cite{WZ17}.
\begin{theorem}
Let $M$ and $N$ be compact orientable 3-manifolds with incompressible boundary. Let the Kneser-Milnor decompositions of $M$ and $N$ be $M=M_1\# \cdots \# M_r\# F$ and $N=N_1\# \cdots \# N_s\# F'$ where each $M_i$ and $N_j$ is irreducible and $F_k$ and $F'_l$ are connect sums of $k$ and $l$ copies of $\ss{2}{1}$ respectively. Assume that there is an isomorphism $\Phi\colon \widehat{\pi_1 M}\iso  \widehat{\pi_1 N}$. Then $r=s$, $k=l$ and, up to reordering, each $\widehat{\pi_1 M_i}$ is isomorphic to (a conjugate of)  $\widehat{\pi_1 N_i}$ via $\Phi$.
\end{theorem}
\begin{rmk}
For the case of closed 3-manifolds this is Theorem 2.2 of \cite{WZ17}.
\end{rmk}
\begin{proof}
Consider the profinite tree $T$ dual to the splitting 
\[\widehat{\pi_1 N}=\widehat{\pi_1 N_i}\amalg \cdots \amalg\widehat{\pi_1 N_i}\amalg \widehat{\pi_1 F'} \]
of $\widehat{\pi_1 N}$ as an efficient graph of groups, and consider the action of each $\widehat{\pi_1 M_i}$ on $T$ via $\Phi$. If $\Sigma_i$ denotes the family of fundmental groups boundary components of $M$ lying in $M_i$ (the boundary components are incompressible, so the Kneser-Milnor decomposition only involves spheres which do not meet boundary components) then by Theorem \ref{relPD3} the profinite completion $(\widehat{\pi_1 M_i},\famS_i)$ of the pair $(\pi_1 M_i, \Sigma_i)$ is a PD$^3$ pair at every prime $p$. Therefore, since the edge stabilisers of the action are trivial, Corollary \ref{FixedVertex} implies that $\widehat{\pi_1 M_i}$ fixes some vertex of $T$, and hence is conjugate into some $\widehat{\pi_1 N_j}$ (it is not conjugate into $\widehat{F_s}$ by reason of cohomological dimension). By symmetry every $\widehat{\pi_1 N_j}$ is conjugate into some $\widehat{\pi_1 M_i}$. As subgroups of profinite groups cannot be conjugate into proper subgroups of themselves, every $\widehat{\pi_1 M_i}$ is isomorphic via $\Phi$ to a conjugate of some $\widehat{\pi_1 N_j}$. By considering the action on $T$ one may readily see that the different $\widehat{\pi_1 M_i}$ are not conjugate to each other. Therefore $r=s$ and after reordering we have that each $\widehat{\pi_1 M_i}$ is isomorphic to a conjugate of $\widehat{\pi_1 N_i}$. Finally taking the quotient by the normal subgroup generated by the $\widehat{\pi_1 M_i}$ and the $\widehat{\pi_1 N_i}$ gives an isomorphism $\widehat{\pi_1 F}\iso\widehat{\pi_1 F'}$ whence $k=l$ as different free groups are distinguished by their profinite completions.
\end{proof}
\begin{prop}
Let $M$ be a compact hyperbolic 3-manifold with empty or toroidal boundary. If $\widehat{\pi_1 M}$ acts on a profinite tree $T$ with abelian edge stabilisers then $\widehat{\pi_1 M}$ fixes a (unique) vertex.
\end{prop}
\begin{rmk}
The statement of this proposition is identical with Lemma 4.4 of \cite{WZ17}, where it is proved using Dehn filling techniques. In that paper it is suggested that a theory of profinite duality pairs would provide an alternative proof, and we include it here to illustrate that relative cohomology can indeed by used for this purpose.
\end{rmk}
\begin{proof}
Let $P_1,\ldots, P_r$ be representatives of the peripheral subgroups of $\pi_1 M$, let $G=\widehat{\pi_1 M}$ and let $\famS=\{\widehat{P}_1,\ldots, \widehat{P}_r\}$. Then by Theorem \ref{relPD3} the pair $(G,\famS)$ is a PD$^3$ pair at every prime. Suppose that $G$ acts on a profinite tree $T$ with abelian edge stabilisers. If $G$ fixes a vertex then it fixes a unique vertex since $G$ is non-abelian and edge stabilisers are abelian (Corollary 2.9 of \cite{ZM89}). So assume for a contradiction that $G$ does not fix any vertex. Then by Corollary \ref{FixedVertex} there is some edge $e\in E(T)$ such that $\cd_p(G_e, \famS^{G_e})\geq 2$.  

We will prove that $G_e$ is contained in some conjugate $\widehat{P}_i^\gamma$ of a peripheral subgroup. First we recall that for every $g\in G$ and $\widehat{P}_i\cap \widehat{P}_j^g\neq 1$ then $i=j$ and $g\in \widehat{P}_i$. This critical fact---the malnormality of the peripheral subgroups---is Lemma 4.5 of \cite{WZ14}.

Now consider the group pair $(G_e, \famS^{G_e})$. If $\cd_p(G_e)\geq 2$ (recalling that $G_e$ is abelian) then by Theorem 9.3 of \cite{WZ14} $G_e$ is conjugate into some peripheral subgroup. If on the other hand $\cd_p(G_e)=1$ then in order that $\cd_p(G_e, \famS^{G_e}) \geq 2$ we must have some non-trivial group in $\famS^{G_e}$. By definition of $\famS^{G_e}$ this means that there exists some $\gamma\in G$ and some $i$ such that $G_e\cap \widehat{P}_i^\gamma$ contains some non-trivial element $h$. Since $G_e$ is abelian, $h$ is therefore also an element of $\widehat{P}_i^{\gamma g}$ for every $g\in G_e$. By malnormality this forces $g\in \widehat{P}_i^\gamma$. So $G_e$ is indeed contained in the conjugate $\widehat{P}_i^\gamma$ of a peripheral subgroup.

Consider the family of subgroups 
\[\famS^{G_e} = \left\{ G_e \cap \sigma(y)\widehat{P}_i \sigma(y)^{-1} \mid x\in X, y\in G_e\lqt G/\widehat{P}_i, 1\leq i\leq r\right\} \]
where $\sigma\colon G_e\lqt G\to G$ is some continuous section. By malnormality, if the left hand side of
\[G_e \cap \sigma(y)\widehat{P}_j \sigma(y)^{-1} \leq \widehat{P}_i^\gamma \cap \sigma(y)\widehat{P}_j \sigma(y)^{-1}\]
is non-trivial for any $y$ and $j$ then $i=j$ and $\gamma\sigma(y)\in \widehat{P}_i$, and hence $G_e y \widehat{P}_j = G_e \gamma \widehat{P}_i$. So precisely one of the groups in the family $\famS^{G_e}$ is non-trivial, and equals $G_e$. Therefore by Lemma \ref{LemReductionOfFamily} we have 
\[ H_k(G_e, \famS^{G_e}; M) = H_k(G_e, \{G_e\}; M)=0\]
for all $k\geq 2$ and all $M\in\mathfrak{C}(G_e)$ and therefore $\cd_p(G_e,\famS^{G_e})<2$. This contradiction completes the proof.
\end{proof}

\setcounter{section}{0}
\setcounter{subsection}{0}
\setcounter{figure}{0}
\setcounter{table}{0}
\renewcommand\thefigure{\Alph{section}\arabic{figure}}
\renewcommand\thetable{\Alph{section}\arabic{table}}

\renewcommand\thesection{Appendix \Alph{section}}
\section{Profinite Direct Sums}\label{AppSheaves}
\renewcommand\thesection{\Alph{section}}

We have required several facts about profinite direct sums of modules. To the author's knowledge these have not yet appeared in these precise forms in published literature. It has seemed expedient to include them here so that the curious reader has access to the proofs. Analogous facts for free profinite products of groups appear in \cite{Ribes17}. One may adapt these to the case of modules, and this is what we do here. The author claims no particular credit for most of these results or proofs, which are, apart from Proposition \ref{DirSumsAndFreeMods} onwards, taken and adapted from material in Chapter 5 and Section 9.1 of \cite{Ribes17}. In this section all modules are compact.

First we recall the definition.
\begin{defn}
Let $R$ be a profinite ring. A {\em sheaf of $R$-modules} consists of a triple $({\cal M}, \mu, X)$ with the following properties.
\begin{itemize}
\item ${\cal M}$ and $X$ are profinite spaces and $\mu\colon{\cal M}\to X$ is a continuous surjection.
\item Each {\em `fibre'} ${\cal M}_x=\mu^{-1}(x)$ is endowed with the structure of a compact $R$-module such that the maps
\[R\times{\cal M} \to {\cal M}, \quad (r,m)\to r\cdot m \]
\[{\cal M}^{(2)}=\{(m,n)\in{\cal M}^2\mid \mu(m)=\mu(n)\} \to {\cal M}, \quad (m,n)\to m+n \]
are continuous.
\end{itemize} 

A {\em morphism of sheaves} $(\alpha, \bar\alpha)\colon ({\cal M}, \mu, X) \to ({\cal M'}, \mu', X') $ consists of continuous maps $\alpha:{\cal M}\to {\cal M'}$ and $\bar\alpha\colon X\to X'$ such that $\mu'\alpha = \bar\alpha\mu$ and such that the restriction of $\alpha$ to each fibre is a morphism of $R$-modules ${\cal M}_x\to{\cal M'}_{\bar\alpha(x)}$.
\end{defn}

We often contract `the sheaf $({\cal M}, \mu, X)$' to simply `the sheaf $\cal M$'. Regarding an $R$-module as a sheaf over the one-point space one may talk of a sheaf morphism from a sheaf to an $R$-module.

\begin{defn}
A {\em profinite direct sum} of a sheaf $\cal M$ consists of an $R$-module $\bigboxplus_X\cal M$ and a sheaf morphism $\omega\colon {\cal M}\to \bigboxplus_X \cal M$ (sometimes called the `canonical morphism') such that for any $R$-module $N$ and any sheaf morphism $\beta\colon{\cal M}\to N$ there is a unique morphism of $R$-modules $\tilde\beta\colon\bigboxplus_X{\cal M}\to N$ such that $\tilde\beta\omega=\beta$.
\end{defn}
Note that as any compact module is an inverse limit of finite modules it is sufficient to verify this universal property for finite $N$. 
\begin{prop}\label{SheavesConstr}
The profinite direct sum exists and is unique up to canonical isomorphism.
\end{prop}
\begin{proof}
Uniqueness follows immediately from the universal property. Now define 
\[ A= \bigoplus_{x\in X} {\cal M}_x\]
to be the abstract direct sum. Let $f\colon {\cal M}\to A$ be the natural function and let $\cal U$ be the set of all finite index submodules $U$ of $A$ such that ${\cal M}\to A\to A/U$ is continuous. Denote by $M$ the completion of  $A$ with respect to the topology $\cal U$. That is, $M$ is the inverse limit of the $A/U$ for $U\in\cal U$. Then $M$ is a compact $R$-module and $f$ extends to a continuous sheaf morphism  $\omega\colon {\cal M}\to M$. Note that since $f(\cal M)$ generates $A$ the image of $\omega$ topologically generates $M$.

Suppose that $N$ is a finite $R$-module and $\beta \colon{\cal M}\to N$ is a morphism. Then there is a natural map of abstract $R$-modules $g\colon A\to N$ such that $g\circ f=\beta$, which is continuous. Hence $\ker(g)\in\cal U$ so $g$ extends to a continuous map of $R$-modules $\tilde \beta\colon M\to N$ such that $\tilde \beta\omega=\beta$. This map is unique since the image of $\omega$ generates $M$. So $M$ satisfies the relevant universal property and is a profinite direct sum of $\cal M$.
\end{proof}
\begin{lem}\label{ShavesLem1}
Let $\cal M$ be a sheaf of $R$-modules and let $\nu\colon{\cal M}\to N$ be a continuous map to a finite $R$-module such that for some $x\in X$ the restriction $\nu_x\in {\cal M}_x\to N$ is a ring homomorphism. Then there exists a neighbourhood $Y$ of $x$ in $X$ such that $\nu_y$ is a ring homomorphism for all $y\in Y$. 
\end{lem}
\begin{proof}
	Consider the continuous maps $\eta\colon R\times {\cal M}^{(2)}\to N\times N$, $\rho\colon R\times{\cal M}^{(2)}\to X$ given by 
	\[\eta(r, m_1, m_2)= \left(r\nu(m_1)+\nu(m_2), \nu(rm_1 + m_2)\right), \quad \rho(r,m_1, m_2) = \mu(m_1)=\mu(m_2) \]
Since these maps are continuous and $N$ is finite, the preimage of the diagonal $DN = \{(n_1, n_2)\in N\times N\mid n_1=n_2 \}$ under $\eta$ is open. Now consider the subset $Y$ of $X$ consisting of those $y\in X$ such that $\rho^{-1}(y)\subseteq \eta^{-1}(DN)$---that is, those $y\in X$ for which $\nu$ is an $R$-module morphism when restricted to ${\cal M}_y$. We must show that $Y$ is open; but the complement of $Y$ is simply $\rho\left(R\times {\cal M}^{(2)}\smallsetminus\eta^{-1}(DN)\right)$ which is compact, hence closed as required. 
\end{proof}

\begin{prop}\label{WeakProjections}
Let $\cal M$ be a sheaf of $R$-modules, let $x\in X$ and let $W$ be a clopen neighbourhood of $x$. Then every continuous morphism $\sigma_x\colon{\cal M}_x\to N$ to a finite module $N$ can be extended to  a morphism of sheaves $\sigma\colon {\cal M}\to N$ such that $\sigma_y=0$ for all $y\notin W$.
\end{prop}
\begin{proof}
	For every $n\in N$ the set $\sigma_x^{-1}(n)$ is compact. Since $\cal M$ is a profinite space there exists a clopen subset $U_n$ of $X$ for each $n$ such that the $U_n$ are disjoint for different $n$ and such that $\sigma_x^{-1}(n)\subseteq U_n$. Define a continuous map $\nu$ from $\cal M$ to $N$ by setting 
	\[\nu(m)= \begin{cases} n & \text{if $m\in U_n$} \\ 0 & \text{if $m\notin \bigcup_{n\in N} U_n$}     \end{cases} \] 	
	This agrees with the $R$-module morphism $\sigma_x$ on ${\cal M}_x$. Applying Lemma \ref{ShavesLem1} there is a clopen neighbourhood $Z$ of $x$ in $X$, which without loss of generality is contained in $W$, such that $\nu$ restricts to an $R$-module morphism on ${\cal M}_z$ for all $z\in Z$. Finally define $\sigma\colon{\cal M}\to N$ to agree with $\nu$ on $\mu^{-1}(Z)$ and by the zero map elsewhere. This is the required sheaf morphism.	
\end{proof}
	
\begin{prop}\label{PropsOfExtDirSum}
	Let $({\cal M}, \mu, X)$ be a sheaf of $R$-modules, let $M$ be the profinite direct sum of $\cal M$ and let $\omega\colon{\cal M}\to M$ be the canonical morphism. Then 
	\begin{itemize}
	\item $M$ is generated by the subgroups $M_x=\omega({\cal M}_x)$
	\item If $x\neq y$ then $M_x\cap M_y=0$
	\item $\omega$ maps ${\cal M}_x$ isomorphically onto $M_x$ for all $x\in X$
	\end{itemize}
\end{prop}
\begin{proof}
	The first point holds by the explicit construction in Proposition \ref{SheavesConstr}. For the second, take $m\in M_x\smallsetminus \{0\}$ and let $\overline{m}$ be a preimage of $m$ in ${\cal M}_x$. There is then a finite module $N$ and a module morphism $\sigma_x\colon{\cal M}_x\to N$ such that the image of $\overline{m}$ is non-trivial. By Proposition \ref{WeakProjections} (taking $W$ to be any clopen neighbourhood of $x$ that does not include $y$) we may extend this to a sheaf morphism $\sigma\colon {\cal M}\to N$ vanishing on ${\cal M}_y$. By the universal property this induces a morphism $M\to N$ sending $M_y$ to 0 and $m$ to $\sigma_x(\overline{m})\neq 0$. This shows that $M_x\cap M_y=0$. A similar argument shows that for any $\overline{m}\in{\cal M}_x\smallsetminus\{ 0 \}$ there is a morphism $M\to N$ to a finite module sending $\omega(\overline{m})$ to a non zero element, which proves the final item.
\end{proof}

We now discuss the connections between profinite direct sums and inverse limits.
\begin{lem}
Let $({\cal M}_i, \mu_i, X_i)_{i\in I}$ be an inverse system of sheaves indexed over a directed poset $(I, \succcurlyeq)$ and let $({\cal M},\mu, X)$ be the inverse limit sheaf. Then every sheaf morphism $\beta\colon {\cal M} \to N$ to a finite module $N$ factors through one of the ${\cal M}_k$.
\end{lem}
\begin{proof}
By assumption $\beta$ is a continuous map from the profinite space $\cal M$ to $N$. Since $\cal M$ is the inverse limit of the ${\cal M}_i$ as a topological space $\beta$ factors through a continuous function $\beta_{i_0}\colon {\cal M}_{i_0}\to N$ for some $i_0\in I$. Of course at this stage $\beta_{i_0}$ need not be a sheaf morphism. Set $I_0=\{i\in I\mid i\succcurlyeq i_0 \}$. For every $i\in I_0$ define $\beta_i\colon {\cal M}_i\to N$ to be $\beta_i=\beta_{i_0}\phi_{i i_0}$ where $\phi_{ii_0}$ is the transition map ${\cal M}_i\to {\cal M}_{i_0}$. We claim that for some $k$ the map $\beta_k$ is a sheaf morphism, which will give the required factorisation of $\beta$. Note that $\beta=\varprojlim_{i\geq i_0}\beta_i$.

Consider the map
\[ \eta\colon R\times {\cal M}^{(2)}\to N\times N,\quad \eta(r, m_1, m_2)= \left(r\beta(m_1)+\beta(m_2)\right)\]
and similar maps $\eta_i\colon R\times {\cal M}_i^{(2)}\to N\times N$ for each $i\in I_0$  with $\beta_i$ in place of $\beta$. It is readily seen that ${\cal M}^{(2)}=\varprojlim_{i\geq i_0} {\cal M}_i^{(2)}$, that $\eta=\varprojlim_{i\geq i_0} \eta_i$ and that $\eta({\cal M}^{(2)})=\bigcap_{i\geq i_0} \eta_i({\cal M}_i^{(2)})$. Now $N\times N$ is a finite set and $I_0$ is a directed poset so at some point   this intersection must stabilise and there exists $k\in I_0$ such that $\eta({\cal M}^{(2)})=\eta_k({\cal M}_k^{(2)})$. Now $\beta$ is a sheaf morphism so $\eta({\cal M}^{(2)})$ is contained in the diagonal $DN=\{(n_1, n_2)\in N\times N\mid n_1=n_2\}$. Hence the image of $\eta_k$ is also contained in this diagonal. This is precisely the statement that $\beta_k$ is a sheaf morphism and we are done.   
\end{proof}
\begin{prop}\label{ExtDirSumsAndLimits}
For an inverse system of sheaves $({\cal M}_i, \mu_i, X_i)$ indexed over a directed poset $(I, \succcurlyeq)$ we have
\[\bigboxplus_X \varprojlim_{i\in I} {\cal M}_i = \varprojlim_{i\in I} \bigboxplus_{X_i}{\cal M}_i\]
\end{prop}	
\begin{proof}
Let $M_i=\bigboxplus_{X_i}{\cal M}_i$ with canonical morphism $\omega_i\colon {\cal M}_i\to M_i$. If $i,j\in I$ with $i\succcurlyeq j$ let $\phi_{ij}\colon {\cal M}_i\to {\cal M}_j$ be the transition map in the inverse system and let $\psi_{ij}\colon M_i\to M_j$ be the canonical morphism induced by $\omega_j\phi_{i j}\colon {\cal M}_i\to M_j$. Then $\{M_i, \psi_{ij}, I\}$ is an inverse system of compact $R$-modules. Let $M$ be its inverse limit and let $\omega=\varprojlim \omega_i\colon {\cal M}\to M$, $\psi_i\colon M\to M_i$ be the natural maps. We shall prove that $M=\bigboxplus_X\cal M$ with universal morphism $\omega$.

Let $N$ be a finite $R$-module and let $\beta\colon {\cal M}\to N$ be a morphism. By the previous lemma, $\beta$ factors through some ${\cal M}_k$, and thereby induces a map $\tilde\beta_k\colon M_k\to N$. Define $\tilde\beta=\tilde\beta_k\psi_k$. This is a map $M\to N$ such that $\tilde \beta \omega=\beta$. It is the unique such map since $\omega_i({\cal M}_i)$ generates $M_i$ for each $i$ by Lemma \ref{PropsOfExtDirSum}, so $\omega(\cal M)$ generates $M$. Thus $M$ has the appropriate universal property and we are done.
\end{proof}

Let $M$ be an $R$-module and suppose ${\cal F}=\{M_x\}$ is a family of submodules continuously indexed by a profinite space $X$ (in the same sense as in Definition \ref{DefCtsIndex}). We say $M$ is an {\em internal direct sum} of $\cal F$ if:
\begin{itemize}
\item $M_x\cap M_y = 0$ if $x\neq y$
\item If $\beta\colon\bigcup_{x\in X} M_x\to N$ is a continuous function to an $R$-module $N$ which restricts to a morphism on each $M_x$ then there is a unique extension of $\beta$ to a morphism $M\to N$.
\end{itemize}
Recall that since $\cal F$ is continuously indexed by $X$ if and only if the triple $({\cal M}, \mu, X)$ is a sheaf of $R$-modules where \[{\cal M}= \left\{(m,x)\in M\times X\mid g\in M_x \right\}\] and $\mu$ is the restriction of the projection map---that is, if  and only if $\cal M$ is a closed subset of $M\times X$. 

\begin{prop} \label{InternalExternal}
The definitions of internal and external direct sums agree. More precisely, if $M$ is the internal direct product of the continuously indexed family ${\cal F}=\{M_x \}_{x\in X}$ then $M$ is the external direct sum of the sheaf \[{\cal M}= \left\{(x,m)\in X\times M\mid m\in M_x \right\}\] Conversely if $M$ is the external direct sum of a sheaf $\cal M$ with canonical morphism $\omega\colon {\cal M}\to M$ then $M$ is the internal direct sum of ${\cal F}=\{\omega({\cal M}_x)\}_{x\in X}$.
\end{prop}
\begin{proof}
First suppose that $M$ is the internal direct sum of $\cal F$. Let $N$ be an $R$-module and let $\beta\colon {\cal M}\to N$ be a morphism. Noting that by definition $\bigcup_{x\in X} M_x$ is the quotient space of $\cal M$ under collapsing $X\times\{1\}$ to a point and $X\times\{1\}$ is mapped to $1\in N$ by $\beta$, the map $\beta$ factors through a continuous map $\bigcup M_x\to N$ which agrees with $\beta$ on each $M_x$. Therefore $\beta$ extends uniquely to a map $\tilde\beta M\to N$ with $\tilde\beta \omega =\beta$. So $M$ is the external direct sum of the sheaf $\cal M$.

Conversely suppose that $M$ is the external direct sum of $\cal M$ (with canonical morphism $\omega$). That the family $\cal F$ is continuously indexed by $X$ follows easily from the definition of a sheaf. Let $\beta\colon\bigcup_{x\in X} M_x\to N$ be a continuous function to an $R$-module $N$ which restricts to a morphism on each $M_x$. Then $\beta\omega$ is a sheaf morphism ${\cal M}\to N$ and so there is a unique ring morphism $M\to N$ extending $\beta\omega$---and hence a unique morphism extending $\beta$ itself, since any such morphism would extend $\beta\omega$. It only remains to check that $M_x\cap M_y=0$ if $x\neq y$. This follows from Proposition \ref{PropsOfExtDirSum}.
\end{proof}
We record the following simple lemma for later use.
\begin{lem}\label{LemSheafOfTrivs}
Let $({\cal M}, \mu, X)$ be a sheaf of $R$-modules, and suppose that there is at most one $X$ for which ${\cal M}_x$ is not the zero module. Then $\bigboxplus_X{\cal M}= {\cal M}_x$.
\end{lem}
\begin{proof}
One may readily see that ${\cal M}_x$ is the internal direct sum of the ${\cal M}_y$ for $y\in Y$. So the result follows from Lemma \ref{InternalExternal}.
\end{proof}

\begin{prop}\label{DirSumsAndLimits}
Let $M$ be an $R$-module and ${\cal F}=\{M_x\}$ be a family of submodules indexed by a profinite space $X$. Suppose there exist inverse systems $(A_i, \phi_{i,j}, I)$ and $(X_i, f_{i,j}, I)$ where the $A_i$ are compact $R$-modules and the $X_i$ are profinite spaces such that 
\begin{enumerate}[(1)]
\item $X=\varprojlim X_i$
\item For every $i\in I$ we have $A_i = \bigboxplus_{x_i\in X_i} A_{i, x_i}$ for some collection of submodules $A_{i,x_i}$ continuously indexed by $X_i$.
\item $\phi_{i,j}(A_{i,t})\subseteq A_{j,f_{i,j}(x_i)}$ for every $i\geq j$ and every $x_i\in X_i$. 
\item For every $x=(x_i)\in X=\varprojlim X_i$ we have $M_x=\varprojlim A_{i, x_i}$
\item $M=\varprojlim A_i$ 
\end{enumerate}
Then ${\cal F}$ is continuously indexed by $X$ and $M$ is the profinite direct sum of ${\cal F}$.
\end{prop}
\begin{proof}
This is a straightforward translation of Proposition \ref{ExtDirSumsAndLimits} into the language of internal direct sums using Proposition \ref{InternalExternal}. The various conditions of the proposition guarantee that the natural sheaf corresponding to the family $\cal F$ is an inverse limit of sheaves corresponding to the families $\{A_{i,x_i}\}$.
\end{proof}

The converse statement is that there always exists an expression of a profinite direct sum as an inverse limit of finite direct sums. Furthermore the summands may always be taken to be finite.
\begin{prop}\label{DirSumsGiveLimits}
Let $M$ be an $R$-module and ${\cal F}=\{M_x\}$ be a family of submodules indexed by a profinite space $X$. Suppose ${\cal F}$ is continuously indexed by $X$ and $M$ is the profinite direct sum of ${\cal F}$. Then there exist inverse systems $(A_i, \phi_{i,j}, I)$ and $(X_i, f_{i,j}, I)$ where the $A_i$ are finite $R$-modules and the $X_i$ are finite discrete spaces such that 
\begin{enumerate}[(1)]
\item $X=\varprojlim X_i$
\item For every $i\in I$ we have $A_i = \bigoplus_{x_i\in X_i} A_{i, x_i}$ for some collection of finite $R$-modules $A_{i,x_i}$ indexed by $X_i$.
\item $\phi_{i,j}(A_{i,x_i})\subseteq A_{j,f_{i,j}(x_i)}$ for every $i\geq j$ and every $x_i\in X_i$. 
\item For every $x=(x_i)\in X=\varprojlim X_i$ we have $M_x=\varprojlim A_{i, x_i}$
\item $M=\varprojlim A_i$ 
\end{enumerate} 
\end{prop}
\begin{proof}
Let $\cal R$ be the set of all equivalence relations on $X$ whose equivalence classes are clopen. Let 
\[I=\left\{ i=(R,U)\mid R\in{\cal R}, \,U \text{ an open submodule of }M\right\} \]
and make $I$ into a directed poset by declaring $(R,U)\succcurlyeq (R',U')$ if and only if $U\subseteq U'$ and $R\subseteq R'$ (that is, $xRy$ implies $xR' y$). For $i=(R,U)\in I$ define $X_i = X/R$, let $x_i = [x]_R$ be the equivalence class of $x\in X$ under $R$, let $A_{i,x_i} = M(x_i)U/U$ and $A_i = \bigoplus_{x_i\in X_i} A_{i,{x_i}}$. Here $M(x_i)$ is the submodule of $M$ generated by all $M_y$ for $y\in x_i =[x]_R$.

For $i=(R,U)\succcurlyeq (R',U')=j$ define $f_{ij}\colon X_i\to X_j$ to be the natural map of quotient spaces. Note that the inclusion $[x]_R\subseteq [x]_{R'}$ induces an inclusion $M(x_i)\subseteq M(x_j)$ and therefore induces maps $A_{i,x_i}\to A_{j,f_{ij}(x_i)}$ for all $x_i\in X_i$. Thus we also have a natural continuous morphism $\psi_{ij}\colon A_i\to A_j$ which is an epimorphism since both sides are generated by the images of the $M_y$ as $y$ ranges over all $y\in X$. Certainly $\{A_i, \psi_{ij}, I\}$ and $\{X_i, f_{ij}, I\}$ are inverse systems and $X= \varprojlim X_i$. So we have conditions (1)--(3).

Let $x\in X$ and let $x_i$ be the image of $x$ in $X_i$. We claim that $M_x=\varprojlim A_{i,x_i} = \varprojlim M(x_i)U/U$. First note that 
\[I'=\left\{ i'=(R', U')\mid M([x]_{R'})U'/U' = M_xU'/U' \right\} \]
is a directed sub-poset of $I$. For let $i'=(R',U')$ and $i''=(R'',U'')$ be in $I'$. Set $U=U'\cap U''$. Now since the family $\cal F$ is continuously indexed by $X$ the set of $y$ such that $M_y\subseteq M_x U$ is open in $X$. Then we may choose $R\in \cal R$ such that $R\subseteq R'\cap R''$ and $M_y\subseteq M_xU$ for all $y\in [x]_R$. Then certainly $i=(R,U)\succcurlyeq i', i''$. Furthermore we certainly have $M([x]_R)U=M_x U$ so $i\in I'$. The sub-poset $I'$ is also cofinal in $I$. For if $i=(R,U)\in I$ then again $\{y\in X\mid M_y\subseteq M_xU\}$ is open so we may find $R'\in \cal R$ with $R'\subseteq R$ and $[x]_{R'}\subseteq\{y\in X\mid M_y\subseteq M_xU\}$. Then $(R', U)\succcurlyeq (R,U)$ and $(R', U)\in I'$. It follows that
\[ \varprojlim_{i\in I} A_{i,{x_i}} = \varprojlim_{i\in I'} A_{i,{x_i}} = \varprojlim_{(R,U)\in I'} M([x]_R)U/U = \varprojlim_{U} M_xU/U = M_x\]
as claimed.

Finally we have a natural epimorphism $\rho\colon M\to \varprojlim A_i$. By point (4) each $M_x$ embeds in the inverse limit under this map, and by Proposition \ref{DirSumsAndLimits} the inverse limit of the $A_i$ is the profinite direct sum of the image of the family $\cal F$. Since $M$ is also this direct sum, the map $\rho$ is an isomorphism and  $M=\varprojlim A_i$ as required. 
\end{proof}
\begin{prop}\label{DirectSumsAndTor}
Let $({\cal M}, \mu, X)$ be a sheaf of compact $R$-modules where $R=\Zpiof{G}$ and let $M=\bigboxplus_X \cal M$. As usual identify each fibre ${\cal M}_x$ with its image $M_x$ in the profinite direct sum $M$. Let $F$ be a functor from  \Gmod[\pi]{C} to itself which is additive and commutes with inverse limits.  Then:
\begin{itemize}
\item for each $x\in X$, the natural map $M_x\to M$ gives a canonical embedding of $F(M_x)$ in $F(M)$
\item $F(M) = \bigboxplus_{x\in X} F(M_x)$
\end{itemize}
In particular this holds for $\Tor$ functors. 
\end{prop}
\begin{proof}
	Write $M$ as an inverse limit of finite direct sums $A_{i,x_i}$ as in Proposition \ref{DirSumsGiveLimits}. Since both finite direct sums and inverse limits commute with $F$ we have
\[F(M)=F\left(\varprojlim_i \bigoplus_{t\in X_i} A_{i,x_i}\right)=\varprojlim_i \bigoplus_{x_i\in X_i} F(A_{i,x_i}) = \bigboxplus_{x\in X} F(M_x) \]
as required.  The fact that the final result is a well-defined (internal) direct sum follows since the final inverse limit satisfies the conditions of Proposition \ref{DirSumsAndLimits}. In particular the canonical map $F(M_x)\to F(M)$ is an embedding by Lemma \ref{PropsOfExtDirSum}.
\end{proof} 

Next we remark upon relations between profinite direct sums and free modules. For the theory of free profinite modules over profinite spaces see Section 5.2 of \cite{RZ00}.
\begin{prop}\label{DirSumsAndFreeMods}
Let $\mu\colon Y\to X$ be a surjection of profinite spaces and let $R$ be a profinite ring. Then $R[\![Y]\!]$ is the internal profinite direct sum
\[R[\![Y]\!]=\bigboxplus_{x\in X}R[\![\mu^{-1}(x)]\!] \]
\end{prop}
\begin{proof}
Write $\mu$ as an inverse limit of maps of finite spaces $\mu_i\colon Y_i\to X_i$. Then we have
\[ R[\![Y]\!] = \varprojlim R[\![Y_i]\!] = \varprojlim \bigoplus_{x\in X_i} R[\![\mu^{-1}_i(x)]\!] = \bigboxplus_{x\in X}R[\![\mu^{-1}(x)]\!] \]
noting that when $Y_i$ is a disjoint union of finitely many clopen sets $\mu^{-1}_i(x)$ the second equality follows immediately from the universal properties of finite direct sums and free modules. The fact that the final result is a well-defined (internal) direct sum follows since the inverse limit satisfies the conditions of Proposition \ref{DirSumsAndLimits}.
\end{proof}
There is one final proposition we will need.
\begin{prop}\label{InfGenSheaves}
Let $R=\Zpiof{G}$ where $\pi$ is a set of primes and $G$ is a profinite group. Let $({\cal M}, \mu, X)$ be a sheaf of compact $R$-modules and suppose that there exists $p\in\pi$ such that the fibres ${\cal M}_x$ admits a surjection to the trivial $R$-module $\F_p$ for infinitely many $x\in X$. Then $M=\bigboxplus_X \cal M$ is not finitely generated over $R$. 
\end{prop}
\begin{proof}
Let $n$ be a natural number. It is sufficient to prove that for any $n$ there is a surjection $M\to \F_p^{n}$, since such a sum cannot be generated by fewer than $n$ elements and hence neither can $M$. 

Take points $x_1,\ldots, x_n$ in $X$ with ${\cal M}_{x_i}$ admitting a surjection $\phi_i\colon  {\cal M}_{x_i}\to \F_p$ and take disjoint clopen neighbourhoods $W_i$ of each $x_i$.  By Proposition \ref{WeakProjections} there is then a sheaf morphism $\beta_i\colon {\cal M}\to \F_p$ extending $\phi_i$ and vanishing outside $W_i$. This induces an epimorphism $\tilde  \beta_i\colon M\to \F_p$ which is a surjection when restricted to $M_{x_i}$ and vanishes on all $M_{x_j}$ for $j\neq i$. The product map to $\F_p^{n}$ is therefore also a surjection and we are done.
\end{proof}
\begin{rmk}
The requirement in the proposition is stronger than `infinitely many ${\cal M}_X$ are non-zero', and is necessary. For instance suppose there exist at least countably many distinct finite simple $R$-modules $S_n (n\in\N)$. This is the case unless $G$ is virtually pro-$p$ by Corollary 5.3.5 of \cite{benson98}. Then one may form the inverse limit 
\[M=\varprojlim_n \big( S_1\oplus\cdots\oplus S_n \big)\]
where the maps in the inverse limit are the obvious projections. One may see by Proposition \ref{DirSumsAndLimits} that $M$ is the direct sum of a sheaf whose base space is the one-point compactification of $\mathbb{N}$ and with infinitely many non-zero fibres. However each $(S_1\oplus\cdots\oplus S_n)$ in the limit may be generated by a single element, hence so can $M$.

For arbitrary $G$ and a collection $\famS=\{S_x\}_{x\in X}$ of subgroups continuously indexed by a space $X$ we note that every fibre of 
\[\Zpiof{G/\famS}=\bigboxplus_{x\in X} \Zpiof{G/S_x} \] 
surjects to the simple module $\F_p$ for any $p\in \pi$ and so $\Zpiof{G/\famS}$ is finitely generated if and only if $X$ is finite.
\end{rmk}

\renewcommand\thesection{Appendix \Alph{section}}
\section{Supplement on Graphs of Profinite Groups}\label{AppGofGs}
\renewcommand\thesection{\Alph{section}}

Here we will prove some propositions which are necessary for the results in Section \ref{SecGraphsOfPDn}. We did not include it in that section as it would have rather disrupted the flow of the paper.

\begin{prop}\label{NoFpSplittings}
Let $p$ be a prime, $G$ be a profinite group and $S$ a closed subgroup of $G$ such that $p^\infty$ divides  $[G:S]$. Then $\Fpof{G/S}$ has no non-zero $G$-invariant elements. In particular the natural augmentation map $\Fpof{G/S}\to \F_p$ does not split.
\end{prop} 
\begin{proof}
Write $S$ as the intersection of a sequence of open subgroups $U_i$ of $G$ such that $p\mid [U_i: U_{i+1}]$ for all $i$. Then 
\[\Fpof{G/S}=\varprojlim \Fpof{G/U_i}=\varprojlim \F_p [G/U_i] \]
Now assume that $\Fpof{G/S}$ has a non zero $G$-invariant element $m$, and let $m_i$ be its image in  $\F_p [G/U_i]$ for each $i$. Then the $m_i$ map to each other under the maps in the inverse system and for some $i$, the element $m_i$ is non-zero. Now the $G$-invariant elements of $\F_p [G/U_i]$ are precisely the elements
\[\lambda N_i := \sum_{gU_i\in G/U_i} \lambda gU_i \] 
for $\lambda\in\F_p$. However the image of $\lambda N_{i+1}$ in $ \F_p [G/U_i]$ is exactly
$\lambda [U_i:U_{i+1}] N_i$ for all $i$ and all $\lambda$, which vanishes since $p\mid [U_i: U_{i+1}]$. We have reached a contradiction.
\end{proof}
\begin{rmk}
The assumption that $p^\infty \mid [G:S]$ is of course necessary. One may readily see that the element
\[\left( q^{-k} N_k \right)_{k\geq 1} \in \varprojlim \F_p(\Z/q^k) = \F_p(\Z[q]) \]
is $\Z[q]$-invariant where $p$ and $q$ are any two distinct primes.
\end{rmk}
Before the next proposition we must recall a fact about the structure of fundamental groups of graphs of pro-$\pi$ groups. Specifically take a finite graph of profinite groups $(X,G_\bullet)$. Let $G^{\rm abs}$ be the fundamental group of $(X,G_\bullet)$ {\em as a graph of abstract groups}. Then the pro-$\pi$ fundamental group of $(X,G_\bullet)$ is exactly the completion of $G^{\rm abs}$ with respect to the topology
\[{\cal U}=\left\{ N\nsgp[f] G^{\rm abs}\mid G_i\cap N \nsgp[o] G_i \text{ and } G^{\rm abs}/N \text{ is a $\pi$-group}\right\} \]
where $\nsgp[f]$ and $\nsgp[o]$ mean normal subgroups of finite index and open normal subgroups respectively. See Proposition 6.5.1 of \cite{Ribes17} and the discussion leading to it.

Let $\cal C$ be an variety of finite groups closed under taking isomorphisms, subgroups, quotients and extensions. Let $\pi(\cal C)$ be the set of primes which divide the order of some finite groups in $\cal C$.
\begin{prop}\label{VxGrpsAreInfIndex}
Let $G$ be either a proper pro-$\cal C$ HNN extension $G=G_1\amalg_L$ or a proper pro-$\cal Ci$ amalgamated free product $G=G_1\amalg_L G_2$ where $L\neq G_i$ for each $i$. Then $p^\infty\mid [G:G_1]$ for all $p\in \pi(\cal C)$. 
\end{prop}
\begin{proof}
In the HNN extension case there is a map $G\to \Z[\pi]$ whose kernel is the normal subgroup generated by $G_1$ and the result follows immediately. For the amalgamated free product case we must work a little harder. By the assumption of the theorem there is a map from $\phi\colon G\to P$ to a finite $\pi$-group such that $\phi(L)$ is a proper subgroup of $\phi(G_i)$ for each $G$. Then $G'=\ker \phi$ is the fundamental group of a graph of groups whose base graph is not a tree. More precisely the Euler characteristic of the graph is 
\[\chi= [P:\phi(G_1)]+[P:\phi(G_2)]-[P:\phi(L)] \]  
which is non-positive. One may see this either by considering the action of $G'$ on the standard tree of the graph of pro-$\cal C$ groups decomposition of $G$, or by translating to the language of abstract graphs of groups and back again using the discussion prior to the theorem. Since this graph now has a loop, the quotient of $G'$ by all its vertex groups is a non-trivial free pro-$\cal C$ group of rank $1-\chi$. This follows for example by the argument in Theorem 6.2.4 of \cite{Ribes17}, or using the above translation to the theory of graphs of abstract groups. Hence all vertex groups $G'_v$ of $G'$ have $p^\infty\mid [G':G'_v]$. Since for each $i$, $G$ and $G_i$ are finite index overgroups of $G'$ and $G'_v$ for some $v$ they also have this property as required.
\end{proof}

\bibliographystyle{alpha}
\bibliography{RelCoh.bib}
\end{document}